\documentclass[12pt,reqno,a4paper]{amsart}

\usepackage{fullpage,enumitem,todonotes,framed,nameref,enumitem,csvsimple,amsmath,amssymb,color}
\usepackage[utf8]{inputenc}
\usepackage[T1]{fontenc}
\usepackage[list=true, labelfont=bf, labelformat=brace, position=top]{subcaption}
\definecolor{tuwBlue}{RGB}{0,116,178}
\usepackage[section]{placeins}
\usepackage{hyperref}
\def\cred{\color{black}}

\def\revision#1{{\cred#1}}

\usepackage{moreverb}

\title{Functional a~posteriori error estimates\\ for boundary element methods}
\author{Stefan Kurz}
\author{Dirk Pauly}
\author{Dirk Praetorius}
\author{Sergey Repin}
\author{Daniel Sebastian}

\address{TU Wien, Institute for Analysis and Scientific Computing, Austria}
\email[Dirk Praetorius]{dirk.praetorius@asc.tuwien.ac.at \quad \rm(corresponding author)}
\email[Daniel Sebastian]{daniel.sebastian@asc.tuwien.ac.at}

\address{TU Darmstadt, Centre for Computational Engineering, Germany}
\email[Stefan Kurz]{kurz@gsc.tu-darmstadt.de}

\address{Universit\"at Duisburg-Essen, Fakult\"at f\"ur Mathematik, Germany}
\email[Dirk Pauly]{dirk.pauly@uni-due.de}

\address{University of Jyv\"askyl\"a, MIT, Finland}
\email[Sergey Repin]{sergey.repin@mit.jyu.fi}

\address{RAS, Steklov Institute of Mathematics at St. Petersburg, Russia}
\email[Sergey Repin]{repin@pdmi.ras.ru}

\thanks{{\bf Acknowledgement.} D.~Sebastian and D.~Praetorius thankfully acknowledge support by the 
Austrian Science Fund (FWF) through 
the SFB \emph{Taming complexity in partial differential systems}, 
and the stand-alone project \emph{Optimal adaptivity for BEM and FEM-BEM coupling} (grant P27005).
The work of S.~Kurz was supported in part by the Excellence Initiative of the German Federal and 
State Governments, and in part by the \emph{Graduate School of Computational Engineering} at TU Darmstadt.}

\makeatother

\def\set#1#2{\big\{#1 \,:\, #2 \big\}}
\def\d#1{\,{\rm d}#1}

\def\str#1#2{#1|_{#2}}
\def\ntr#1#2{\normal \cdot #1|_{#2}}
\def\ttr#1#2{\normal \times #1|_{#2}}

\renewcommand{\H}{\mathsf{H}}
\DeclareMathOperator{\sfC}{\mathsf{C}}
\DeclareMathOperator{\sfL}{\mathsf{L}}
\newcommand{\mmin}{\mathbf{\underline{\mathfrak{M}}}}
\newcommand{\mmax}{\mathbf{\overline{\mathfrak{M}}}}
\def\N{\mathbb{N}}
\def\MM{\mathcal{M}}
\def\TT{\mathcal{T}}
\def\OO{\mathcal{O}}
\def\R{\mathbb{R}}

\def\SSS{\mathcal{S}}
\def\FF{\mathcal{E}}
\def\FF{\mathcal{F}}
\def\PP{\mathcal P}
\def\ttau{\boldsymbol{\tau}}
\def\ssigma{\boldsymbol{\sigma}}
\def\normal{\boldsymbol{n}}
\def\aa{\boldsymbol{a}}
\def\eps{\varepsilon}

\DeclareMathOperator{\diam}{diam}
\DeclareMathOperator{\dist}{dist}
\DeclareMathOperator{\supp}{supp}

\DeclareMathOperator{\di}{div}
\renewcommand{\div}{\di}
\DeclareMathOperator{\curl}{curl}
\def\grad{\nabla}

\DeclareMathOperator{\Linfty}{\sfL^\infty}
\def\Lt{\sfL^2}
\def\Ho{\H^{1}}
\def\Hoom{\Ho(\Omega)}
\def\Hozom{\Ho_{0}(\Omega)}
\def\Hdiv{\H(\div)}
\def\Hdivom{\H(\div,\Omega)}

\def\Hcurlom{\H(\curl,\Omega)}
\def\RTz{\mathsf{RT}^{0}}
\def\RTq{\mathsf{RT}^{q}}
\def\RTzg{\RTz_{\Gamma^{c}}}
\def\RTqg{\RTq_{\Gamma^{c}}}

\def\norm#1#2{|\hspace*{-1.5pt}|#1|\hspace*{-1.5pt}|_{#2}}
\def\bnorm#1#2{\big|\hspace*{-2pt}\big|#1\big|\hspace*{-2pt}\big|_{#2}}
\def\scp#1#2#3{\langle#1 \, , \, #2\rangle_{#3}}
\def\bscp#1#2#3{\big\langle#1 \, , \, #2\big\rangle_{#3}}
\def\dualpga#1#2{\langle#1\,,\,#2\rangle_{\Gamma}}
\def\bdualpga#1#2{\big\langle#1,#2\big\rangle_{\Gamma}}

\def\conststyle{\sf}

\def\Cstab{C_{\conststyle stab}}

\def\Cosc{C_{\conststyle osc}}

\usepackage{fancyhdr}
\advance\footskip0.4cm
\advance\textheight-0.4cm
\calclayout
\newcommand*\patchAmsMathEnvironmentForLineno[1]{%
  \expandafter\let\csname old#1\expandafter\endcsname\csname #1\endcsname
  \expandafter\let\csname oldend#1\expandafter\endcsname\csname end#1\endcsname
  \renewenvironment{#1}%
     {\linenomath\csname old#1\endcsname}%
     {\csname oldend#1\endcsname\endlinenomath}}% 
\newcommand*\patchBothAmsMathEnvironmentsForLineno[1]{%
  \patchAmsMathEnvironmentForLineno{#1}%
  \patchAmsMathEnvironmentForLineno{#1*}}%
\AtBeginDocument{%
\patchBothAmsMathEnvironmentsForLineno{equation}%
\patchBothAmsMathEnvironmentsForLineno{align}%
\patchBothAmsMathEnvironmentsForLineno{flalign}%
\patchBothAmsMathEnvironmentsForLineno{alignat}%
\patchBothAmsMathEnvironmentsForLineno{gather}%
\patchBothAmsMathEnvironmentsForLineno{multline}%
}
\usepackage[mathlines]{lineno}

\makeatletter
\def\@seccntformat#1{\vspace*{-2mm}\newline\hspace*{4mm}%
  \protect\textup{\protect\@secnumfont
    \ifnum\pdfstrcmp{subsection}{#1}=0 \bfseries\fi% subsection # in \bfseries
    \csname the#1\endcsname
    \protect\@secnumpunct
  }%
}
\def\paragraph{\@startsection{paragraph}{4}%
  \z@\z@{-\fontdimen2\font}%
  {\normalfont\bfseries}}
\makeatother

\makeatletter
\def\section{\@startsection{section}{1}%
\z@{.7\linespacing\@plus\linespacing}{.5\linespacing}%
{\normalsize\scshape\bfseries\centering}}
\makeatother

\makeatletter
\renewcommand{\@secnumfont}{\bfseries}
\makeatother

\newcounter{statement}
\newenvironment{statement}[2][!]{%
\vskip2mm
\noindent%
\refstepcounter{statement}%
\bf#2~\thestatement%
\ifthenelse{\equal{#1}{!}}{.\ }{~(#1).\ }%
\it%
}{%
\vskip1mm
}

\newenvironment{theorem}[1][!]{\begin{statement}[#1]{Theorem}}{\end{statement}}
\newenvironment{lemma}[1][!]{\begin{statement}[#1]{Lemma}}{\end{statement}}
\newenvironment{definition}[1][!]{\begin{statement}[#1]{Definition}}{\end{statement}}

\newenvironment{corollary}[1][!]{\begin{statement}[#1]{Corollary}}{\end{statement}}
\newenvironment{remark}[1][!]{\begin{statement}[#1]{Remark}}{\end{statement}}
\newenvironment{algorithm}[1][!]{\begin{statement}[#1]{Algorithm}}{\end{statement}}
\begin{document}

%%%%%%%%%%%%%%%%%%%%%%%%%%%%%%%%%%%%%%%%%%%%%%%%%%%%%%%%%%%%%%%%%%%%%%%%%%%%%%%%%%%
%%%%%%%%%%%%%%%%%%%%%%%%%%%%%%%%%%%%%%%%%%%%%%%%%%%%%%%%%%%%%%%%%%%%%%%%%%%%%%%%%%%
\begin{abstract}
Functional error estimates are \revision{well-established} tools for \emph{a~posteriori} error estimation and related adaptive mesh-refinement for the finite element method (FEM). The present work proposes a first functional error estimate for the boundary element method (BEM). One key feature is that the derived error estimates are independent of the BEM discretization and provide guaranteed lower and upper bounds for the unknown error.
%In particular, our analysis covers Galerkin BEM as well as collocation, which makes our approach of particular interest for people working in engineering. 
\revision{In particular, our analysis covers Galerkin BEM and the collocation method, what makes the approach of particular interest for scientific computations and engineering applications.}
Numerical experiments for the Laplace problem confirm the theoretical results.
\end{abstract}
%%%%%%%%%%%%%%%%%%%%%%%%%%%%%%%%%%%%%%%%%%%%%%%%%%%%%%%%%%%%%%%%%%%%%%%%%%%%%%%%%%%
%%%%%%%%%%%%%%%%%%%%%%%%%%%%%%%%%%%%%%%%%%%%%%%%%%%%%%%%%%%%%%%%%%%%%%%%%%%%%%%%%%%

\maketitle
\thispagestyle{fancy}

\tableofcontents

%%%%%%%%%%%%%%%%%%%%%%%%%%%%%%%%%%%%%%%%%%%%%%%%%%%%%%%%%%%%%%%%%%%%%%%%%%%%%%%%%%%
%%%%%%%%%%%%%%%%%%%%%%%%%%%%%%%%%%%%%%%%%%%%%%%%%%%%%%%%%%%%%%%%%%%%%%%%%%%%%%%%%%%

%%%%%%%%%%%%%%%%%%%%%%%%%%%%%%%%%%%%%%%%%%%%%%%%%%%%%%%%%%%%%%%%%%%%%%%%%%%%%%%%%%%
%%%%%%%%%%%%%%%%%%%%%%%%%%%%%%%%%%%%%%%%%%%%%%%%%%%%%%%%%%%%%%%%%%%%%%%%%%%%%%%%%%%
\section{Introduction}
%%%%%%%%%%%%%%%%%%%%%%%%%%%%%%%%%%%%%%%%%%%%%%%%%%%%%%%%%%%%%%%%%%%%%%%%%%%%%%%%%%%
%%%%%%%%%%%%%%%%%%%%%%%%%%%%%%%%%%%%%%%%%%%%%%%%%%%%%%%%%%%%%%%%%%%%%%%%%%%%%%%%%%%

\noindent
Let $\Omega \subset \R^d$, $d \geq 2$, be a bounded Lipschitz domain with polygonal boundary $\Gamma := \partial\Omega$.
\revision{To present the main ideas and our first numerical results, we consider}
the Poisson problem with inhomogeneous Dirichlet boundary data $g$, i.e.,
\begin{align}
\label{eq:strongform}
\Delta u=0\quad\text{in }\Omega,\qquad 
u=g\quad\text{on }\Gamma.
\end{align}
Throughout the paper, we assume that $d \in \{2,3\}$. However, all results can easily be extended
to higher dimensions.
For the numerical solution of~\eqref{eq:strongform}, 
we employ the boundary element method (BEM); see, e.g.,~\cite{steinbach,sauter-schwab,gwinner-stephan}. 
\revision{Again for} the ease of presentation, let us consider an indirect ansatz based on the single-layer potential 
\begin{align}\label{eq:slp}
(\widetilde V\phi)(x) := \int_{\Gamma} G(x-y) \, \phi(y) \d{y} = u(x)
\quad \text{for all } x\in\Omega
\end{align}
with unknown integral density $\phi$, where, \revision{for $z \in \R^d \backslash \{0\}$,}
$G(z) = -\frac{1}{2\pi} \, \log |z|$ for $d = 2$ resp. 
$G(z) = \frac{1}{4\pi} |z|^{-1}$ for $d = 3$ 
denotes the fundamental solution of the Laplacian.
Taking the trace on $\Gamma$, the potential ansatz leads to the weakly-singular integral equation
\begin{align}\label{eq:weaksing}
(V \phi)(x) = g(x) \quad \text{for almost all } x\in\Gamma,
\end{align}
where the integral representation of $g=V\phi$ 
coincides with that of $u=\widetilde V\phi$ (at least for bounded densities) 
but is now evaluated on $\Gamma$ (instead of inside $\Omega$). 
\revision{For ellipticity of the operator $V$, we suppose
that $\diam(\Omega) < 1$ in case of $d = 2$, which can always be achieved by scaling.}
Given a triangulation $\FF^\Gamma_h$ of the boundary $\Gamma$, 
the latter equation is solved by the lowest-order BEM and provides some
piecewise constant approximation $\phi_h$, i.e.,
\begin{align}
\label{eq:def:P0}
\phi\approx\phi_h\in\PP^0(\FF^\Gamma_h),
\end{align}
where the precise discretization (e.g., Galerkin BEM, collocation, etc.) will not be exploited by our analysis.
However, as a BEM inherent characteristic, we obtain an approximation of the potential $u \approx u_h:= \widetilde V \phi_h$, which satisfies the Laplace problem 
\revision{\begin{align}
\label{Laplaceprob}
\Delta u_h= 0 \quad\text{in } \Omega.
\end{align}}%
%in $\Omega$. 
\revision{Note that here --- contrary to the usual notations ---
$u_{h}$ is \emph{not} a discrete function
but computed by an integral operator applied to a discrete function, i.e., $u_h$ is data sparse.
We emphasize that~\eqref{Laplaceprob}} is the key argument for the error identity
\begin{align}
\label{eq:erroridentity}
\max_{\substack{\ttau\in\Lt(\Omega)\\\div\ttau=0}} 
\Big(2\,\dualpga{g-\str{u_h}{\Gamma}}{\ntr{\ttau}{\Gamma}}
-\norm{\ttau}{\Lt(\Omega)}^2\Big)
&\!=\!\bnorm{\grad(u-u_h)}{\Lt(\Omega)}^2
\!=\!\!\!\min_{\substack{w\in\Hoom\\\str{w}{\Gamma}=g-\str{u_h}{\Gamma}}}
\norm{\grad w}{\Lt(\Omega)}^2,
\end{align}
where $\dualpga{\cdot}{\cdot}$ denotes the extended $\Lt(\Gamma)$ scalar product (see Theorem~\ref{prop:continuousmajmin} below).
The identities~\eqref{eq:erroridentity} generate {\it a posteriori error estimates of the functional type} that
are independent of the discretization and provide fully guaranteed lower and upper bounds for the unknown error 
without any constants at all.
In general, these functional type a posteriori estimates
involve only constants in  basic functional inequalities associated with the concrete problem
(e.g., Poincar\'e--Friedrichs type or trace inequalities) and are applicable for {\em any} approximation
from the admissible energy class (see~\cite{repin00,zbMATH06636825,MR3935892,zbMATH07143777} or the monograph~\cite{repinbook} 
and the references cited therein).
In particular, the equations~\eqref{eq:erroridentity} have also been used in~\cite{repinbook} 
for the analysis of errors arising in the Trefftz method.

From~\eqref{eq:erroridentity}, constant-free (i.e., with known constant $1$) lower
and upper bounds for the unknown \emph{potential error} $\norm{\grad(u-u_h)}{\Lt(\Omega)}$ 
can be obtained by choosing \emph{arbitrary} instances of $\ttau$ and $w$.
In the present work, we compute these bounds by solving problems 
in a suitable boundary layer $S \subset \Omega$ along $\Gamma$ by use of the finite element method (FEM).
Moreover, these bounds are then employed to drive an adaptive mesh-refinement 
for the triangulation $\FF^{\Gamma}_h$ of $\Gamma$ 
and, as novelty, also quantify the accuracy $\norm{\nabla(u-u_h)}{\Lt(\Omega)}$ of the BEM induced potential 
$u_h= \widetilde V \phi_h$ in each step of the adaptive algorithm. 
In particular, the latter quantification is essentially constant-free (up to data oscillations terms arising for the FEM majorant) and can thus also be used as reasonable stopping criterion for adaptive BEM computations. 
Especially for practical applications, this is an important step forward, 
since there exist neither \emph{a posteriori} error estimates with constant $1$ 
nor estimates for physically relevant errors. While available results focus 
on the density $\phi_h$ (see, e.g., \cite{cs1995,cc1997,mundstephanweiss98,cms2001,eh2006,MR2529605} for some prominent results or the surveys \cite{cf01,Feischl2015} and the references therein), estimating rather the energy error of $u_h$ circumvents, 
in particular, BEM-natural challenges like the localization of non-integer Sobolev norms. 
\revision{It is quite natural that these serious advantages of the proposed error estimation strategy are associated with certain technical complications that arise because we need to generate a volume mesh at least for some boundary layer $S \subset \Omega$ along $\Gamma$ on which we solve
auxiliary FEM problems. However, the ratio between the number of degrees of freedom (DoF) for obtaining the error estimates and the BEM DoF remains bounded, so that additional computational expenditures remain limited. Moreover, examples show that very good error bounds can be obtained when the ratio is between one and three.} \revision{Finally, we note that the generation of the volume mesh appears to be a standard problem for FEM mesh generation, where usually, like in computer aided design (CAD), only the surface $\Gamma$ is given.} 
%\revision{We stress that the advantages of the present error estimation strategy come with the possible disadvantage that we have to generate a volume mesh at least for some boundary layer $S \subset \Omega$ along $\Gamma$ on which we solve the auxiliary FEM problems. However, this appears to be a standard problem for FEM mesh generation, where usually, like in CAD, only the surface $\Gamma$ is given.}

\bigskip
\textbf{Outline.} 
The remainder of this work is organized as follows:
In Section~\ref{section:notation}, we collect the necessary notations
as well as the fundamental properties of (Galerkin) BEM. In Section~\ref{section:approach}, 
we formulate our approach for functional {\sl a~posteriori} error estimation. 
Theorem~\ref{prop:continuousmajmin} states the error identity~\eqref{eq:erroridentity}. 
Theorem~\ref{cor:continuousupperbound}
provides a computable upper bound on~\eqref{eq:erroridentity} by means of an $\Ho$-conforming FEM approach
as well as a computable lower bound on~\eqref{eq:erroridentity} 
by means of an $\Hdiv$-conforming mixed FEM approach. 
Section~\ref{section:algorithm} shows how these findings can be used to steer an adaptive mesh-refinement. 
Algorithm~\ref{algorithm} formulates such a strategy with reliable error control on 
$\norm{\grad(u-u_h)}{\Lt(\Omega)}$.
In Section~\ref{section:numerics}, we employ the proposed adaptive algorithm 
to underpin our theoretical findings by some numerical experiments \revision{with lowest-order Galerkin BEM in 2D}.
Section~\ref{section:extensions} concludes the work with natural 
extensions of our 
%analysis 
\revision{approach (even covered by our analytical results)} like higher-order BEM, alternative BEM discretizations like collocation, direct BEM formulations, and error control for exterior domain problems (where $\Omega$ is unbounded),
\revision{underlining the independence of our error estimators of the actual problem and approximation method.}
The final Section~\ref{section:conclusion} summarizes the contributions of the present work and addresses possible topics for future research.

%%%%%%%%%%%%%%%%%%%%%%%%%%%%%%%%%%%%%%%%%%%%%%%%%%%%%%%%%%%%%%%%%%%%%%%%%%%%%%%%%%%
%%%%%%%%%%%%%%%%%%%%%%%%%%%%%%%%%%%%%%%%%%%%%%%%%%%%%%%%%%%%%%%%%%%%%%%%%%%%%%%%%%%
\section{Preliminaries and notation}
\label{section:notation}
%%%%%%%%%%%%%%%%%%%%%%%%%%%%%%%%%%%%%%%%%%%%%%%%%%%%%%%%%%%%%%%%%%%%%%%%%%%%%%%%%%%
%%%%%%%%%%%%%%%%%%%%%%%%%%%%%%%%%%%%%%%%%%%%%%%%%%%%%%%%%%%%%%%%%%%%%%%%%%%%%%%%%%%

%%%%%%%%%%%%%%%%%%%%%%%%%%%%%%%%%%%%%%%%%%%%%%%%%%%%%%%%%%%%%%%%%%%%%%%%%%%%%%%%%%%
\subsection{Domains and function spaces}
\label{section:domainsfspaces}
%%%%%%%%%%%%%%%%%%%%%%%%%%%%%%%%%%%%%%%%%%%%%%%%%%%%%%%%%%%%%%%%%%%%%%%%%%%%%%%%%%%

Throughout this paper, 
let $\Omega \subset \R^d$, $d\in\{2,3\}$, 
be a bounded Lipschitz domain (i.e., locally below the graph of some Lipschitz function)
with boundary $\Gamma = \partial \Omega$ and \revision{exterior unit} normal vector field $\normal$. 
For all numerical results involving discretisations, we assume that $\Gamma$ is a polygon.
We denote by $\scp{\cdot}{\cdot}{\Lt(\Xi)}$ and $\norm{\cdot}{\Lt(\Xi)}$
the standard inner product and norm in $\Lt(\Xi)$, respectively, 
where, e.g., $\Xi\in\{\Omega,\Gamma\}$.
Based on $\Lt(\Omega)$, we define the Hilbert spaces
\begin{align*}
\Hoom 
&:= \set{\varphi\in \Lt(\Omega)}{\grad\varphi\in \Lt(\Omega)},\\
\Hdivom 
&:= \set{\ssigma\in\Lt(\Omega)}{\div\ssigma\in\Lt(\Omega)}.
\end{align*}
The corresponding inner products and (induced) norms are $\scp{\cdot}{\cdot}{\Hoom}$ and $\norm{\cdot}{\Hoom}$ resp.\ $\scp{\,\cdot\,}{\,\cdot\,}{\Hdivom}$ and $\norm{\cdot}{\Hdivom}$.
Moreover, introducing the scalar trace operator 
$\str{(\cdot)}{\Gamma}:\Hoom\to\Lt(\Gamma)$,
%together with its range $\H^{1/2}(\Gamma) := \set{\varphi|_\Gamma}{\varphi \in H^1(\Omega)}$
%equipped with the natural quotient norm
%\begin{align*}
% \norm{f}{\H^{1/2}(\Gamma)} := \inf\set{\norm{\varphi}{\Hoom}}{\varphi\in \Hoom 
% \text{ with } \str{\varphi}{\Gamma} = f},
%\end{align*}
\revision{our analysis also employs the closed subspace of $\Hoom$
\begin{align*}
\Hozom :=\set{\varphi\in\Hoom}{\str{\varphi}{\Gamma}=0}
\end{align*}
and the trace space $\H^{1/2}(\Gamma) := \set{\varphi|_\Gamma\!}{\!\varphi \in H^1(\Omega)}$ equipped with the natural quotient norm
\begin{align*}
% \H^{1/2}(\Gamma)& := \set{\varphi|_\Gamma\!}{\!\varphi \in H^1(\Omega)}, \\
 \norm{f}{\H^{1/2}(\Gamma)} &:= \inf\set{\norm{\varphi}{\Hoom}\!}{\!\varphi\in \Hoom 
 \text{ with } \str{\varphi}{\Gamma} = f}\quad
 \text{for all } f\in\H^{1/2}(\Gamma).
\end{align*}
A standard construction (see the subsequent Remark~\ref{remharmext}) yields a harmonic extension operator
$\widehat{(\cdot)} : \H^{1/2}(\Gamma) \to \Ho(\Omega)$
which satisfies $\norm{\grad\widehat{f}}{\Lt(\Omega)}\leq\norm{f}{\H^{1/2}(\Gamma)}$
for all $f\in\H^{1/2}(\Gamma)$.}

\revision{\begin{remark}%[Harmonic extension operator with constant $1$]
\label{remharmext}
%The natural norm in $\H^{1/2}(\Gamma)$ is given by
%\begin{align*}
% \norm{f}{\H^{1/2}(\Gamma)} := \inf\set{\norm{\varphi}{\Hoom}}{\varphi\in \Hoom 
% \text{ with } \str{\varphi}{\Gamma} = f}.
%\end{align*}
In fact, the minimal extension $\varphi\in \Hoom$ of $f \in \H^{1/2}(\Gamma)$ 
satisfies $\norm{f}{\H^{1/2}(\Gamma)} = \norm{\varphi}{\Hoom}$ 
and can be found as the unique weak solution of 
\begin{align}\label{eq:lifting:kill}
%\begin{split}
- \Delta \varphi + \varphi
= 0 \quad \text{in } \Omega,
\qquad
\str{\varphi}{\Gamma}
= f \quad \text{on } \Gamma.
%\end{split}
\end{align}
The ansatz $\widehat{f}=\varphi+\varphi_{0}$ with $\varphi_{0}\in\Hozom$
solving $\Delta\varphi_{0}=-\varphi$
yields a harmonic extension $\widehat{f}\in\Hoom$
of $f\in\H^{1/2}(\Gamma)$, i.e.,
\begin{align}
\label{eq:lifting}
%\begin{split}
\Delta \widehat{f}
= 0 \quad \text{in } \Omega, 
\qquad
\str{\widehat{f}}{\Gamma}
= f \quad \text{on } \Gamma.
%\end{split}
\end{align}
From $\norm{\grad\widehat{f}}{\Lt(\Omega)}^2
=\scp{\grad\widehat{f}}{\grad\varphi}{\Lt(\Omega)}$,
%+\underbrace{\scp{\grad\widehat{f}}{\grad\varphi_{0}}{\Lt(\Omega)}}_{=0}$,
it follows that
%with 
$\norm{\grad\widehat{f}}{\Lt(\Omega)}\leq\norm{\grad\varphi}{\Lt(\Omega)}\leq\norm{f}{\H^{1/2}(\Gamma)}$. 
%as
%$$\norm{\grad\widehat{f}}{\Lt(\Omega)}^2
%=\scp{\grad\widehat{f}}{\grad\varphi}{\Lt(\Omega)}
%+\underbrace{\scp{\grad\widehat{f}}{\grad\varphi_{0}}{\Lt(\Omega)}}_{=0}.$$ 
\qed
\end{remark}}%

Finally, we need the dual space $\H^{-1/2}(\Gamma) := \H^{1/2}(\Gamma)'$
equipped with the natural norm
\begin{align*}
 \norm{f}{\H^{-1/2}(\Gamma)} := \sup_{0\neq\psi\in\H^{1/2}(\Gamma)} 
 \frac{\dualpga{\psi}{f}}{\norm{\psi}{\H^{1/2}(\Gamma)}},
\end{align*}
where the $\H^{1/2}(\Gamma) \times \H^{-1/2}(\Gamma)$-duality product $\dualpga{\cdot}{\cdot}$ extends, as usual, the $\Lt(\Gamma)$ scalar product $\scp{\cdot}{\cdot}{\Lt(\Gamma)}$. \revision{We stress that $\Gamma = \partial\Omega$ and hence $\H^{1/2}(\Gamma)=\widetilde\H^{1/2}(\Gamma)$.}
%\todo{Clearly, $H^{1/2}(\Gamma) \cong H^1(\Omega) / H^1_0(\Omega)$ is isometrically isomorphic and hence a Hilbert space so that the dual space of $\H^{-1/2}(\Gamma)$ can indeed be identified with $H^{1/2}(\Gamma)$.}
%\dont{Note that $\H^{1/2}(\Gamma)=\H^{1/2}_{0}(\Gamma)$ as $\Gamma=\partial\Omega$.}
We recall the Gelfand triple $\H^{1/2}(\Gamma)\subset\Lt(\Gamma)\subset\H^{-1/2}(\Gamma)$
and refer to \cite{boffibrezzifortin} for the fact that $\H^{-1/2}(\Gamma)$ can also be characterised as the range of normal traces
$\ntr{(\boldsymbol{\cdot})}{\Gamma}:\Hdivom\to\H^{-1/2}(\Gamma)$
of $\Hdivom$-vector fields, i.e.,
\begin{align*}
\H^{-1/2}(\Gamma) = \set{\ntr{\ssigma}{\Gamma}}{\ssigma\in\Hdivom}.
\end{align*}

\begin{definition}[Boundary layer]
\label{defS}
A subset $S \subset \Omega$ is called a \emph{boundary layer}, if it is a Lipschitz domain with $\Gamma \subset \partial S$, which admits a conforming triangulation $\TT_h^S$ into simplices. We then define $\Gamma^{c}:=\partial S \setminus \Gamma$. 
In particular, we define the corresponding induced triangulation of $\Gamma$ by
\begin{align}\label{eq:boundary-mesh}
\FF^{\Gamma}_h
:=\TT^S_h|_{\Gamma}
:=\set{F}{F\subset\Gamma\text{ and }F\text{ is a face of some simplex $T\in\TT^S_h$}}.
\end{align}
\end{definition}

For $q\in\N_{0}$ and $\PP^{q}$ being the space of polynomials of degree $q$, we define
\begin{align*}
\PP^{q}(\TT^S_h)
&:=\set{\varphi_h\in\Linfty(S)}{\str{\varphi_h}{T}\in\PP^{q}
\text{ for all }T\in\TT^S_h},\\
\PP^{q}(\FF^{\Gamma}_h)
&:=\set{\psi_h\in\Linfty(\Gamma)}{\str{\psi_h}{F}\in\PP^{q}
\text{ for all }F\in\FF^{\Gamma}_h}. \\
\intertext{Moreover, for $p \in \N$, we employ the standard $\Ho$-conforming FEM spaces}
\SSS^p(\TT^S_h) 
&:=\set{\varphi_h\in\sfC^{0}(\overline{S})}{\str{\varphi_h}{T}\in\PP^p
\text{ for all }T\in\TT^S_h} \subset \Ho(S), \\
\SSS^p_{0}(\TT^S_h)
&:= \set{\varphi_h \in \SSS^p(\TT_h^S)}{\str{\varphi_h}{\partial S} = 0}\subset \Ho_0(S), \\
\SSS^p_{\Gamma^c}(\TT^S_h)
&:= \set{\varphi_h \in \SSS^p(\TT_h^S)}{\str{\varphi_h}{\Gamma^c} = 0}.
\end{align*}
Let $\FF^S_h$ denote the set of all interior faces, i.e., 
all $F\in\FF^S_h$ admit unique $T_{+},\, T_{-}\in\TT^S_h$ with $F = T_{+} \cap T_{-}$. 
For $q\in\N_{0}$, we define the $\Hdiv$-conforming Raviart--Thomas space
\begin{align*}
\RTq (\TT^S_h) 
&= \set{\ssigma_h\in\Linfty(S)}{\forall \, T\in\TT^S_h\quad 
\exists\, (\aa ,b)\in\PP^{q}(\R^d)^{d} \times \PP^{q}(\R^d) \quad \forall\, x\in T 
\\ &\hspace{10mm} 
\ssigma_h(x) = \aa(x) + b(x) \, x \quad\text{ and }\quad
 \forall \, F\in\FF^S_h\quad \normal_F \cdot[\ssigma_h]_F  = 0} \subset \H(\div,S),
\end{align*}
where $\normal_F$ is a normal vector for the face $F\in\FF^S_h$
and $[\ssigma_h]_F := \str{\ssigma_h}{T_+} - \str{\ssigma_h}{T_-}$ 
denotes the jump of $\ssigma_h$ across $F$. 
Based on that, we let
\begin{align*}
\RTq_{\Gamma^{c}} (\TT^S_h) 
:= \set{\ssigma_h\in\RTq (\TT^S_h)}{\ntr{\ssigma_h}{\Gamma^{c}} = 0}.
\end{align*}

\begin{remark}
\label{rem:extensionbyzero}
In the proofs of Section~\ref{section:approach} below, we exploit that for arbitrary $v_h \in \SSS^p_{\Gamma^c}(\TT_h^S)$ and $\ssigma_h \in \RTq_{\Gamma^c}(\TT_h^S)$ the definitions
\begin{align}
\check v_h := \begin{cases} v_h & \text{in }S \\ 0 & \text{in } \Omega \setminus S\end{cases} \qquad \text{and} \qquad \check \ssigma_h :=\begin{cases} \ssigma_h & \text{in } S \\ 0 & \text{in } \Omega\setminus S\end{cases}
\end{align}
provide conforming extensions $\check v_h \in \Ho(\Omega)$ and $\check\ssigma_h \in \Hdivom$. In particular, we will implicitly identify $v_h$ (resp. $\ssigma_h$) with its zero-extension $\check v_h$ (resp.~$\check\ssigma_h$).
\qed
\end{remark}

%%%%%%%%%%%%%%%%%%%%%%%%%%%%%%%%%%%%%%%%%%%%%%%%%%%%%%%%%%%%%%%%%%%%%%%%%%%%%%%%%%%
\subsection{General problem setting}\label{sec:genprobset}
%%%%%%%%%%%%%%%%%%%%%%%%%%%%%%%%%%%%%%%%%%%%%%%%%%%%%%%%%%%%%%%%%%%%%%%%%%%%%%%%%%%

From now on, we assume 
that $\Omega$, $\Gamma$, and a boundary layer $S$  
together with $\Gamma^{c}$ and corresponding FEM spaces are given.
Let $g\in\H^{1/2}(\Gamma)$ and let $u\in\Hoom$ be the unique 
solution of the homogeneous Dirichlet--Laplace problem
\begin{subequations}\label{eq:strongform0}
\begin{align}
\label{eq:strongform1}\Delta u &= 0 \quad \text{in } \Omega,\\
\label{eq:strongform2}u &= g \quad \text{on } \Gamma.
\end{align}
\end{subequations}
In particular, we have $\grad u\in\Hdivom$ with $\div\grad u=0$. 
Note that $u=\widehat g\in\Hoom$ is the unique harmonic extension of $g$
with $\norm{\grad u}{\Lt(\Omega)}\leq\norm{g}{\H^{1/2}(\Gamma)}$;
see~\eqref{eq:lifting}.

%%%%%%%%%%%%%%%%%%%%%%%%%%%%%%%%%%%%%%%%%%%%%%%%%%%%%%%%%%%%%%%%%%%%%%%%%%%%%%%%%%%
\subsection{Weakly-singular integral equation}
%%%%%%%%%%%%%%%%%%%%%%%%%%%%%%%%%%%%%%%%%%%%%%%%%%%%%%%%%%%%%%%%%%%%%%%%%%%%%%%%%%%

The single-layer potential~\eqref{eq:slp} provides a continuous linear operator 
$\widetilde V: \H^{-1/2}(\Gamma) \to \Hoom$. 
Moreover, its concatenation with the trace defines a continuous linear operator
$V:\H^{-1/2}(\Gamma)\to\H^{1/2}(\Gamma)$, which is elliptic on $\H^{-1/2}(\Gamma)$
(under the scaling condition $\diam(\Omega) < 1$ for $d = 2$). 
Hence, the Lax--Milgram lemma guarantees existence and uniqueness of 
$\phi\in\H^{-1/2}(\Gamma)$ such that
\begin{align}
\label{eq:weakform}
\dualpga{V\phi}{\psi}
=\dualpga{g}{\psi}
\quad\text{for all }\psi\in\H^{-1/2}(\Gamma).
\end{align}
According to the Hahn--Banach theorem, the latter variational formulation 
is equivalent to the identity $V \phi = g$ in $\H^{1/2}(\Gamma)$ from~\eqref{eq:weaksing}.
For details on elliptic boundary integral equations, we refer, e.g., to the monographs~\cite{mclean,hsiao-wendland}.

%%%%%%%%%%%%%%%%%%%%%%%%%%%%%%%%%%%%%%%%%%%%%%%%%%%%%%%%%%%%%%%%%%%%%%%%%%%%%%%%%%%
\subsection{Galerkin boundary element method}
%%%%%%%%%%%%%%%%%%%%%%%%%%%%%%%%%%%%%%%%%%%%%%%%%%%%%%%%%%%%%%%%%%%%%%%%%%%%%%%%%%%

Given a triangulation $\FF^{\Gamma}_h$ of $\Gamma$, 
the lowest-order Galerkin BEM seeks $\phi_h\in \PP^0(\FF^{\Gamma}_h)$, which solves the discretized weak form
\begin{align}
\label{eq:discreteweakform}
\scp{V\phi_h}{\psi_h}{\Lt(\Gamma)}
=\scp{g}{\psi_h}{\Lt(\Gamma)}
\quad\text{for all }\psi_h\in\PP^0(\FF^{\Gamma}_h). 
\end{align}
The Lax--Milgram lemma also applies to the conforming Galerkin discretization and proves 
existence and uniqueness of $\phi_h\in \PP^0(\FF^{\Gamma}_h)$. 
We note that in the discrete version~\eqref{eq:discreteweakform} of~\eqref{eq:weakform} the $\H^{1/2}(\Gamma) \times \H^{-1/2}(\Gamma)$ duality product coincides, in fact, with the $\Lt(\Gamma)$ scalar product.
For details on the (Galerkin) boundary element method, we refer, e.g., to the monographs~\cite{steinbach,sauter-schwab,gwinner-stephan}.

%%%%%%%%%%%%%%%%%%%%%%%%%%%%%%%%%%%%%%%%%%%%%%%%%%%%%%%%%%%%%%%%%%%%%%%%%%%%%%%%%%%
%%%%%%%%%%%%%%%%%%%%%%%%%%%%%%%%%%%%%%%%%%%%%%%%%%%%%%%%%%%%%%%%%%%%%%%%%%%%%%%%%%%
\section{Functional a~posteriori BEM error estimation}
\label{section:approach}

In this section, we prove the error identity~\eqref{eq:erroridentity} and provide efficiently computable upper and lower bounds for the potential error $\norm{\nabla (u - u_h)}{\Lt(\Omega)}$, where $u \in \Ho(\Omega)$ solves~\eqref{eq:strongform0} and $u_h := \widetilde V \phi_h$ is defined in~\eqref{eq:slp}.

%%%%%%%%%%%%%%%%%%%%%%%%%%%%%%%%%%%%%%%%%%%%%%%%%%%%%%%%%%%%%%%%%%%%%%%%%%%%%%%%%%%
\subsection{Functional error identity}
%%%%%%%%%%%%%%%%%%%%%%%%%%%%%%%%%%%%%%%%%%%%%%%%%%%%%%%%%%%%%%%%%%%%%%%%%%%%%%%%%%%

The fact that the error $u-u_h$ satisfies~\eqref{eq:strongform1} exactly is a powerful tool. However, the consideration of the potential $u_h$ from a BEM comes with a drawback: it is not a discrete function and lacks further {\sl a~priori} knowledge like the Galerkin orthogonality, which is obviously never available for any approximation $u_h := \widetilde V \phi_h \approx u \in \Hoom$. Functional {\sl a~posteriori} error estimates are eminently suitable for the BEM, since they do not require any such {\sl a~priori} assumption. On top of that, for problems with homogeneous (volume) right-hand sides,
they provide \emph{constant-free} error identities. % by means of so-called primal and dual problems. 
For the Laplacian, the key argument is the Dirichlet principle: 
\begin{center}Harmonic functions are minimisers of the Dirichlet energy $\norm{\grad w}{\Lt(\Omega)}^2$.\end{center}
Note that the boundary residual $g-\str{u_h}{\Gamma}\in\H^{1/2}(\Gamma)$ is essential 
for both the majorant $\mmax$ and the minorant $\mmin$, 
\revision{see~\eqref{eq:prop:continuousmajmin} in Theorem \ref{prop:continuousmajmin} for definitions,}
and comprises all relevant information about the error.

\begin{theorem}[Functional \textsl{a~posteriori} error identities]
\label{prop:continuousmajmin}
\revision{Let $g\in\H^{1/2}(\Gamma)$ and let $u \in \Ho(\Omega)$ be the unique solution of~\eqref{eq:strongform0}.}
For any approximation $v\in\Hoom$ with $\Delta v=0$,
the equalities~\eqref{eq:erroridentity} hold true.
More precisely, 
\begin{subequations}
\label{eq:prop:continuousmajmin}
\begin{align}
\label{eq:continuousmajmin}
 \max_{\substack{\ttau\in\Lt(\Omega) \\ \div\ttau = 0}} \mmin(\ttau;\str{v}{\Gamma},g)
 = \bnorm{\grad(u-v)}{\Lt(\Omega)}^2
 = \min_{\substack{w\in\Hoom \\ \str{w}{\Gamma} = g - \str{v}{\Gamma}}} \mmax(\grad w),
\end{align}
where 
\begin{align}
\label{def:functionals}
\mmin(\ttau;\str{v}{\Gamma},g) 
:=  2 \,\dualpga{g-\str{v}{\Gamma}}{\ntr{\ttau}{\Gamma}} 
 - \norm{\ttau}{\Lt(\Omega)}^2,\qquad
\mmax(\grad w) 
:= \norm{\grad w}{\Lt(\Omega)}^2.
\end{align}
\end{subequations}
The unique maximiser is $\underline{\ttau} = \grad (u-v)$. 
The unique minimiser is $\overline w = u - v$.
\end{theorem}

\begin{proof}
The proof is split into two parts.

$\bullet$ \textbf{Upper bound:} 
Let $\widetilde w\in\Hoom$ with $\str{\widetilde w}{\Gamma} = \str{u}{\Gamma} = g$.
Since we have $\Delta(u-v)=0$ and $u-\widetilde w\in\Hozom$, integration by parts shows that
\begin{align*}
\bnorm{\grad(u-v)}{\Lt(\Omega)}^2 
= \underbrace{\bscp{\grad(u-\widetilde w)}{\grad(u-v)}{\Lt(\Omega)}}_{=0} + \bscp{\grad (\widetilde w-v)}{\grad(u-v)}{\Lt(\Omega)}.
\end{align*}
This yields
$\norm{\grad(u-v)}{\Lt(\Omega)} \leq
\norm{\grad(\widetilde w-v)}{\Lt(\Omega)}$.
The substitution $w:=\widetilde w - v$ proves that
\begin{align*}
\bnorm{\grad(u-v)}{\Lt(\Omega)} \leq \inf_{\substack{w\in\Hoom \\ \str{w}{\Gamma}
= g - \str{v}{\Gamma}}} \norm{\grad w}{\Lt(\Omega)}.
\end{align*}
The unique infimum is attained at $w = u-v$.

$\bullet$ \textbf{Lower bound:} 
In any Hilbert space $\mathcal{H}$
\revision{with inner product $\scp{\cdot}{\cdot}{\mathcal{H}}$ and induced norm $\norm{\cdot}{\mathcal{H}}$,} it holds that
\begin{align*}
\norm{a}{\mathcal{H}}^2 
= \max_{b\in\mathcal{H}} \left(2 \,\scp{a}{b}{\mathcal{H}} - \norm{b}{\mathcal{H}}^2 \right)
\quad \text{for all } a\in\mathcal{H},
\end{align*}
where the maximum is unique and attained for $b = a$.
Since 
\begin{align*}
\grad (u-v)\in\mathcal{H}
:=\set{\ssigma\in\Hdivom}{\div\ssigma=0},
\end{align*}
we have
\begin{align*}
 \bnorm{\grad(u-v)}{\Lt(\Omega)}^2 
 = \bnorm{\grad(u-v)}{\mathcal{H}}^2 
 &= \max_{\substack{\ttau\in\Lt(\Omega) \\ \div \ttau = 0}}
 \Big(2 \,\bscp{\grad(u-v)}{\ttau}{\Lt(\Omega)} 
 - \norm{\ttau}{\Lt(\Omega)}^2 \Big)\\
 &= \max_{\substack{\ttau\in\Lt(\Omega) \\ \div \ttau = 0}}
  \Big( \, 2 \,\dualpga{g-\str{v}{\Gamma}}{\ntr{\ttau}{\Gamma}} 
 - \ \norm{\ttau}{\Lt(\Omega)}^2 \Big).
\end{align*}
In particular, the maximum is attained for $\ttau = \grad(u-v)$.
This concludes the proof.
\end{proof}

%%%%%%%%%%%%%%%%%%%%%%%%%%%%%%%%%%%%%%%%%%%%%%%%%%%%%%%%%%%%%%%%%%%%%%%%%%%%%%%%%%%
To ease the readability, the remainder of this chapter focusses on our numerical setup. For the functional analytic framework in a Sobolev space setting, which might be of independent interest, we refer to 
to the appendix of the extended preprint~\cite{KPPRS2019a} of this work.
%Appendix~\ref{section:appendix} below.

%%%%%%%%%%%%%%%%%%%%%%%%%%%%%%%%%%%%%%%%%%%%%%%%%%%%%%%%%%%%%%%%%%%%%%%%%%%%%%%%%%%
\subsection{Computable error bounds}
%%%%%%%%%%%%%%%%%%%%%%%%%%%%%%%%%%%%%%%%%%%%%%%%%%%%%%%%%%%%%%%%%%%%%%%%%%%%%%%%%%%

We aim at error bounds obtained by solving FEM problems on a boundary layer $S \subset \Omega$. 
For the maximization problem in~\eqref{eq:prop:continuousmajmin}, the constraint $\div \ttau = 0$ can be realized by a mixed formulation (see also \cite[Lemma~15]{KPPRS2019a} of the extended preprint of this work).
However, the boundary condition $\str{w}{\Gamma} = g - \str{v}{\Gamma}$ 
cannot be satisfied exactly by any piecewise polynomial solution $w_h$ 
corresponding to~\eqref{eq:prop:continuousmajmin}. Therefore, the upper bound involves an additional oscillation term given by 
a discretisation operator $J_h$, \revision{which will be the $\Lt(\Gamma)$-orthogonal projection in the numerical experiments of Sections~\ref{section:numerics} and~\ref{section:extensions} below.}

\begin{theorem}[Computable bounds via boundary layer]
\label{cor:continuousupperbound}
Let $v\in\Hoom$ with $\Delta v=0$. Let $p\in\N$ and let 
$J_h: \H^{1/2}(\Gamma) \to \SSS^p(\FF_h^{\Gamma})
:=\set{\str{\varphi_h}{\Gamma}}{\varphi_h \in \SSS^p(\TT^S_h)}$
\revision{be an arbitrary projection operator.}
Moreover, let $w_h\in \SSS^p(\TT^S_h)$ be the unique solution of
\begin{align}
\label{eq:upper-bound-discrete}
 \scp{\grad w_h}{\grad\varphi_h}{\Lt(S)} = 0 
 \quad \text{for all }\, \varphi_h\in \SSS^p_{0}(\TT^S_h)
\ \text{ with } \
\str{w_h}{\partial S} = \begin{cases}
J_h(g - \str{v}{\Gamma}) \, & \!\text{on } \Gamma,\\
0 & \!\text{on } \Gamma^{c}.
\end{cases}
\end{align}
For $q\in\N_{0}$, let the pair $(\ttau_h,\omega_h)\in\RTqg (\TT^S_h) \times \PP^{q}(\TT^S_h)$ 
be the unique solution of
\begin{subequations}
\label{eq:mixedproblemstripdiscrete}
\begin{align}
\label{eq:mixedproblemstripdiscrete1}
\scp{\ttau_h}{\ssigma_h}{\Lt(S)} 
+ \scp{\div \ssigma_h}{\omega_h}{\Lt(S)} 
&= \scp{g - \str{v}{\Gamma}}{\ntr{\ssigma_h}{\Gamma}}{\Lt(\Gamma)},\\
%\quad \text{for all } \ssigma_h \in\RTqg (\TT^S_h)
\label{eq:mixedproblemstripdiscrete2}
\scp{\div\ttau_h}{\psi_h}{\Lt(S)} &= 0
%\hspace*{37.5mm} \text{for all } \psi_h \in\PP^{q}(\TT^S_h).
\end{align}
\end{subequations}
for all pairs $(\ssigma_h,\psi_h)\in\RTqg (\TT^S_h) \times \PP^{q}(\TT^S_h)$.
Then, it holds that
\begin{subequations}
\begin{align}
\label{min:discrete}
2 \,\scp{g - \str{v}{\Gamma}}{\ntr{\ttau_h}{\Gamma}}{\Lt(\Gamma)} - \norm{\ttau_h}{\Lt(S)}^2 
 &\le \norm{\grad(u-v)}{\Lt(\Omega)}^2
 \\& \label{maj:discrete}
 \le \norm{\grad w_h}{\Lt(S)} + \bnorm{(1-J_h) (g-\str{v}{\Gamma})}{\H^{1/2}(\Gamma)}.
\end{align}
\end{subequations}
\end{theorem}

\begin{proof}
It is well-known that~\eqref{eq:upper-bound-discrete} admits a unique solution $w_h\in \SSS^p(\TT^S_h)$, being the natural FEM discretization of an homogeneous Dirichlet--Laplace problem with inhomogeneous Dirichlet conditions; see, e.g., \cite{bcd,sacchiveeser,afkpp2013}. 
To prove the upper bound~\eqref{maj:discrete}, let
$\widehat{f}_h\in\Hoom$ be the (unique) harmonic extension of $f_h:=(1-J_h)(g - \str{v}{\Gamma})$;
see Remark~\ref{remharmext}. Then, 
\revision{Theorem~\ref{prop:continuousmajmin} 
and $\norm{\grad\widehat{f}_h}{\Lt(\Omega)}\leq\norm{f_h}{\H^{1/2}(\Gamma)}$ lead to
\begin{align*}
\bnorm{\grad(u-v)}{\Lt(\Omega)}
=\min_{\substack{w\in\Hoom \\ \str{w}{\Gamma} = g - \str{v}{\Gamma}}} 
\norm{\grad w}{\Lt(\Omega)}
&\leq\min_{\substack{w\in\Hoom \\ \str{w}{\Gamma} = g - \str{v}{\Gamma}}} 
\bnorm{\grad(w-\widehat{f}_h)}{\Lt(\Omega)}
+ \norm{\grad\widehat{f}_h}{\Lt(\Omega)} \\
&\leq\min_{\substack{w\in \Hoom \\ \str{w}{\Gamma} = J_h(g - \str{v}{\Gamma})}} 
\norm{\grad w}{\Lt(\Omega)}
+ \norm{f_h}{\H^{1/2}(\Gamma)},
\end{align*}
where we have finally employed the substitution $w-\widehat{f}_h\leadsto w$.
Since the zero-extension of $w_h$ belongs to $\Hoom$ according to Remark~\ref{rem:extensionbyzero}
and satisfies the correct boundary condition,
this proves the computable upper bound~\eqref{maj:discrete}.}
%that
%\begin{align*}
%\bnorm{\grad(u-v)}{\Lt(\Omega)}
%&\leq\min_{\substack{w\in\Hoom \\ \str{w}{\Gamma} = J_h(g - \str{v}{\Gamma})}} 
%\norm{\grad w}{\Lt(\Omega)}
%+\bnorm{(1-J_h)(g-\str{v}{\Gamma})}{\H^{1/2}(\Gamma)} \\
%&\le \norm{\grad w_h}{\Lt(S)}
%+\bnorm{(1-J_h)(g-\str{v}{\Gamma})}{\H^{1/2}(\Gamma)}
%\end{align*}
%and hence verifies the computable upper bound~\eqref{maj:discrete}.}

For existence and uniqueness of~\eqref{eq:mixedproblemstripdiscrete}, we refer, e.g., to \cite{boffibrezzifortin,ccb}. 
Since $\div \ttau_h \in \PP^q(\TT_h^S) \subset \Lt(\Omega)$ by definition of $\RTq(\TT_h^S)$, it follows from~\eqref{eq:mixedproblemstripdiscrete2} that $\ttau_h \in \RTq_{\Gamma^c}(\TT_h^S) \subset \H(\div,S)$ with $\div \ttau_h = 0$ in $S$. According to Remark~\ref{rem:extensionbyzero}, the zero-extension of $\ttau_h$ belongs to $\Hdivom$ with $\div \ttau_h = 0$ in $\Omega$. The computable lower bound~\eqref{min:discrete} thus follows from Theorem~\ref{prop:continuousmajmin}.
\end{proof}
%
%\revision{
%\begin{remark}
%\label{cor:continuousupperbound:rem}
%Later, in the computational sections, the projection operator $J_h$
%will be the $\Lt(\Gamma)$-orthogonal projection.
%\end{remark}
%}

\def\vec{\boldsymbol}
%%%%%%%%%%%%%%%%%%%%%%%%%%%%%%%%%%%%%%%%%%%%%%%%%%%%%%%%%%%%%%%%%%%%%%%%%%%%%%%%%%%
In order to circumvent the implementation of the constraint $\div \ttau = 0$, it is also an option to reformulate the maximization problem in~\eqref{eq:prop:continuousmajmin} by means of potentials.
While the 3D case involves vector potentials, for 2D such an approach is particularly attractive due to the possible use of scalar potentials. In the following, we thus concentrate on $d = 2$ (and refer, for $d = 3$, to the appendix of the extended preprint~\cite{KPPRS2019a} of this work).
%Appendix~\ref{section:appendix} ).
%
To this end, we recall the definitions of the 2D curl operators
\begin{align*}
 \vec\curl \, \varphi
=\begin{bmatrix}
-\partial_2\varphi\\
\partial_1\varphi
 \end{bmatrix}
\quad \text{for } \varphi: \Omega \to \R
\quad \text{resp.} \quad
 \curl \, \boldsymbol\varphi = \partial_1 \boldsymbol\varphi_2 - \partial_2 \boldsymbol\varphi_1
\quad \text{for } \boldsymbol\varphi: \Omega \to \R^2.
\end{align*}
Note that $\div\vec\curl \, \varphi = 0.$
For $\varphi\in\Hoom$, we thus have $\vec\curl \, \varphi\in\Hdivom$ so that the Neumann trace
\revision{$\ntr{\vec\curl\,\varphi}{\Gamma}\in\H^{-1/2}(\Gamma)$ is well-defined. In particular, we have
\begin{align*}
\scp{\grad\varphi}{\grad\psi}{\Lt(\Omega)}
=\scp{\vec\curl\varphi}{\vec\curl\psi}{\Lt(\Omega)}
\quad \text{for all }\varphi,\psi\in\Hoom.
\end{align*}}%

\begin{corollary}[Computable lower bound via boundary layer --- $\H^1$-conforming]
\label{cor:continuouslowerboundtwo}
Suppose that $d=2$. Let $v\in\Hoom$ with $\Delta v=0$. 
For $p \in \N$, let $\widetilde w_h\in\SSS^p_{\Gamma^{c}}(\TT_h^S)$ be the unique solution of
\begin{align}
\label{eq:lower-bound2ddiscrete}
\scp{\grad \widetilde w_h}{\grad\varphi_h}{\Lt(S)}
=\scp{g-\str{v}{\Gamma}}{\ntr{\vec\curl \, \varphi_h}{\Gamma}}{\Lt(\Gamma)}
\quad\text{for all }\varphi_h\in\SSS^p_{\Gamma^{c}}(\TT^S_h).
\end{align}
Then, it holds that
\begin{align}
\label{eq2:lower-bound2ddiscrete}
2\scp{g-\str{v}{\Gamma}}{\ntr{\vec\curl \, \widetilde w_h}{\Gamma}}{\Lt(\Gamma)}
-\norm{\grad \widetilde w_h}{\Lt(S)}^2
\leq\bnorm{\grad(u-v)}{\Lt(\Omega)}^2.
\end{align}
\end{corollary}

\begin{proof}
It is well-known that~\eqref{eq:lower-bound2ddiscrete} admits a unique solution $\widetilde w_h \in\SSS^p_{\Gamma^{c}}(\TT_h^S)$ being the natural FEM discretization 
of a mixed Dirichlet--Neumann--Laplace problem; see, e.g.,~\cite{bcd}. According to Remark~\ref{rem:extensionbyzero}, the zero-extension of $\widetilde w_h$ belongs to $\Hoom$ and hence $\widetilde\ttau_h:=\vec\curl \, \widetilde w_h \in\Hdivom$ satisfies that $\div\widetilde\ttau_h=0$ with $\ntr{\vec\curl \, \widetilde w_h}{\Gamma} = \ntr{\widetilde\ttau_h}{\Gamma}$ and $\norm{\widetilde\ttau_h}{\Lt(\Omega)} = \norm{\vec\curl \, \widetilde w_h}{\Lt(\Omega)} = \norm{\nabla \widetilde w_h}{\Lt(\Omega)}$. The claim thus follows from Theorem~\ref{prop:continuousmajmin}.
\end{proof}

%%%%%%%%%%%%%%%%%%%%%%%%%%%%%%%%%%%%%%%%%%%%%%%%%%%%%%%%%%%%%%%%%%%%%%%%%%%%%%%%%%%
%%%%%%%%%%%%%%%%%%%%%%%%%%%%%%%%%%%%%%%%%%%%%%%%%%%%%%%%%%%%%%%%%%%%%%%%%%%%%%%%%%%
\section{Adaptive algorithm}
\label{section:algorithm}
%%%%%%%%%%%%%%%%%%%%%%%%%%%%%%%%%%%%%%%%%%%%%%%%%%%%%%%%%%%%%%%%%%%%%%%%%%%%%%%%%%%
%%%%%%%%%%%%%%%%%%%%%%%%%%%%%%%%%%%%%%%%%%%%%%%%%%%%%%%%%%%%%%%%%%%%%%%%%%%%%%%%%%%

\begin{figure}[t]
\includegraphics[trim={4.1cm 1.1cm 3.4cm 1cm},clip,width=.22\textwidth]{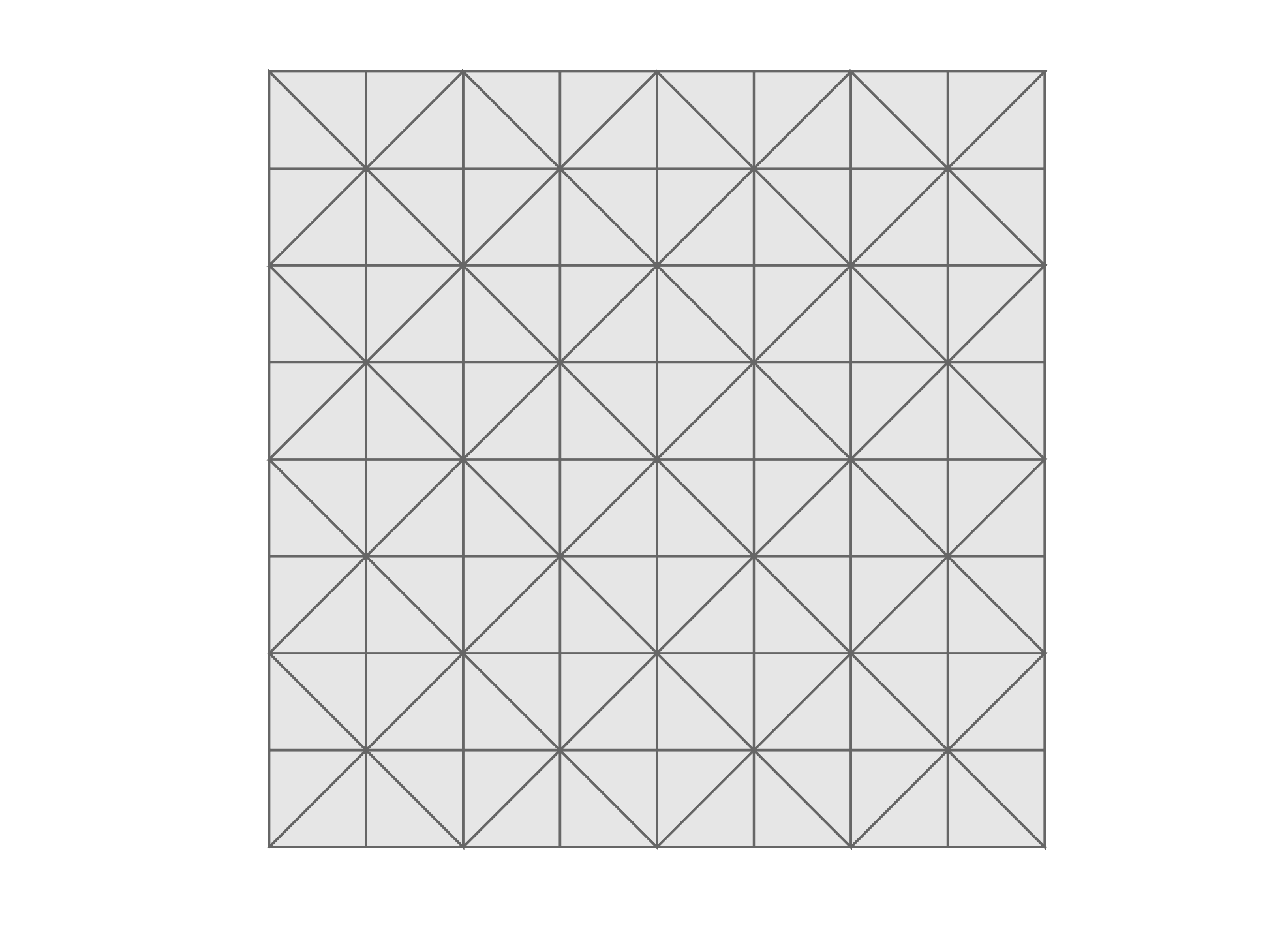}\hfill
\includegraphics[trim={4.1cm 1.1cm 3.4cm 1cm},clip,width=.22\textwidth]{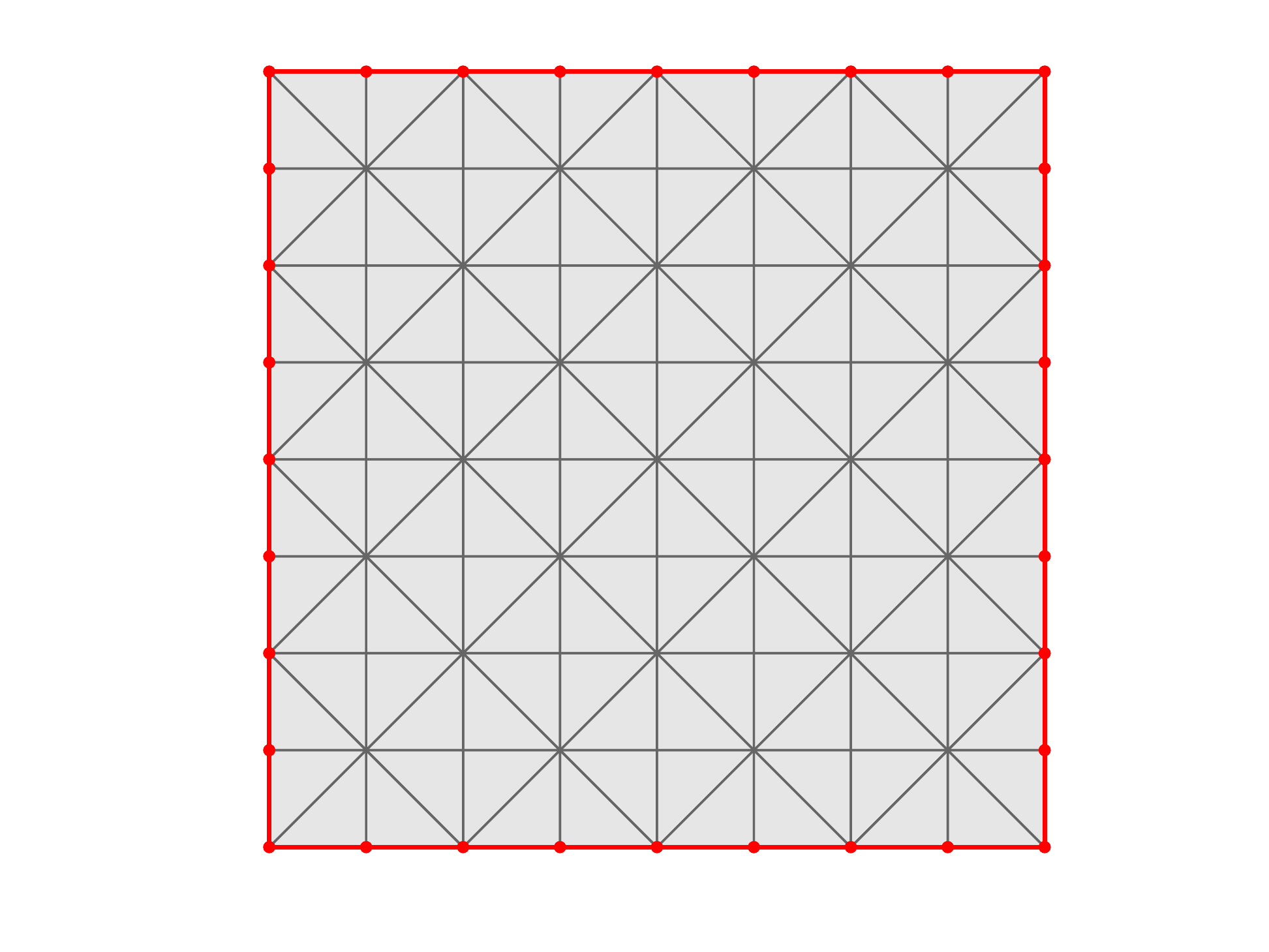}\hfill
\includegraphics[trim={4.1cm 1.1cm 3.4cm 1cm},clip,width=.22\textwidth]{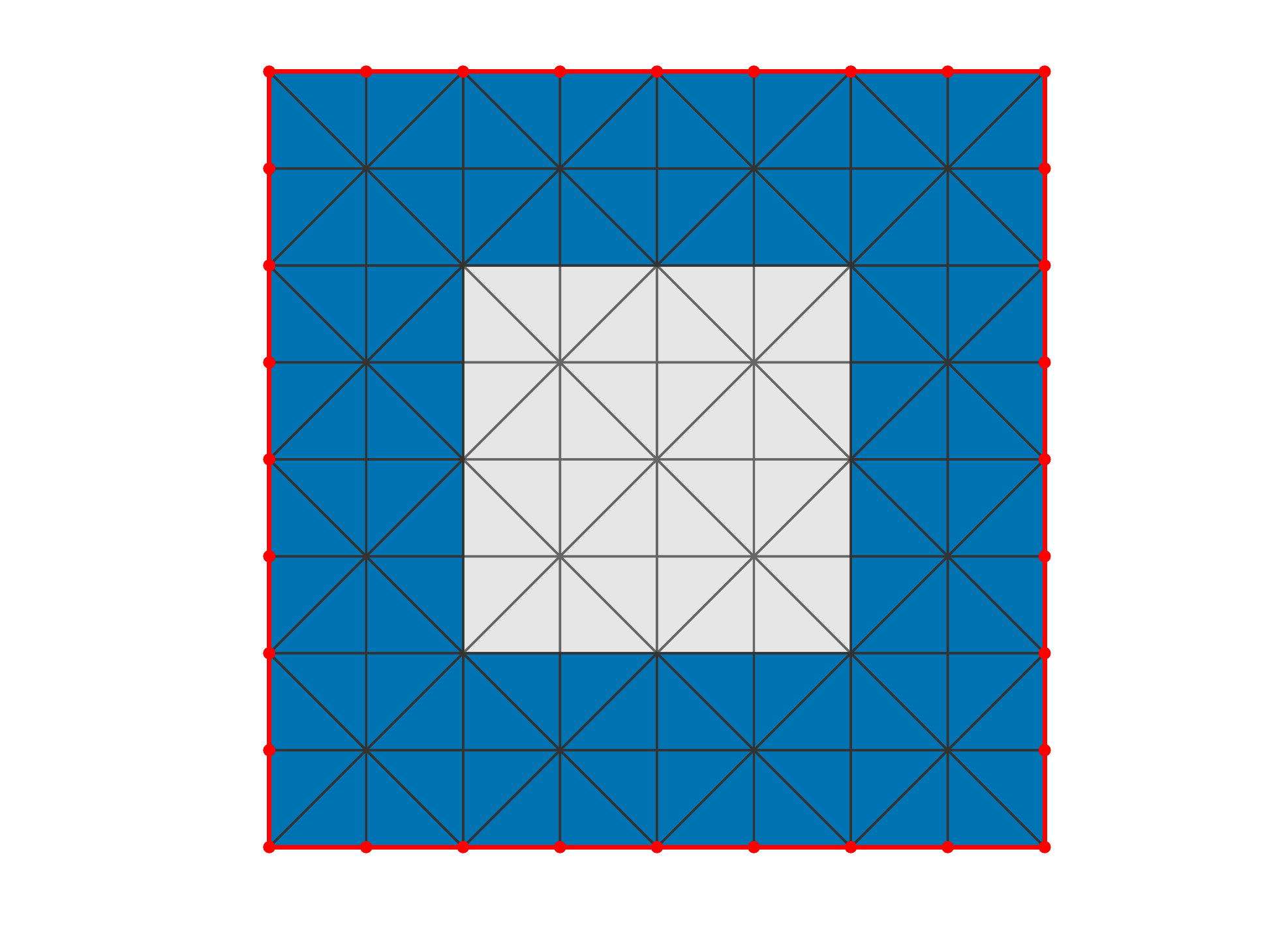}\hfill
\includegraphics[trim={4.1cm 1.1cm 3.4cm 1cm},clip,width=.22\textwidth]{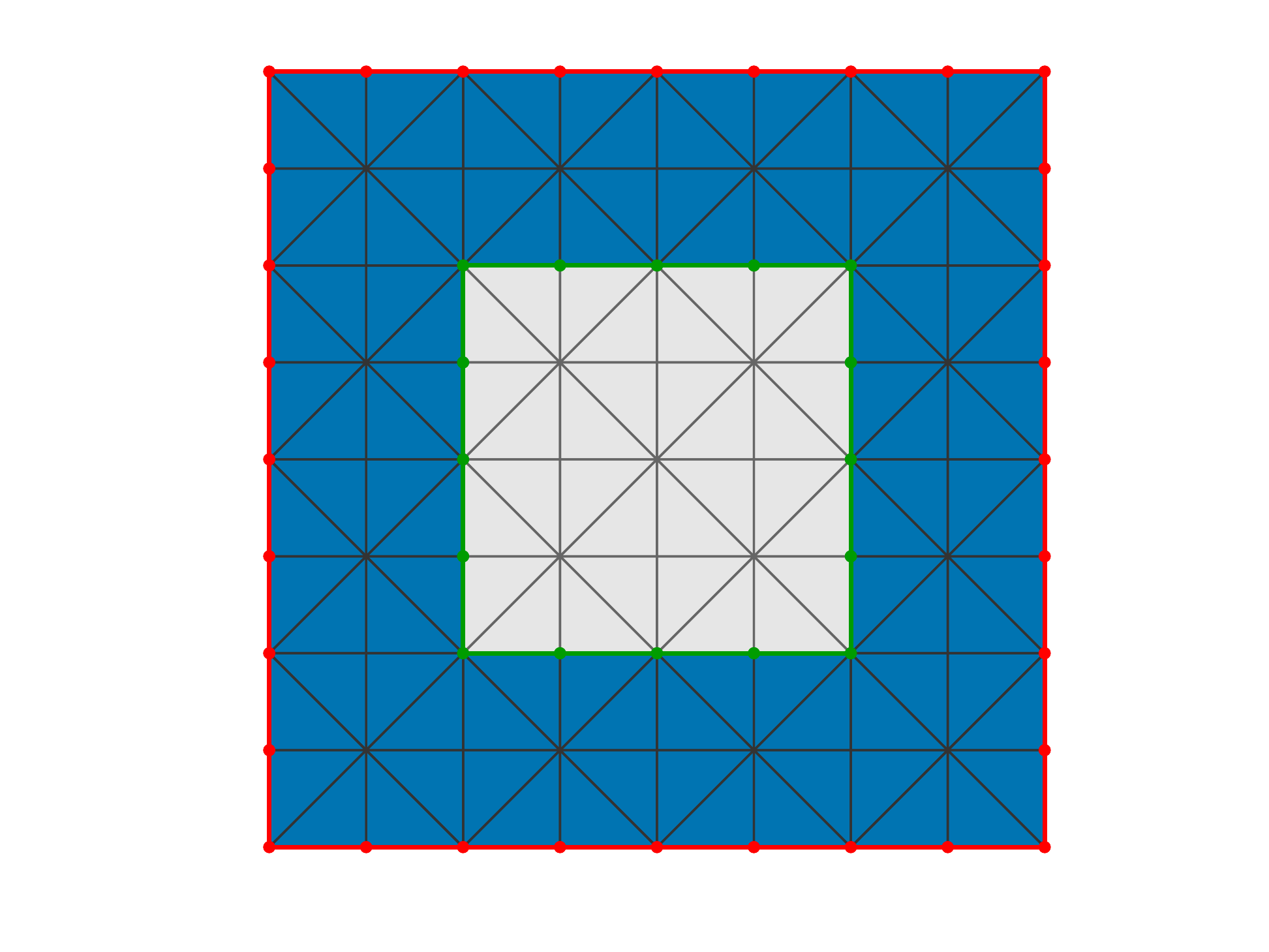}
\captionsetup{width=1.015\textwidth}
\caption{Example geometry $\Omega = (0,1/2)^2$ with FEM triangulation $\TT_h$ (gray, left), induced BEM mesh $\cred\boldsymbol{\FF^\Gamma_h}$ on $\cred\boldsymbol{\Gamma = \partial\Omega}$ {\cred\bf(red)}, generated boundary layer $\color{tuwBlue}\boldsymbol{S}$ with mesh $\color{tuwBlue}\boldsymbol{\TT_h^S}$ {\color{tuwBlue}\bf(blue)}, and interior boundary $\color{green!50!black}\boldsymbol{\Gamma^c}$ {\color{green!50!black}\bf(green)}, illustrated from left to right.}\label{fig:extractmesh}
\end{figure}
%
%%%%%%%%%%%%%%%%%%%%%%%%%%%%%%%%%%%%%%%%%%%%%%%%%%%%%%%%%%%%%%%%%%%%%%%%%%%%%%%%%%%
\subsection{Triangulations and mesh-refinement}
%%%%%%%%%%%%%%%%%%%%%%%%%%%%%%%%%%%%%%%%%%%%%%%%%%%%%%%%%%%%%%%%%%%%%%%%%%%%%%%%%%%

In our numerical experiments, we start from a conforming simplicial triangulation $\TT_h$ such that $\Gamma \subset \bigcup_{T \in \TT_h}T \subseteq \overline\Omega$. We obtain the boundary layer $S \subset \Omega$ as the second-order patch of $\Gamma$ with respect to $\TT_h$, i.e.,
\begin{align}\label{eq:mesh:layer}
 \vspace*{-1mm}\TT_h^S := \set{T\in\TT_h}{\exists \, T'\in\TT_h, \quad T' \cap \Gamma \neq \emptyset \neq T \cap T'}
 \text{ \ and \ } S := {\rm interior}\Big( \bigcup_{T\in\TT_h^S} T \Big).
\end{align}
Moreover, recall the BEM mesh $\FF_h^\Gamma := \TT_h^S|_\Gamma = \TT_h|_\Gamma $ from~\eqref{eq:boundary-mesh}. These definitions are illustrated in Figure~\ref{fig:extractmesh}. 

For (local) mesh-refinement, we employ newest \revision{2D} vertex bisection~\cite{stevenson2008,kpp2013}; \revision{see also~\cite[Section~5.2]{p1afem} for a short but precise statement of the algorithm and the MATLAB implementation we build on.}
The adaptive strategy will only mark elements of $\TT_h^S$, but refinement will be done with respect to the full triangulation $\TT_h$. In particular, we stress that the second-order patch $S$ will generically change, if the triangulation $\TT_h$ is refined; see, e.g., Figure~\ref{fig:example1:mesh}. In this way, we guarantee that the number of degrees of freedom with respect to $\TT_h^S$ will increase proportionally to those with respect to $\FF^\Gamma_h$; see also Tables~\ref{table2}--\ref{table5} below, \revision{where $\#\FF^\Gamma_h$ denotes the number of BEM elements, while $\#\TT_h^S$ denotes the number of FEM elements in the boundary layer $S \subset \Omega$.}

%%%%%%%%%%%%%%%%%%%%%%%%%%%%%%%%%%%%%%%%%%%%%%%%%%%%%%%%%%%%%%%%%%%%%%%%%%%%%%%%%%%
\subsection{Data oscillations}\label{sec:oscillations}
%%%%%%%%%%%%%%%%%%%%%%%%%%%%%%%%%%%%%%%%%%%%%%%%%%%%%%%%%%%%%%%%%%%%%%%%%%%%%%%%%%%

The upper bound~\eqref{maj:discrete}
in 
Theorem~\ref{cor:continuousupperbound}
involves the data approximation term 
$\bnorm{(1-J_h) (g-\str{u_h}{\Gamma})}{\H^{1/2}(\Gamma)}$, where $u_h:= \widetilde V \phi_h$. 
Besides the fact that we still have to specify the operator 
$J_h: \H^{1/2}(\Gamma) \to \SSS^p(\FF_h^\Gamma)$
\revision{from Theorem \ref{cor:continuousupperbound}},
%and Remark \ref{cor:continuousupperbound:rem},}
we note that the nonlocal nature of the $\H^{1/2}(\Gamma)$-norm makes this term hardly computable.

In the following, we choose 
\revision{
\begin{subequations}\label{Jhpdef}
\begin{align}
J_h: \Lt(\Gamma) \to \SSS^p(\FF^{\Gamma}_h)
= \set{\str{\varphi_h}{\Gamma}}{\varphi_h\in\SSS^p(\TT_h)},
%\quad
%p\in\{1,2\},
\end{align}
as the $\Lt(\Gamma)$-orthogonal projection onto %the FEM space 
$\SSS^p(\FF^{\Gamma}_h)$,
which is uniquely determined by
\revision{
\begin{align}
 \scp{J_h\varphi}{\psi_h}{\Lt(\Gamma)}
 = \scp{\varphi}{\psi_h}{\Lt(\Gamma)}
 \quad \text{for all } \varphi \in \Lt(\Gamma) \text{ and all } \psi_h \in \SSS^p(\FF_h^\Gamma).
\end{align}}
\end{subequations}}%
For $d = 2$, it follows under mild conditions on 
$\FF^{\Gamma}_h$ that $J_h$ is $\Ho(\Gamma)$-stable, i.e.,
\begin{align}
\label{eq:H1-stable}
\norm{\grad J_hf}{\Lt(\Gamma)} 
\leq\Cstab\norm{\grad f}{\Lt(\Gamma)}
\quad\text{for all }f\in\Ho(\Gamma);
\end{align}
see~\cite{ct1987}. We note that these conditions are automatically satisfied for $\FF^{\Gamma}_h = \TT_h|_{\Gamma}$, 
since $\TT_h$ is only refined by newest vertex bisection. For $d = 3$, the $\Ho(\Gamma)$-stability~\eqref{eq:H1-stable} 
is known for low-order FEM (on the 2D manifold $\Gamma$); see~\cite{kpp2013} for $p = 1$ and~\cite{ghs2016} for $p \in \{1, \dots, 12\}$. We recall the following result from~\cite{affkp2015}:

\begin{lemma}
\label{lemma:osc}
\revision{If the $\Lt(\Gamma)$-orthogonal projection $J_h: \Lt(\Gamma) \to \SSS^p(\FF_h^\Gamma)$ from~\eqref{Jhpdef}}
is $\Ho(\Gamma)$-stable~\eqref{eq:H1-stable}, then it holds for all $f\in\Ho(\Gamma)$ that
\begin{align*}
\Cosc^{-1}\bnorm{(1-J_h)f}{\H^{1/2}(\Gamma)}
\leq\min_{f_h\in\SSS^p(\FF_h^\Gamma)} 
\bnorm{f-f_h}{\H^{1/2}(\Gamma)}
\leq\Cosc\min_{f_h\in\SSS^p(\FF_h^\Gamma)} 
\bnorm{h^{1/2}\grad_{\Gamma}(f-f_h)}{\Lt(\Gamma)},
\end{align*}
where $h\in\Linfty(\Omega)$ is the local mesh-width function defined by 
$h|_{F} := \diam(F)$ 
for all $F\in\FF^{\Gamma}_h$. 
The constant $\Cosc > 0$ depends only on $\Cstab$ and the shape regularity of $\TT_h$.
\end{lemma}

Provided that the given Dirichlet boundary data satisfy $g\in\Ho(\Gamma)$, 
the foregoing lemma allows to dominate the data approximation term by
\begin{align}
\label{eq:osc}
\Cosc^{-2}
 \bnorm{(1-J_h) (g-\str{u_h}{\Gamma})}{\H^{1/2}(\Gamma)}
 \leq \bnorm{h^{1/2}\grad_{\Gamma} ((1-J_h)(g-\str{u_h}{\Gamma}))}{\Lt(\Gamma)}
 =:{\rm osc}_h,
\end{align}
where ${\rm osc}_h$ is, in fact, computable, 
while the constant $\Cosc$ is generic and hardly accessible. 
With $v=u_h$, the upper bound~\eqref{maj:discrete} becomes 
\begin{align}
\label{eq2:maj:discrete}
\begin{split}
\bnorm{\grad (u-u_h)}{\Lt(\Omega)} 
 &\le \norm{\grad w_h}{\Lt(S)} 
 + \bnorm{(1-J_h) (g-\str{u_h}{\Gamma})}{\H^{1/2}(\Gamma)} \\
 &\le \norm{\grad  w_h}{\Lt(S)} 
 + \Cosc^2 \, {\rm osc}_h,
\end{split}
\end{align}
where $w_h\in \SSS^p(\TT^S_h)$ solves~\eqref{eq:upper-bound-discrete}. For the use in the adaptive algorithm, we note that
\begin{align}\label{eq:def:osc}
 {\rm osc}_h^2 = \! \sum_{T \in \TT_h^S}\! {\rm osc}_h(T)^2,
 \text{ where }
 {\rm osc}_h(T)^2 := \!\! \sum_{\substack{F \in \FF_h^\Gamma \\ F \subset T}} \! \diam(F) \, \norm{\grad_{\Gamma} ((1 \!-\! J_h) (g \!-\! \str{u_h}{\Gamma}))}{\Lt(F)}^2. 
\end{align}

\begin{remark}
In our numerical experiments, we will consider $p = 1$ as well as $p = 2$ to compute the uppermost bound in~\eqref{eq2:maj:discrete}.
Since the lower bound~\eqref{min:discrete} is independent of the data approximation, we did only implement the lowest-order case $q = 0$.
\qed
\end{remark}

\begin{remark}\label{remark:Jh}
Instead of the $\Lt(\Gamma)$-orthogonal projection, one can also employ the Scott--Zhang projector; see~\cite{afkpp2013,ffkmp2014}. 
Then, Lemma~\ref{lemma:osc} as well as~\eqref{eq:osc} hold accordingly. 
For $d = 2$, one can also employ nodal projection. 
While generic $\H^{1/2}(\Gamma)$ functions do not have to be
continuous and Lemma~\ref{lemma:osc} fails, 
one can still prove~\eqref{eq:osc}; see~\cite{fpp2014,ffkmp2014}.
\qed
\end{remark}

%%%%%%%%%%%%%%%%%%%%%%%%%%%%%%%%%%%%%%%%%%%%%%%%%%%%%%%%%%%%%%%%%%%%%%%%%%%%%%%%%%%
\subsection{Adaptive algorithm}
%%%%%%%%%%%%%%%%%%%%%%%%%%%%%%%%%%%%%%%%%%%%%%%%%%%%%%%%%%%%%%%%%%%%%%%%%%%%%%%%%%%

%We propose the following algorithm, which is 
%empirically 
%tested for 2D model problems below.

\revision{The above discussed estimates and relations yield the following adaptive algorithm, whose performance is verified in a series of numerical tests presented in the next section.}

\begin{algorithm}
\label{algorithm}
Let $p\in\N$ and
let $0 < \theta \le 1$ be a fixed marking parameter. 
Let $\TT_h$ be a conforming initial triangulation of $\Omega$. 
Let $\varepsilon > 0$ be the tolerance for the energy error 
$\bnorm{\grad (u-u_h)}{\Lt(\Omega)}$ with $u_h= \widetilde V \phi_h$. 
Then, perform the following steps~\ref{itm:initbemmesh}  --~\ref{itm:refine}:
\renewcommand\labelenumi{\emph{(\roman{enumi})}}
\renewcommand\theenumi\labelenumi
\begin{enumerate}
\item \label{itm:initbemmesh} 
Extract the BEM triangulation $\FF^{\Gamma}_h = \TT_h|_{\Gamma}$ from~\eqref{eq:boundary-mesh}.
\item 
Extract the patch $S \subset \Omega$ of $\Gamma$ and the corresponding triangulation $\TT_h^S$ from~\eqref{eq:mesh:layer}.
\item \label{itm:solvebem} 
Compute the BEM solution $\phi_h\in \PP^0(\FF^{\Gamma}_h)$ of~\eqref{eq:discreteweakform}.
\item
Compute $J_h(g - \str{u_h}{\Gamma})$ together with its oscillations ${\rm osc}_h(T)$ of~\eqref{eq:def:osc} for all $T \in \TT_h$.
%\note{{\bf Editor:} Die Berechnung von $u_h$ muss n\"aher erkl\"art werden, da sie einen nichtlokalen Operator auswertet. Das ist ein wichtiger und teurer Schritt, der nach (iii) explizit in Algorithmus~\ref{algorithm} erkl\"art werden muss. Wie wird das in die L\"osung von~\eqref{eq:upper-bound-discrete}, \eqref{eq:mixedproblemstripdiscrete}, \eqref{eq:lower-bound2ddiscrete} integriert?}
\item
Compute the FEM solution $w_h\in \SSS^p(\TT^S_h)$ of~\eqref{eq:upper-bound-discrete} for the majorant~\eqref{maj:discrete}.
\item
Compute the error indicators
\begin{align}
\label{eq:estimator}
\eta_h(T) &= \begin{cases}
 \displaystyle
  \norm{\grad w_h}{\Lt(T)} 
  & \text{for } T\in\TT_h^S,\\
  0 & \text{for } T\in\TT_h\setminus \TT_h^S.
\end{cases}
\end{align}
\item \label{itm:stoppingcriterion}
If $\mmax(\grad w_h) 
= \sum_{T\in\TT_h^S} \eta_h(T)^2 \le \varepsilon^2$, then {\ttfamily break}.
\item \label{itm:marking} 
Otherwise, determine a set $\MM_h \subseteq \TT_h^S$ of minimal cardinality such that
\begin{align}
\label{eq:doerfler}
\theta \sum_{T\in\TT_h^S} \big[ \eta_h(T)^2 + {\rm osc}_h(T)^2 \big] 
&\le \sum_{T\in\MM_h} \big[ \eta_h(T)^2 + {\rm osc}_h(T)^2 \big].
\end{align}
\item \label{itm:refine} 
Refine (at least) all $T \in \MM_h \subseteq \TT_h$ by newest vertex bisection to obtain a new triangulation $\TT_h$.
\end{enumerate}
\end{algorithm}

\begin{remark}[evaluation of $\boldsymbol{u_h = \widetilde V \phi_h}$]
\revision{A subtle point in our approach is the computation of $J_h(g - u_h|_\Gamma)$ to solve the auxiliary FEM problem~\eqref{eq:upper-bound-discrete} and the computation of ${\rm osc}_h$ from~\eqref{eq:osc} to compute the upper bound~\eqref{eq2:maj:discrete} in Theorem~\ref{cor:continuousupperbound}. Similarly, the solution of the auxiliary FEM problem~\eqref{eq:mixedproblemstripdiscrete} and the computation of the lower bound~\eqref{min:discrete} require to compute $\scp{g - \str{u_h}{\Gamma}}{\ntr{\ssigma_h}{\Gamma}}{\Lt(\Gamma)}$ for the basis functions $\ssigma_h \in \RTqg (\TT^S_h)$ (and analogous comments apply for the alternative lower bound~\eqref{eq:lower-bound2ddiscrete}--\eqref{eq2:lower-bound2ddiscrete} from Corollary~\ref{cor:continuouslowerboundtwo}). All of this is subtle, since $u_h = \widetilde V \phi_h$ is not a discrete function (but data sparse, since $\phi_h$ is discrete). At these points, our implementation follows the approach of~\cite{hilbert}, which can briefly be sketched as follows:} 

\revision{{\rm(i)} Note that, due to the mapping properties of the single-layer potential, $u_h = \widetilde V\phi_h$ is continuous (since $\phi_h \in \mathsf{L}^\infty(\Gamma)$), and that, at least for affine boundaries in 2D and piecewise polynomial $\phi_h$, closed formulae for point evaluations $u_h(x) = \widetilde V\phi_h(x)$ at arbitrary $x \in \R^2$ are known.}

\revision{{\rm(ii)} To compute $J_h(g - u_h|_\Gamma) \in \SSS^p(\FF^{\Gamma}_h)$, we approximate $g - u_h|_\Gamma \approx q \in \PP^{p'}(\FF^{\Gamma}_h)$ by a $\FF^{\Gamma}_h$-piecewise interpolation polynomial $q$ of degree $p' > p$. Replacing $g - u_h|_\Gamma \approx q$ in~\eqref{Jhpdef} so that all arising integrals can be computed exactly (by quadrature), we approximate $J_h(g - u_h|_\Gamma) \approx J_h q$. Moreover, we approximate the local contributions of ${\rm osc}_h$ from~\eqref{eq:def:osc} via $\norm{\grad_{\Gamma} ((1 - J_h) (g \!-\! \str{u_h}{\Gamma}))}{\Lt(F)}^2 \approx \norm{\grad_{\Gamma} (q - J_h q)}{\Lt(F)}^2$, where again the right-hand side can be computed exactly by means of quadrature which only relies on point evaluations of $q(x)$ (and hence $g - u_h|_\Gamma$).}

\revision{{\rm(iii)} It is an empirical observation in~\cite{hilbert} that $p' := p + 1$ is sufficiently accurate. If $g - u_h|_\Gamma$ is smooth, one can even show that the quadrature error is of higher order. Moreover, one can optimize the interpolation nodes (per element) and the quadrature nodes (to compute the approximate integrals in~\eqref{Jhpdef} and~\eqref{eq:osc}) to minimize the number of (expensive) point evaluations of $u_h = \widetilde V \phi_h$.}

\revision{{\rm(iv)} Similar ideas must also be used for any BEM error estimator which involves the residual (see, e.g.,~\cite{hilbert}). By means of matrix compression techniques like planel clustering or $\mathcal{H}$-matrices (see, e.g.,~\cite{MR3445676} and the references therein), one can even lower the cost of the point evaluations of $u_h = \widetilde V \phi_h$. However, this is not exploited by our current implementation.}

\revision{{\rm(v)} Analogous ideas are used to compute the lower error bounds of Theorem~\ref{cor:continuousupperbound} resp.\ Corollary~\ref{cor:continuouslowerboundtwo}.}\qed
\end{remark}

\begin{remark}[comments on the minorant]
\label{rem:placementofminorant}
\revision{{\rm(i)} 
We stress that a reliable adaptive algorithm requires 
only a computable upper error bound.
From that perspective, the minorant should be viewed as an option for eventual practical applications, 
which might be computed in one final step,
i.e., after having achieved
the error tolerance by the stopping criterion of Algorithm~\ref{algorithm}{\rm(vii)}.
If desired, as in~\cite{MAsebastian}, the minorant can then be improved by solving~\eqref{eq:mixedproblemstripdiscrete} 
or~\eqref{eq:lower-bound2ddiscrete} adaptively (by only a few post-processing steps), 
while $u_h$ is fixed and the additional mesh refinement of $\TT_h^S$ (resp. $\TT_h)$ 
is only steered by the minorant.
%DIRK: From that point of view, 
%the minorant should be viewed as an option for eventual practical applications, and it usually suffices 
%to compute the minorant in only one final step, i.e., after having achieved
%the desired error tolerance by the stopping criterion of Algorithm~\ref{algorithm} {\rm(vii)}.
%As in~\cite{MAsebastian}, the minorant can accurately be computed by solving~\eqref{eq:mixedproblemstripdiscrete} 
%or~\eqref{eq:lower-bound2ddiscrete} adaptively, where the additional mesh refinement of $\TT_h^S$ (resp. $\TT_h)$ 
%is only steered by the minorant.
%DANIEL: One computes
%the minorant (if desired: again adaptively) for the final and fixed approximate solution $u_h$. 
%Like in \cite{MAsebastian}, our numerical experiments revealed that --- for fixed
%$u_h$ --- only a few adaptive\footnote{solving \eqref{eq:mixedproblemstripdiscrete} 
%or \eqref{eq:lower-bound2ddiscrete} adaptively by steering the mesh $\TT_h^S$ (resp. $\TT_h)$ from now on by the minorant}
%(post-processing) steps on $\TT_h^S$ have improved the accuracy of the (final)
%minorant considerably.
}

\revision{{\rm(ii)} 
Another option is to include the computation of the lower bound into each step of the
algorithm to provide guaranteed intervals for the error: For instance, computing the FEM solution 
$(\ttau_h,\phi_h)\in
\RTzg(\TT_h^S) \times \PP^{0}(\TT^S_h)$ of~\eqref{eq:mixedproblemstripdiscrete}, we can also assemble 
the discrete minorant 
\begin{align}
\label{eq:experiments:minorant}
\mmin_h(\ttau_h;\str{u_h}{\Gamma},g) 
%&= 2 \, \int_{\Gamma} (g - \str{u_h}{\Gamma}) ( \normal\cdot\ttau_h|_{\Gamma}) \d{x}
&= 2 \, \scp{g - \str{u_h}{\Gamma}}{\normal\cdot\ttau_h|_{\Gamma}}{\Lt(\Gamma)}
- \ \norm{\ttau_h}{\Lt(S)}^2
\end{align}
from Theorem~\ref{cor:continuousupperbound}.
Then again, at least in terms of Algorithm~\ref{algorithm}, the minorant is computed on a mesh (and in particular on the
boundary layer $S$) steered by the majorant alone; see Figure \ref{fig:example1:mesh}. This procedure already leads
to a satisfying minorant (see all experiments in Section~\ref{section:numerics}), but obviously not to the most accurate minorant possible.}

\revision{{\rm(iii)} 
The latter approach can be improved in several ways:
%by several actions:
First, one can solve~\eqref{eq:upper-bound-discrete} and~\eqref{eq:mixedproblemstripdiscrete} (resp.~\eqref{eq:lower-bound2ddiscrete}) by adaptive FEM on separate (generically) 
different boundary layers.
%; see ~\cite{MAsebastian}.
Second, one can consider higher-order elements for the auxiliary FEM problems.
% increase the approximation order for both the finite elements and the projector $J_h$. 
%Apart from increasing degrees, 
Finally, another option is to include the 
%natural 
local contributions of \eqref{eq:experiments:minorant},
\begin{align}
\label{eq:experiments:minorant_local}
\nu_{h}(T):=2 \, \scp{g - \str{u_h}{\Gamma}}{\normal\cdot\ttau_h|_{\Gamma}}{\Lt(\Gamma\cap T)}
- \ \norm{\ttau_h}{\Lt(T)}^2
\end{align}
for all $T\in\TT_{h}^S$,
into the marking procedure. Figure \ref{fig:example1:errmajmin_fullmin} visualizes some results, where (instead of the marking strategy in Algorithm~\ref{algorithm}{\rm(viii)}) the mesh is now steered by the size of the confidence 
interval of the error, i.e., $\eta_h(T) + {\rm osc}_h(T) - \nu_h(T)$, and the minorant improves.
%One from our marking strategy in Algorithm~\ref{algorithm} (viii) deviating numerical experiment
%is visualized in , where 
Overall, all these approaches are
%both 
%this is 
computationally more costly and %does 
only make sense, if sharp confidence intervals of the error are needed during the full runtime of the adaptive algorithm. In our understanding of reliable algorithms, there are not many practical situations in which the minorant becomes relevant before the final solution $u_h$ has been computed and fixed.
\qed}
\end{remark}

\newcommand{\example}[1]{%
  \vskip2mm
  \refstepcounter{subsection}
  {\bf Example~\thesubsection~(#1).}\ 
}

%%%%%%%%%%%%%%%%%%%%%%%%%%%%%%%%%%%%%%%%%%%%%%%%%%%%%%%%%%%%%%%%%%%%%%%%%%%%%%%%%%%
%%%%%%%%%%%%%%%%%%%%%%%%%%%%%%%%%%%%%%%%%%%%%%%%%%%%%%%%%%%%%%%%%%%%%%%%%%%%%%%%%%%
\section{Numerical experiments}
\label{section:numerics}
%%%%%%%%%%%%%%%%%%%%%%%%%%%%%%%%%%%%%%%%%%%%%%%%%%%%%%%%%%%%%%%%%%%%%%%%%%%%%%%%%%%%
%%%%%%%%%%%%%%%%%%%%%%%%%%%%%%%%%%%%%%%%%%%%%%%%%%%%%%%%%%%%%%%%%%%%%%%%%%%%%%%%%%%%

\noindent
This section reports on some 2D numerical experiments to underline the accuracy of the introduced error estimates and the performance of the proposed adaptive strategy from Algorithm~\ref{algorithm}. All computations are done in {\scshape Matlab}, where we build on the toolbox \textsc{Hilbert} from \cite{hilbert} for the lowest-order BEM, on~\cite{p1afem} 
for $\PP^1$-FEM ($p=1$ in \eqref{Jhpdef})
resp.\ \cite{p2afem} for $\PP^2$-FEM ($p=2$ in \eqref{Jhpdef}), 
and on~\cite{ccb} for the lowest-order $\mathsf{RT}$-FEM. Throughout, we consider Algorithm~\ref{algorithm} for uniform mesh-refinement (i.e., $\theta = 1$) as well as for adaptive mesh-refinement (i.e., $0 < \theta < 1$).
%\note{{\bf Editor:} Der Code sollte \"offentlich zug\"anglich sein, z.B. Github.}

%%%%%%%%%%%%%%%%%%%%%%%%%%%%%%%%%%%%%%%%%%%%%%%%%%%%%%%%%%%%%%%%%%%%%%%%%%%%%%%%%%%%
%%%%%%%%%%%%%%%%%%%%%%%%%%%%%%%%%%%%%%%%%%%%%%%%%%%%%%%%%%%%%%%%%%%%%%%%%%%%%%%%%%%%
%%%%%%%%%%%%%%%%%%%%%%%%%%%%%%%%%%%%%%%%%%%%%%%%%%%%%%%%%%%%%%%%%%%%%%%%%%%%%%%%%%%%
% EXAMPLE 1
%
\begin{figure}[t]
\begin{subfigure}[c]{.48\textwidth}
\includegraphics[trim={4.1cm 1.1cm 3.4cm 1cm},clip,width=\textwidth]{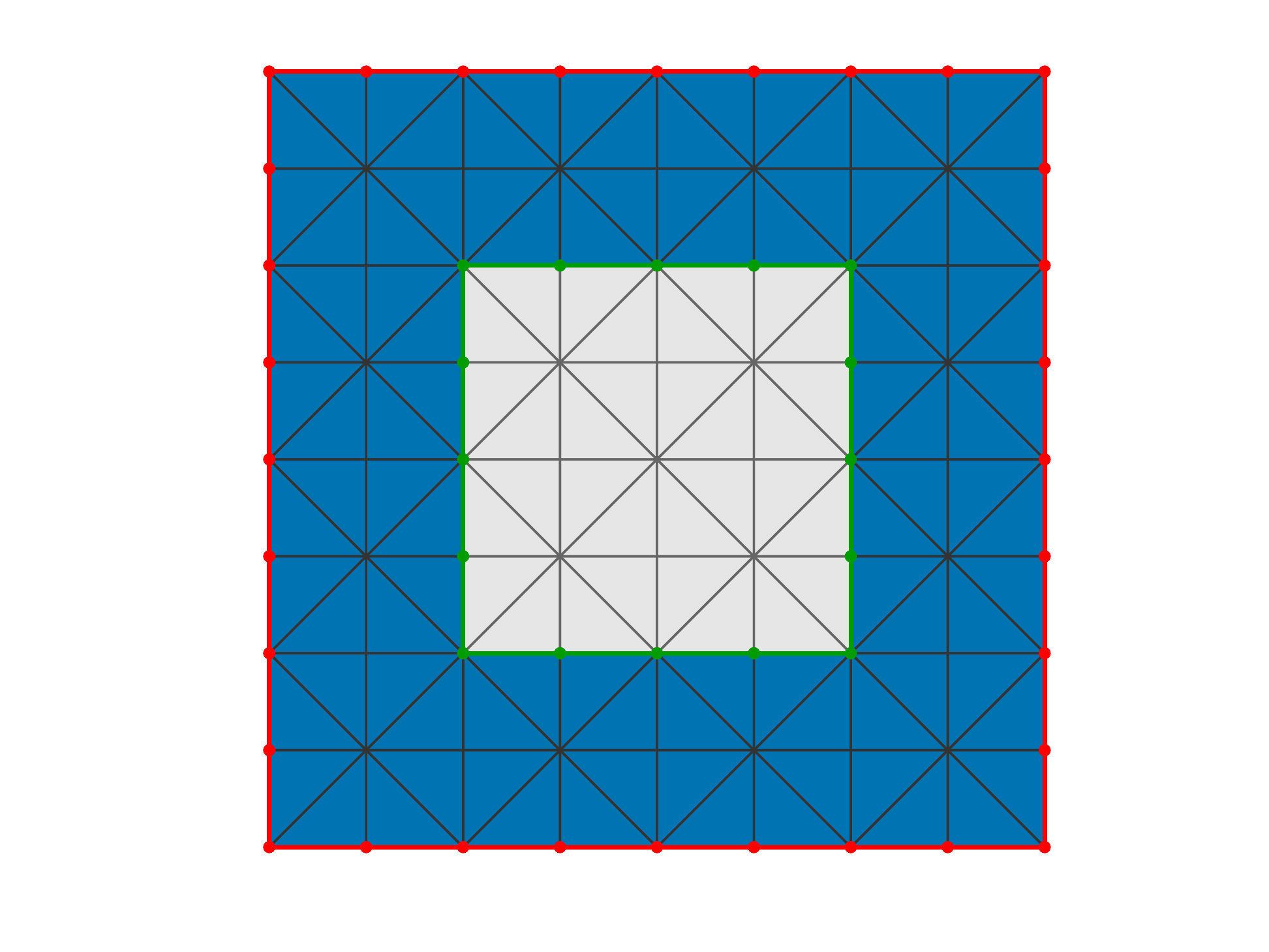}
\caption{$\#\TT_h^S = 72$, $\#\FF_h^\Gamma = 32$, $\ell = 0$}
\end{subfigure}%
\hfill
\begin{subfigure}[c]{.48\textwidth}
\includegraphics[trim={4.1cm 1.1cm 3.4cm 1cm},clip,width=\textwidth]{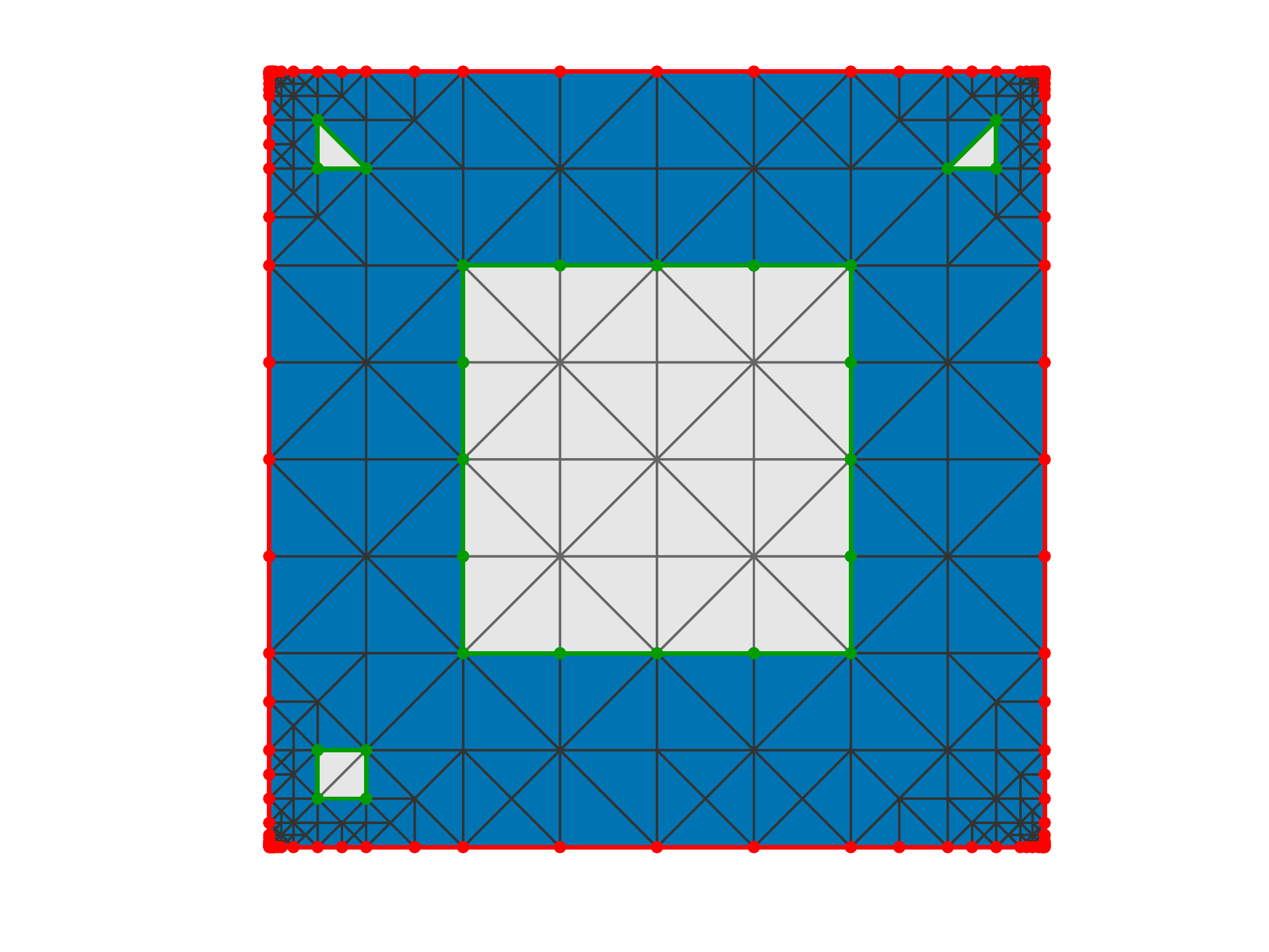}\hfill
\caption{$\#\TT_{h}^S = 314$, $\#\FF_{h}^\Gamma = 119$, $\ell = 17$}
\end{subfigure}
\\
\begin{subfigure}[c]{.48\textwidth}
\includegraphics[trim={4.1cm 1.1cm 3.4cm 1cm},clip,width=\textwidth]{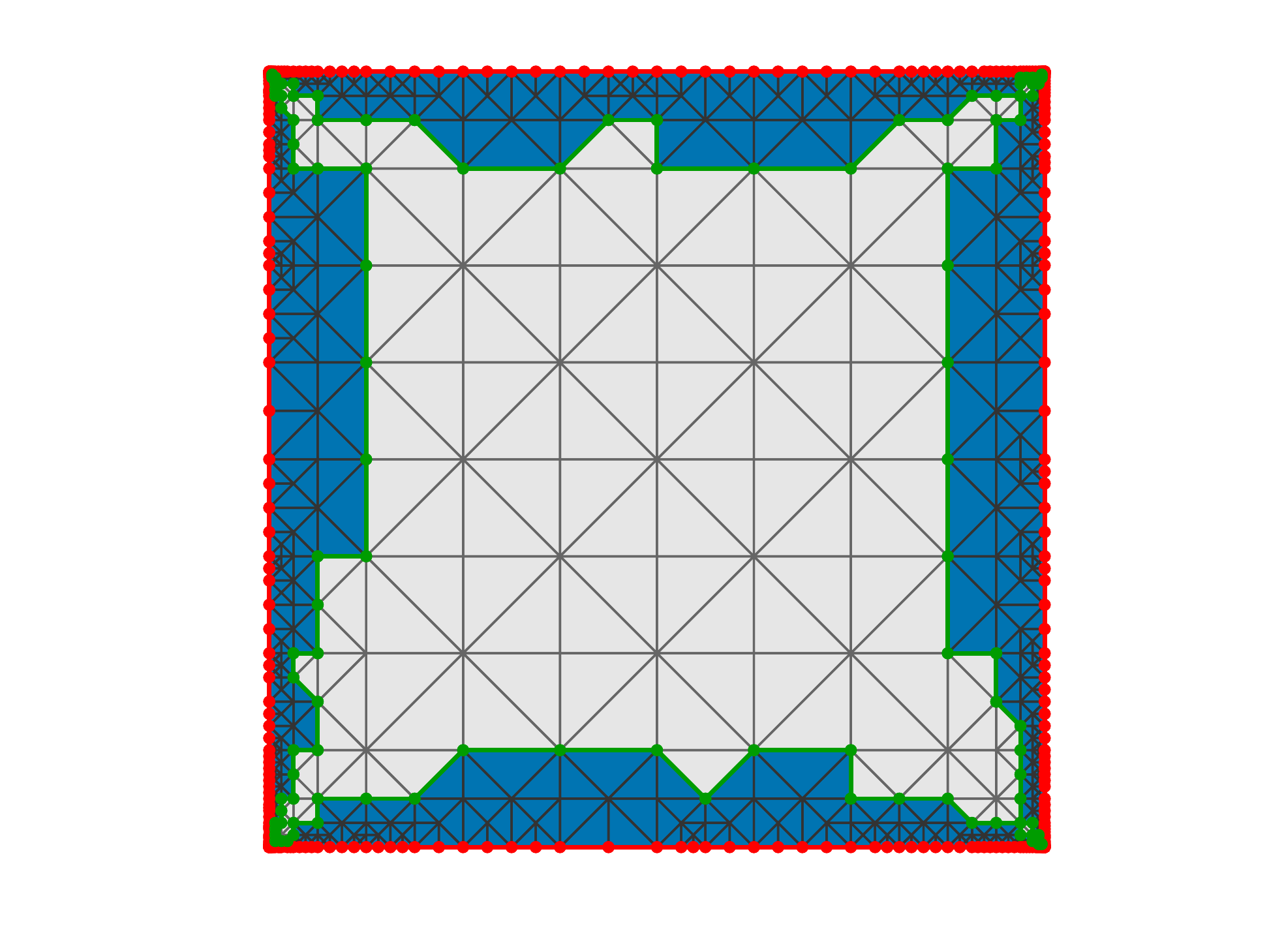}\hfill
\caption{$\#\TT_{h}^S = 923$, $\#\FF_{h}^\Gamma =352$, $\ell = 27$}
\end{subfigure}
\hfill
\begin{subfigure}[c]{.48\textwidth}
\includegraphics[trim={4.1cm 1.1cm 3.4cm 1cm},clip,width=\textwidth]{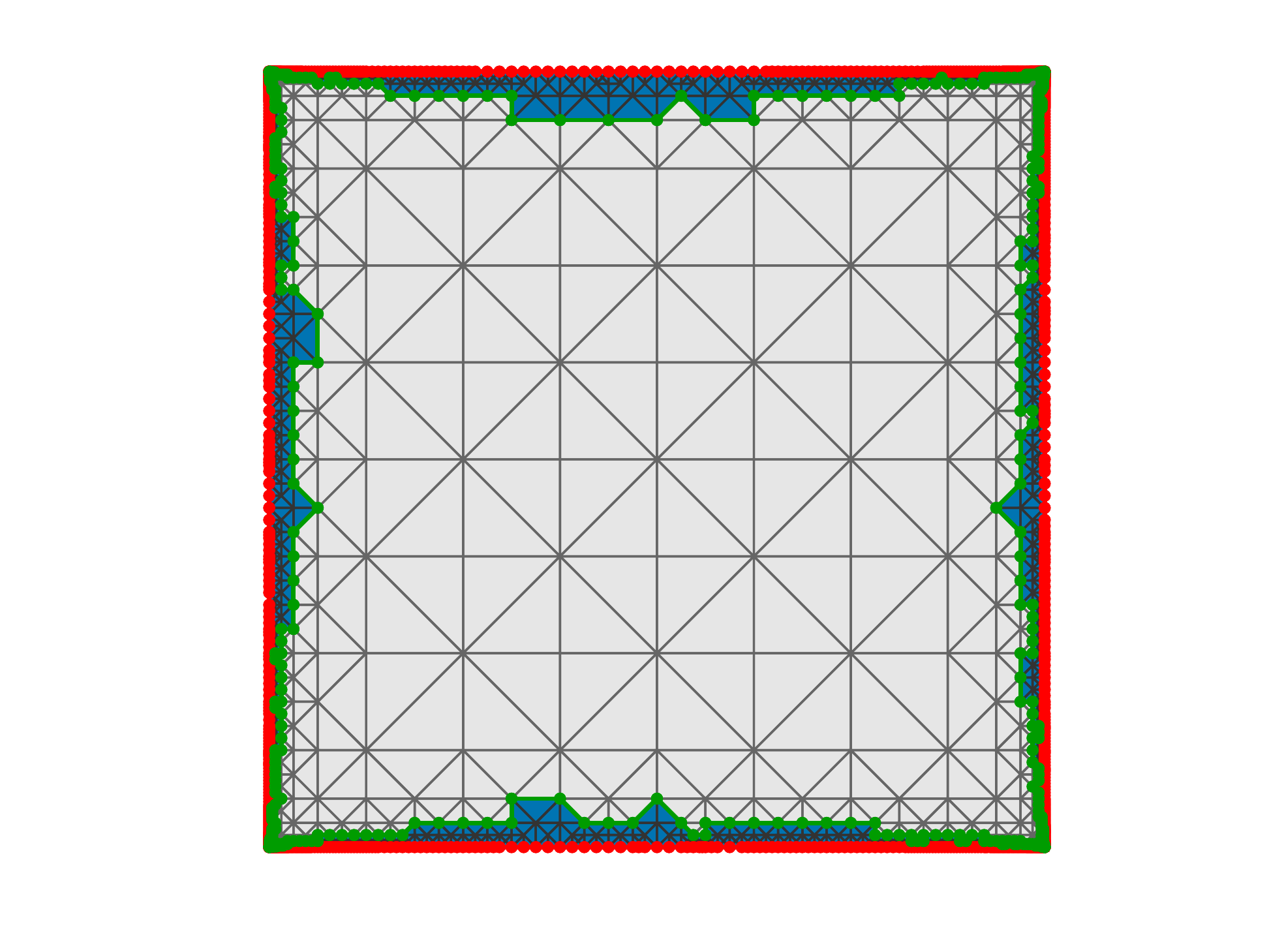}\hfill
\caption{$\#\TT_{h}^S = 3176$, $\#\FF_{h}^\Gamma = 1148$, $\ell = 38$}
\end{subfigure}
\caption{Adaptively generated meshes in Example~\ref{example1} for $p=1$ and $\theta = 0.6$. We indicate the boundary layer {\bf\color{tuwBlue}$\boldsymbol{S}$~(blue)}, the boundary {\bf\cred$\boldsymbol{\Gamma}$~(red)}, and the interior boundary {\bf\color{green!50!black}$\boldsymbol{\Gamma^c = \partial S \setminus \Gamma}$~\bf(green)}. The {triangles} $\color{tuwBlue}\boldsymbol{T \in \TT_h^S}\color{black} \subset \TT_h$ are indicated in {\bf\color{tuwBlue}blue}. The {triangles} ${\color{gray}\boldsymbol{T \in \TT_h \backslash \TT_h^S}}$ are indicated in {\bf\color{gray}gray}.}
\label{fig:example1:mesh}
\end{figure}
\begin{figure}[t]
\centering
\includegraphics[trim={.7cm 0 0 0},clip,width=.49\textwidth]{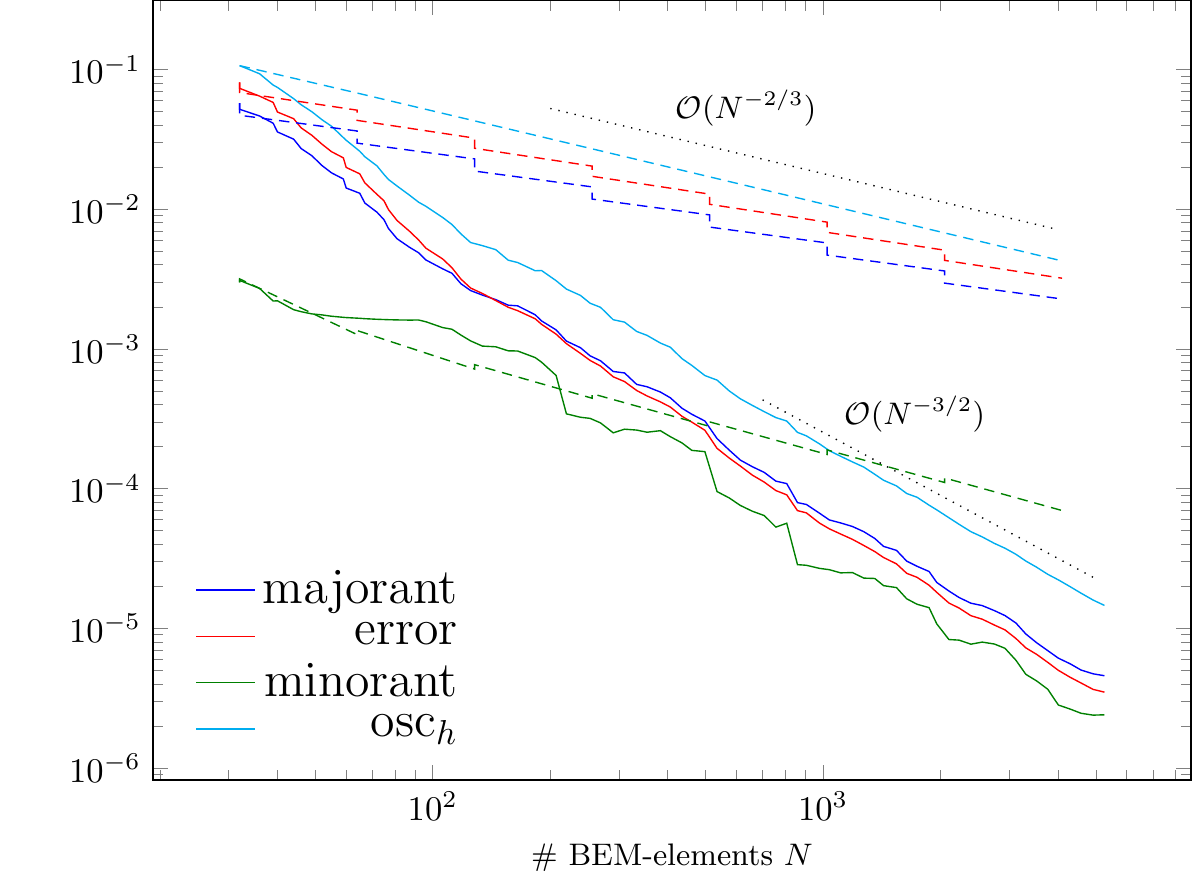}
\includegraphics[trim={.7cm 0 0 0},clip,width=.49\textwidth]{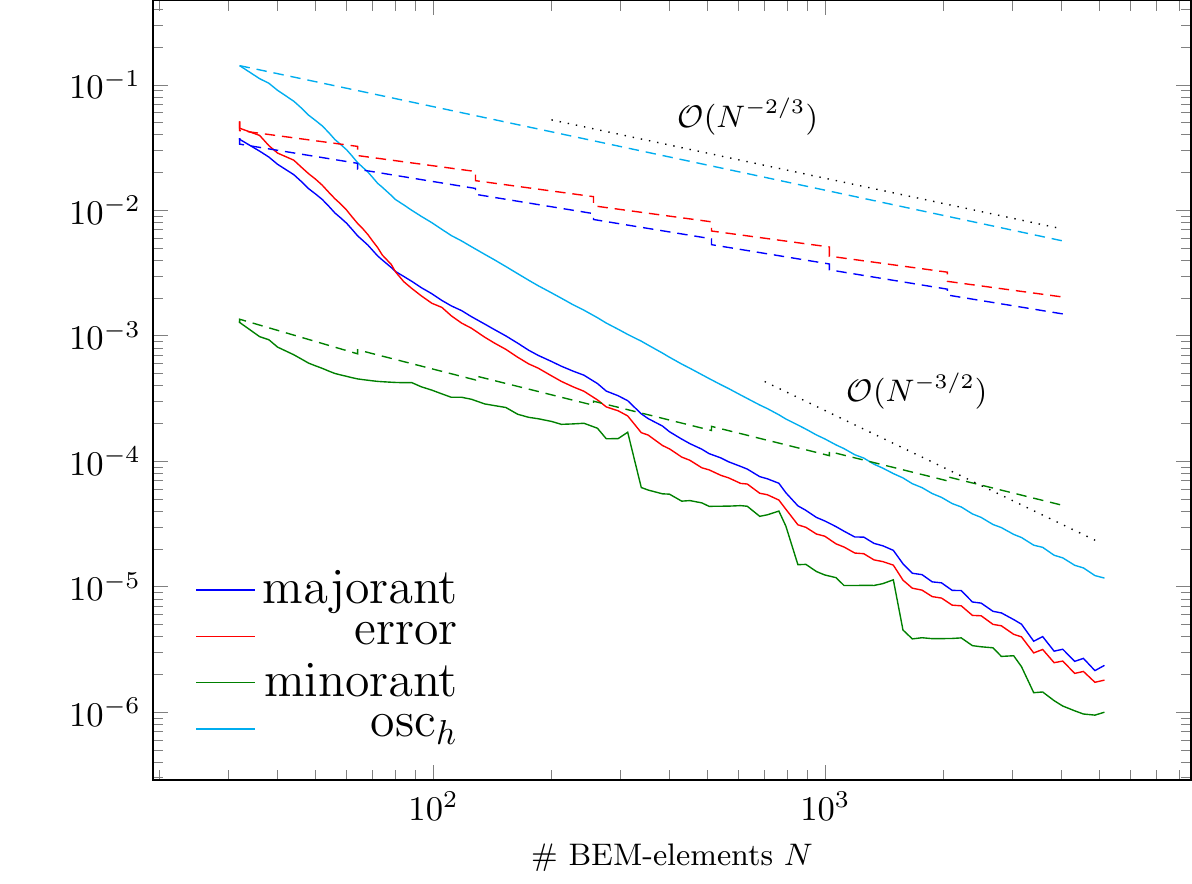}
\caption{Comparison of adaptive mesh-refinement with $\theta = 0.4$ (solid) vs.\ uniform mesh-refinement (dashed) in Example~\ref{example1}. The majorant is computed by $\PP^1$-FEM (left) and $\PP^2$-FEM (right). We compare the potential error $\norm{\nabla(u-u_h)}{\Lt(\Omega)}$, the majorant $\norm{\nabla w_h}{\Lt(\Omega)}$ from~\eqref{maj:discrete}, the data oscillations ${\rm osc}_h$ from~\eqref{eq:osc}, and the minorant $\mmin(\ttau_h)^{1/2}$ from~\eqref{eq:experiments:minorant}.}
\label{fig:example1:error-estimator}
\end{figure}
\begin{figure}[t]
\centering
\includegraphics[trim={.7cm 0 0 0},clip,width=.49\textwidth]{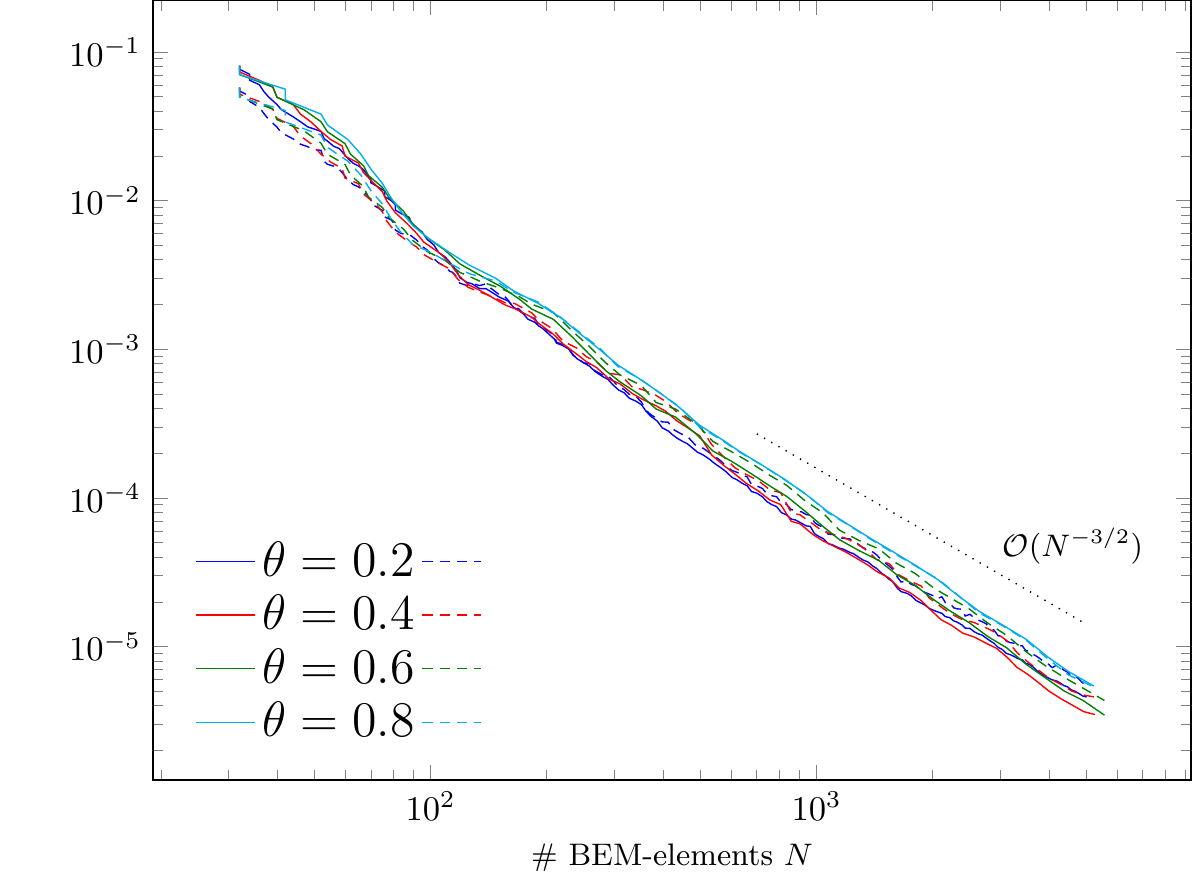}
\includegraphics[trim={.7cm 0 0 0},clip,width=.49\textwidth]{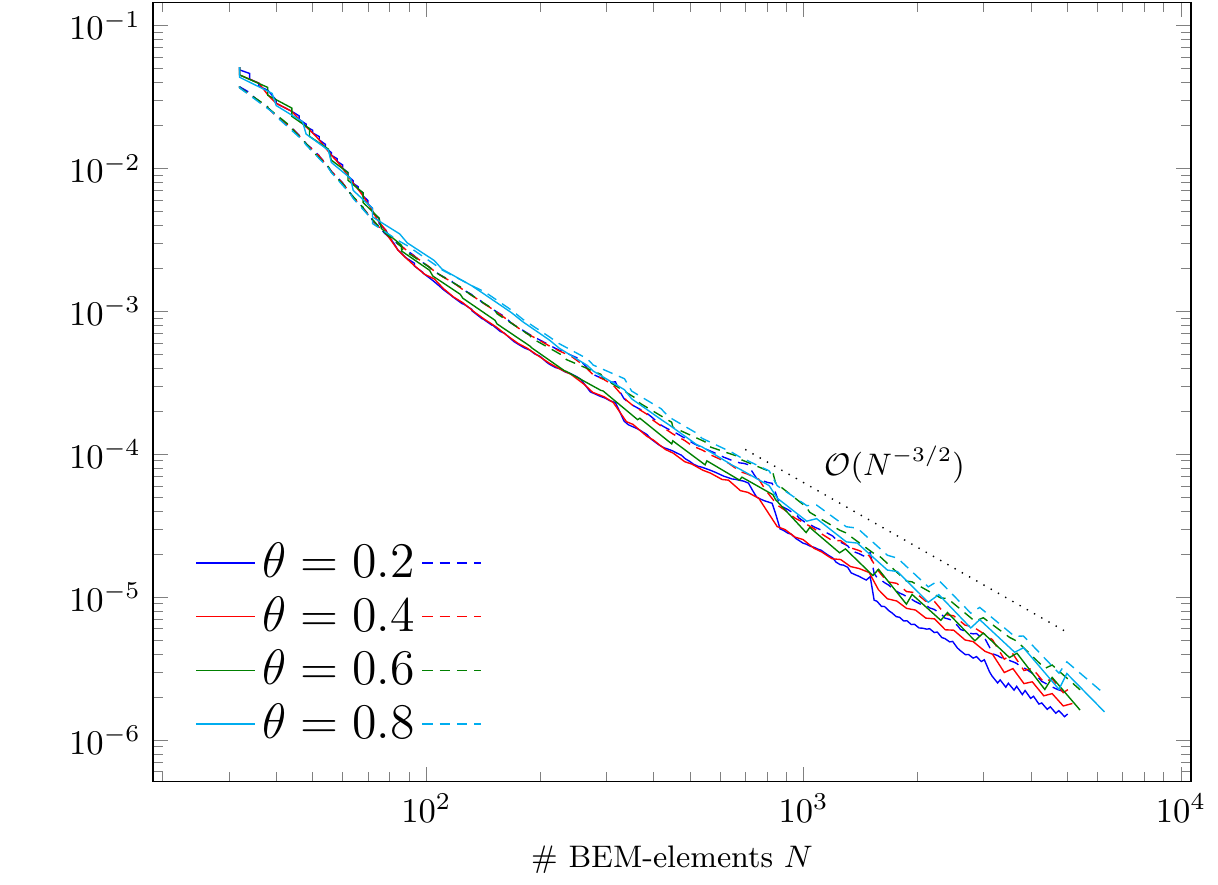}
\caption{Influence of the marking parameter $\theta \in \{0.2, 0.4, 0.6, 0.8\}$ on adaptive mesh-refinement in Example~\ref{example1}. The majorant is computed by $\PP^1$-FEM (left) and $\PP^2$-FEM (right). We compare the potential error (solid) $\norm{\nabla(u-u_h)}{\Lt(\Omega)}$ as well as the majorant (dashed) $\norm{\nabla w_h}{\Lt(\Omega)}$ from~\eqref{maj:discrete}.}
\label{fig:example1:theta}
\end{figure}

\begin{table}[th]
\begin{tabular}{|c||c|c|c|c|c|c|c|}
\hline
$\ell$ & $\# \FF^\Gamma_h$ & $\frac{\# \TT^S_h}{\# \FF^\Gamma_h}$ & \footnotesize$\norm{\nabla (u-u_h)}{\Lt(\Omega)}$& \footnotesize$\norm{\nabla w_h}{\Lt(S)}$ & \footnotesize$\frac{\norm{\nabla w_h}{\Lt(S)}}{\norm{\nabla(u-u_h)}{\Lt(\Omega)}}$ & $\footnotesize\frac{\norm{\grad w_h}{\Lt(S)}}{\mmin(\ttau_h)^{1/2}}$ \\ \hline\hline
$0$ & $32$ & $2.25$ & $8.01 e-2$ & $5.75 e-2$ & $0.71$ & $19.16$ \\ \hline
$1$ & $64$ & $2.63$ & $5.12 e-2$ & $3.63 e-2$ & $0.71$ & $28.67$ \\ \hline
$2$ & $128$ & $2.81$ & $3.23 e-2$ & $2.30 e-2$ & $0.71$ & $31.96$ \\ \hline
$3$ & $256$ & $2.91$ & $2.03 e-2$ & $1.45 e-2$ & $0.71$ & $32.56$ \\ \hline
$4$ & $512$ & $2.95$ & $1.28 e-2$ & $9.11 e-3$ & $0.71$ & $32.66$ \\ \hline
$5$ & $1024$ & $2.98$ & $8.08 e-3$ & $5.74 e-3$ & $0.71$ & $32.67$ \\ \hline
$6$ & $2048$ & $2.99$ & $5.09 e-3$ & $3.62 e-3$ & $0.71$ & $32.67$ \\ \hline
$7$ & $4096$ & $3.00$ & $3.21 e-3$ & $2.28 e-3$ & $0.71$ & $32.67$ \\ \hline
\end{tabular}
\caption{Uniform mesh-refinement in Example~\ref{example1}. We focus on the degrees of freedom, the potential error $\norm{\nabla(u-u_h)}{\Lt(\Omega)}$, the accuracy of the $\PP^1$-FEM majorant $\norm{\nabla w_h}{\Lt(S)}$ from~\eqref{maj:discrete}, 
and the quotient of the majorant and minorant.}
\label{table1}
\end{table}

\begin{table}[th]
\begin{tabular}{|c||c|c|c|c|c|c|c|}
\hline
$\ell$ & $\# \FF^\Gamma_h$ & $\frac{\# \TT^S_h}{\# \FF^\Gamma_h}$ & \footnotesize${\rm dof}(\TT_h^S)$ & \footnotesize$\norm{\nabla (u-u_h)}{\Lt(\Omega)}$ & \footnotesize$\norm{\nabla w_h}{\Lt(S)}$ & \footnotesize$\frac{\norm{\nabla w_h}{\Lt(S)}}{\norm{\nabla(u-u_h)}{\Lt(\Omega)}}$ & $\footnotesize\frac{\norm{\grad w_h}{\Lt(S)}}{\mmin(\ttau_h)^{1/2}}$ \\ \hline\hline
$0$ & $32$ & $2.25$ & $15$ & $8.01 e-2$ & $5.75 e-2$ & $0.71$ & $19.16$ \\ \hline
$4$ & $40$ & $2.33$ & $28$ & $4.97 e-2$ & $3.58 e-2$ & $0.72$ & $16.17$ \\ \hline
$10$ & $59$ & $2.44$ & $60$ & $2.33 e-2$ & $1.65 e-2$ & $0.71$ & $9.78$ \\ \hline
$16$ & $77$ & $2.66$ & $103$ & $9.95 e-3$ & $7.29 e-3$ & $0.73$ & $4.50$ \\ \hline
$22$ & $112$ & $2.60$ & $148$ & $3.81 e-3$ & $3.48 e-3$ & $0.91$ & $2.51$ \\ \hline
$28$ & $165$ & $2.82$ & $234$ & $1.88 e-3$ & $2.03 e-3$ & $1.08$ & $2.11$ \\ \hline
$34$ & $253$ & $2.83$ & $343$ & $8.27 e-4$ & $8.92 e-4$ & $1.08$ & $2.80$ \\ \hline
$40$ & $383$ & $2.81$ & $512$ & $4.18 e-4$ & $4.91 e-4$ & $1.18$ & $1.89$ \\ \hline
$46$ & $575$ & $2.70$ & $707$ & $1.66 e-4$ & $1.89 e-4$ & $1.14$ & $2.20$ \\ \hline
$52$ & $860$ & $2.63$ & $978$ & $6.96 e-5$ & $7.94 e-5$ & $1.14$ & $2.78$ \\ \hline
$58$ & $1072$ & $2.61$ & $1389$ & $3.92 e-5$ & $4.92 e-5$ & $1.25$ & $2.15$ \\ \hline
$64$ & $1869$ & $2.61$ & $2008$ & $2.04 e-5$ & $2.55 e-5$ & $1.25$ & $1.82$ \\ \hline
$70$ & $2748$ & $2.58$ & $2803$ & $1.06 e-5$ & $1.34 e-5$ & $1.27$ & $1.73$ \\ \hline
$76$ & $4007$ & $2.55$ & $3976$ & $5.00 e-6$ & $6.12 e-6$ & $1.22$ & $2.16$ \\ \hline
$80$ & $5259$ & $2.53$ & $5077$ & $3.50 e-6$ & $4.58 e-6$ & $1.31$ & $1.90$ \\ \hline
\end{tabular}
\caption{Adaptive mesh-refinement with $\theta = 0.4$ in Example~\ref{example1}. We focus on the degrees of freedom, the potential error $\norm{\nabla(u-u_h)}{\Lt(\Omega)}$, the accuracy of the $\PP^1$-FEM majorant $\norm{\nabla w_h}{\Lt(\Omega)}$ from~\eqref{maj:discrete}, and the quotient of the majorant and minorant.}
\label{table2}
\end{table}
\begin{figure}[t]
\centering
\includegraphics[trim={.7cm 0 0 0},clip,width=.49\textwidth]{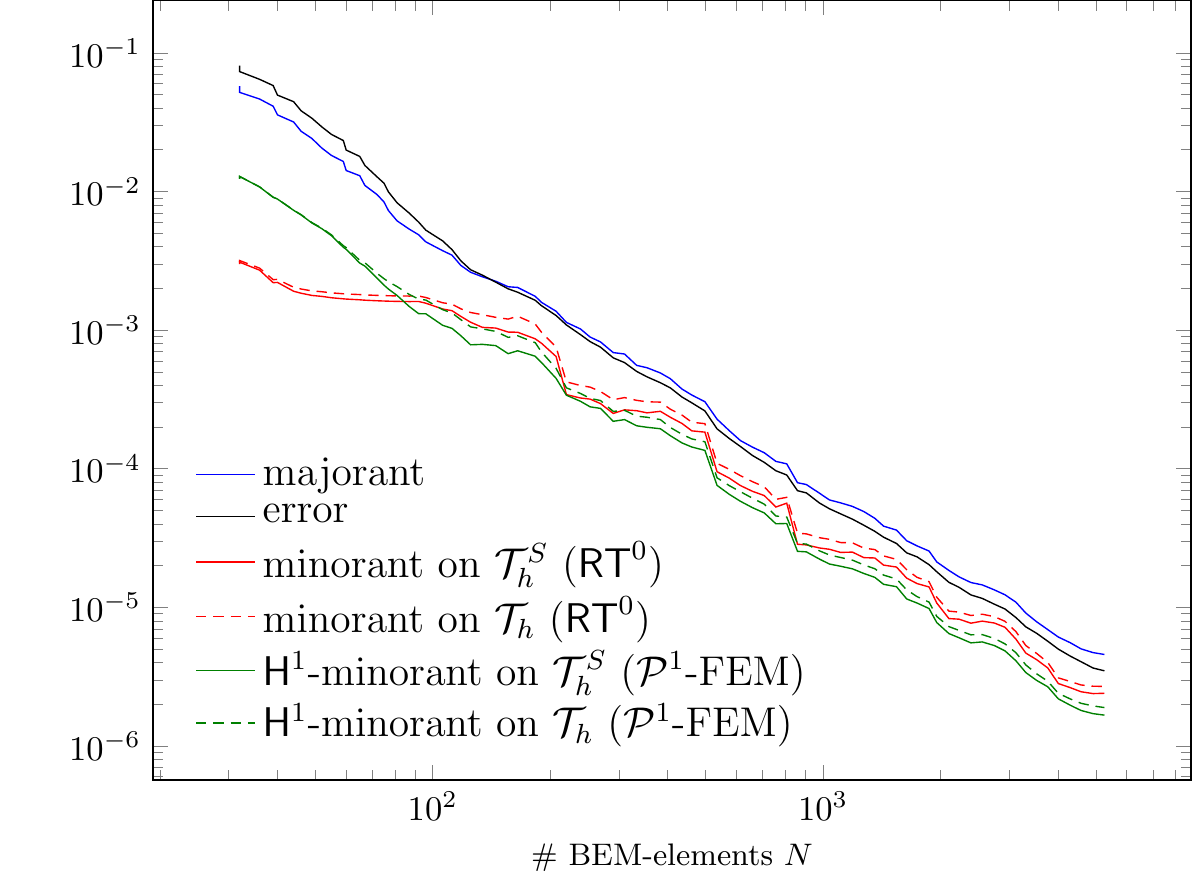}
\includegraphics[trim={.7cm 0 0 0},clip,width=.49\textwidth]{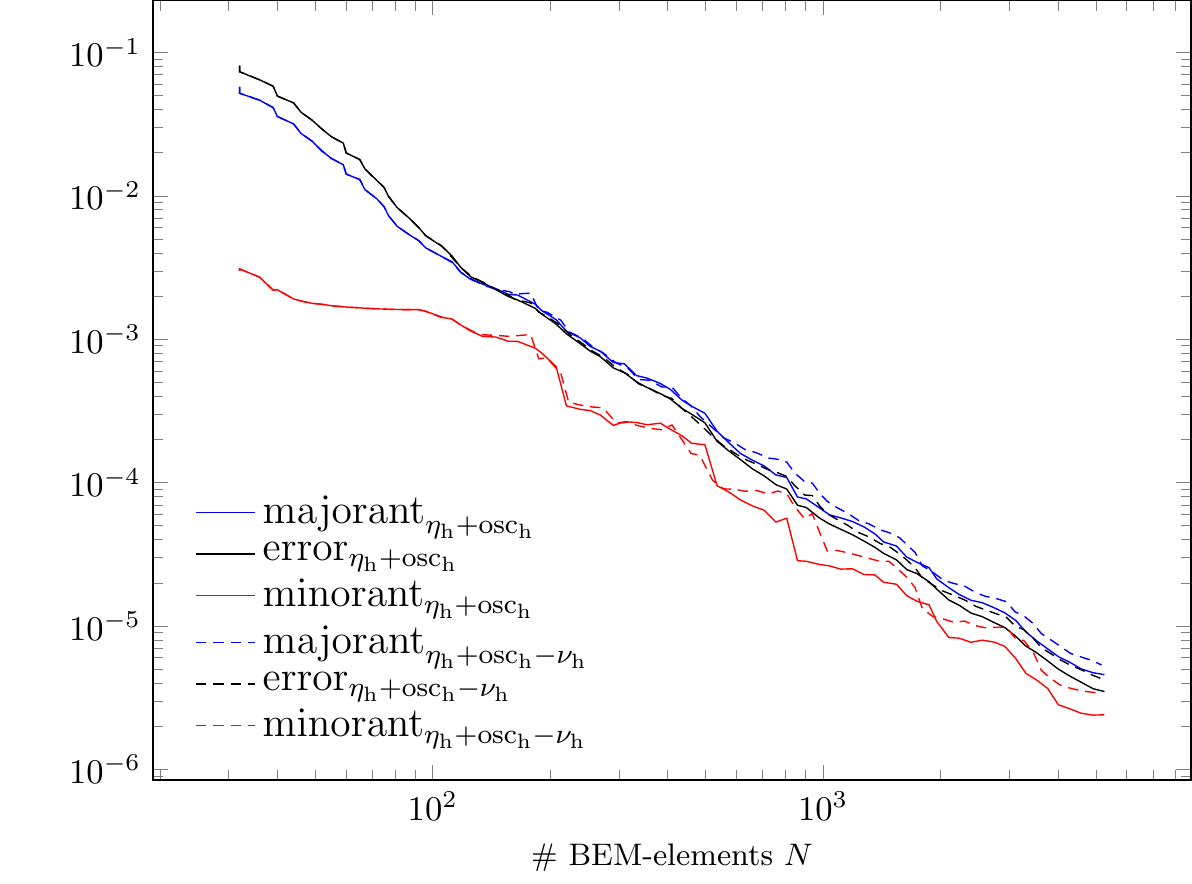}
\caption{\revision{\emph{Left:} Comparison of two versions of the minorant by either solving \eqref{eq:mixedproblemstripdiscrete} with $\RTz$-elements or \eqref{eq:lower-bound2ddiscrete} with $\mathcal{P}^1$-elements on both $\TT_{h}^{S}$ and $\TT_{h}$ with respect to the adaptive mesh-refinement with $\theta = 0.4$ 
in Example~\ref{example1}. We observe that solving on full $\TT_h$ instead of the boundary layer $\TT_h^S$ leads only to a marginal improvement of the minorant.
\emph{Right:} We compute two versions of the triple (majorant, error, minorant) 
in Example~\ref{example1} with $\theta = 0.4$.
First, we repeat the computations obtained by Algorithm \ref{algorithm} (solid lines).
In the second case, we add the local contributions $-\nu_{h}$ of the minorant
from~\eqref{eq:experiments:minorant_local}
to the error indicator in~\eqref{eq:doerfler}, i.e., 
the minorant is now part of the adaptive mesh-refinement strategy (dashed).}}
\label{fig:example1:errmajmin_fullmin}
\end{figure}

%%%%%%%%%%%%%%%%%%%%%%%%%%%%%%%%%%%%%%%%%%%%%%%%%%%%%%%%%%%%%%%%%%%%%%%%%%%%%%%%%%%%
\example{Smooth potential in square domain}
\label{example1}
%%%%%%%%%%%%%%%%%%%%%%%%%%%%%%%%%%%%%%%%%%%%%%%%%%%%%%%%%%%%%%%%%%%%%%%%%%%%%%%%%%%%
%
We consider problem~\eqref{eq:strongform} with prescribed exact solution
\begin{align}\label{eq:example1}
 u(x) = \cosh(x_1) \, \cos(x_2)
 \quad \text{for all } x \in \Omega := (0, 1/2)^2
\end{align}
on the square domain $\Omega$ with diameter $\diam(\Omega) = \sqrt{1/2}$. 
We start Algorithm~\ref{algorithm} with an initial triangulation $\TT_h$ of $\Omega$ into $\#\TT_h = 128$ right triangles.

Even though $u$ as well as its Dirichlet data $g = \str{u}{\Gamma}$ are smooth, we note that the sought integral density $\phi\in\H^{-1/2}(\Gamma)$ of the indirect formulation~\eqref{eq:weaksing} has no physical meaning and usually lacks smoothness (by inheriting the generic singularities from the interior as well as the exterior domain problem). Consequently, one may expect that uniform mesh-refinement (on the boundary) will not reveal the optimal convergence behavior $\norm{\phi-\phi_h}{\H^{-1/2}(\Gamma)} = \OO(h^{3/2}) = \OO(N^{-3/2})$, where $N = \#\FF_h^\Gamma$ is the number of elements of a uniform mesh $\FF_h^\Gamma$ of $\Gamma$ and $3/2$ is the best possible convergence rate for a piecewise constant approximation $\phi \approx \phi_h\in \PP^0(\FF^{\Gamma}_h)$. 

The initial meshes and some adaptively generated meshes are visualized in Figure~\ref{fig:example1:mesh}. 
Figure~\ref{fig:example1:error-estimator} shows the resulting potential error and the computed minorant~\eqref{min:discrete} and majorant~\eqref{maj:discrete}, as well as the corresponding data oscillations~\eqref{eq:osc} for $p = 1$ resp.~$p=2$. Here, the potential error $\norm{\nabla ( u - u_h )}{\Lt(\Omega)} \approx \norm{\nabla I_h( u - u_h )}{\Lt(\Omega)}$ is computed by numerical quadrature. More precisely, we employ  the $\PP^2$-nodal interpolant $I_h : C(\overline\Omega) \to \SSS^2(\TT_h^{\rm unif})$ on a (three times) uniform refinement $\TT_h^{\rm unif}$ of the finest adaptive mesh $\TT_h$. We stress that the plot neglects the non-accessible constant $\Cosc$ from~\eqref{eq:osc}. The results for $p = 1$ and $p = 2$ are similar. For uniform mesh-refinement, we obtain the expected reduced order of convergence. For adaptive mesh-refinement, we regain the optimal order of convergence. Moreover, for adaptive mesh-refinement, we see that the majorant is, in fact, a sharp estimate for the (in general unknown) potential error. 

The computed minorant is less accurate. \revision{With reference to Remark \ref{rem:placementofminorant}, we stress that the minorant is always computed with lowest-order Raviart-Thomas elements
on the same boundary layer as the majorant (which is obtained by adaptivity driven by the majorant).
In Figure~\ref{fig:example1:errmajmin_fullmin}, we even see that the minorant hardly enhances
when the mixed problem~\eqref{eq:mixedproblemstripdiscrete} is solved on the full domain $\TT_{h}$.
In our view, this indicates that the numerical treatment of the boundary residual $g-\str{u_{h}}{\Gamma}$ and its oscillations
is a key-point for accuracy, i.e., one should consider higher-order elements for the minorant.}

The empirical values for uniform (resp.\ adaptive) mesh-refinement are also provided in Table~\ref{table1} (resp.\ Table~\ref{table2}).
\revision{In particular, we note that the ratio between the FEM DoF for obtaining the error estimates and the BEM DoF remains bounded, so that additional computational expenditures remain limited. The same observation is made if we compare the corresponding expenditures in terms of CPU time.}
%The accuracy of the lower bound could be improved by appropriately adapting a second boundary layer (not displayed). 
Figure~\ref{fig:example1:theta} compares the numerical results for different choices of the adaptivity parameter $\theta \in \{0.2, 0.4, 0.6, 0.8\}$. We observe that any choice of $\theta$ regains, in fact, the optimal convergence rate.

%%%%%%%%%%%%%%%%%%%%%%%%%%%%%%%%%%%%%%%%%%%%%%%%%%%%%%%%%%%%%%%%%%%%%%%%%%%%%%%%%%%%
%%%%%%%%%%%%%%%%%%%%%%%%%%%%%%%%%%%%%%%%%%%%%%%%%%%%%%%%%%%%%%%%%%%%%%%%%%%%%%%%%%%%
%%%%%%%%%%%%%%%%%%%%%%%%%%%%%%%%%%%%%%%%%%%%%%%%%%%%%%%%%%%%%%%%%%%%%%%%%%%%%%%%%%%%
% EXAMPLE 2
\begin{figure}[t]
\begin{subfigure}[c]{.24\textwidth}
\centering\tiny
$\ell = 0$\\[1mm]
\includegraphics[trim={4.1cm 1.1cm 3.4cm 1cm},clip,width=\textwidth]{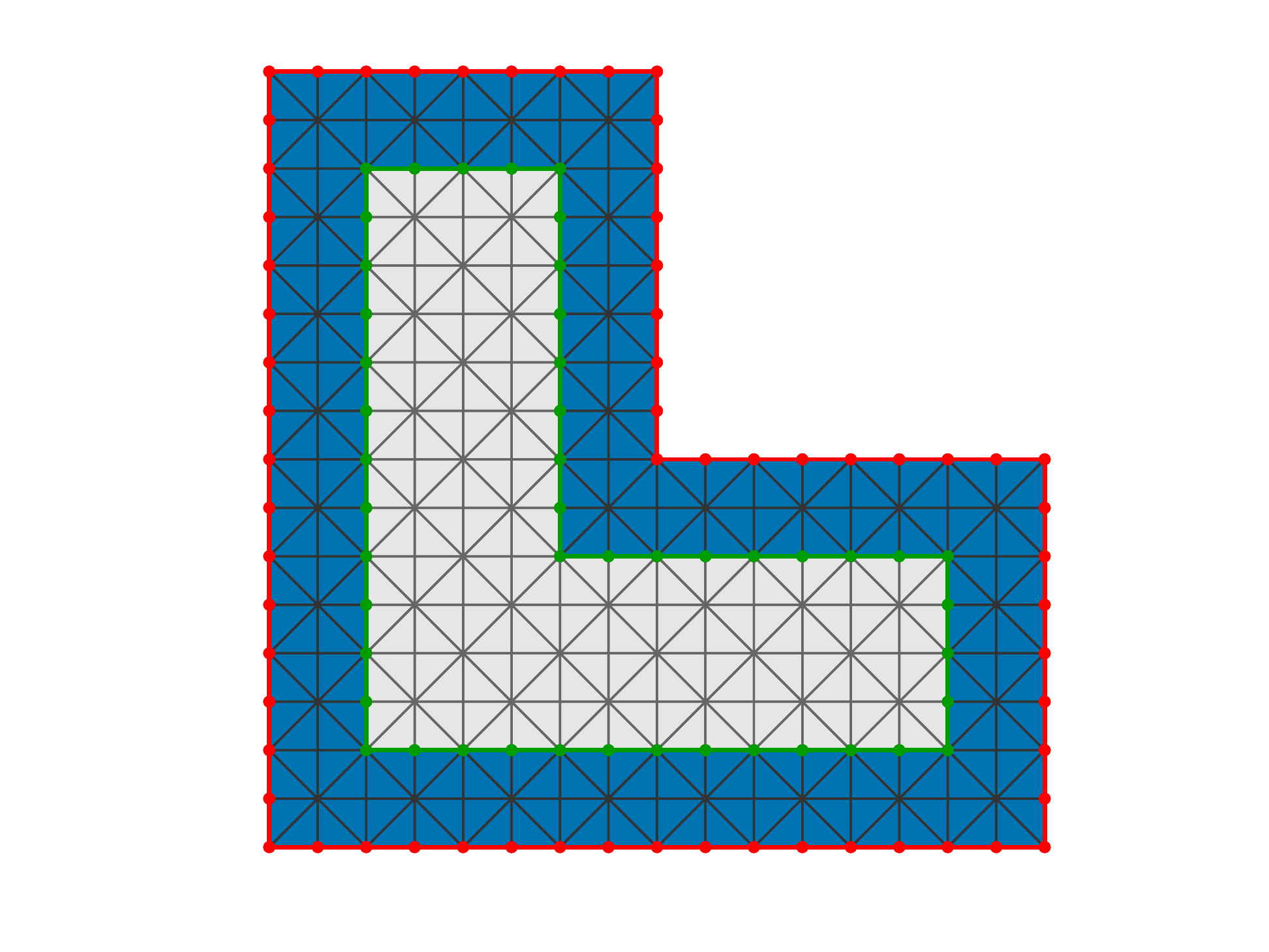}
$\#\TT_h^S = 168$, $\#\FF_h^\Gamma = 64$
\end{subfigure}
\hfill
\begin{subfigure}[c]{.24\textwidth}
\centering\tiny
$\ell = 17$\\[1mm]
\includegraphics[trim={4.1cm 1.1cm 3.4cm 1cm},clip,width=\textwidth]{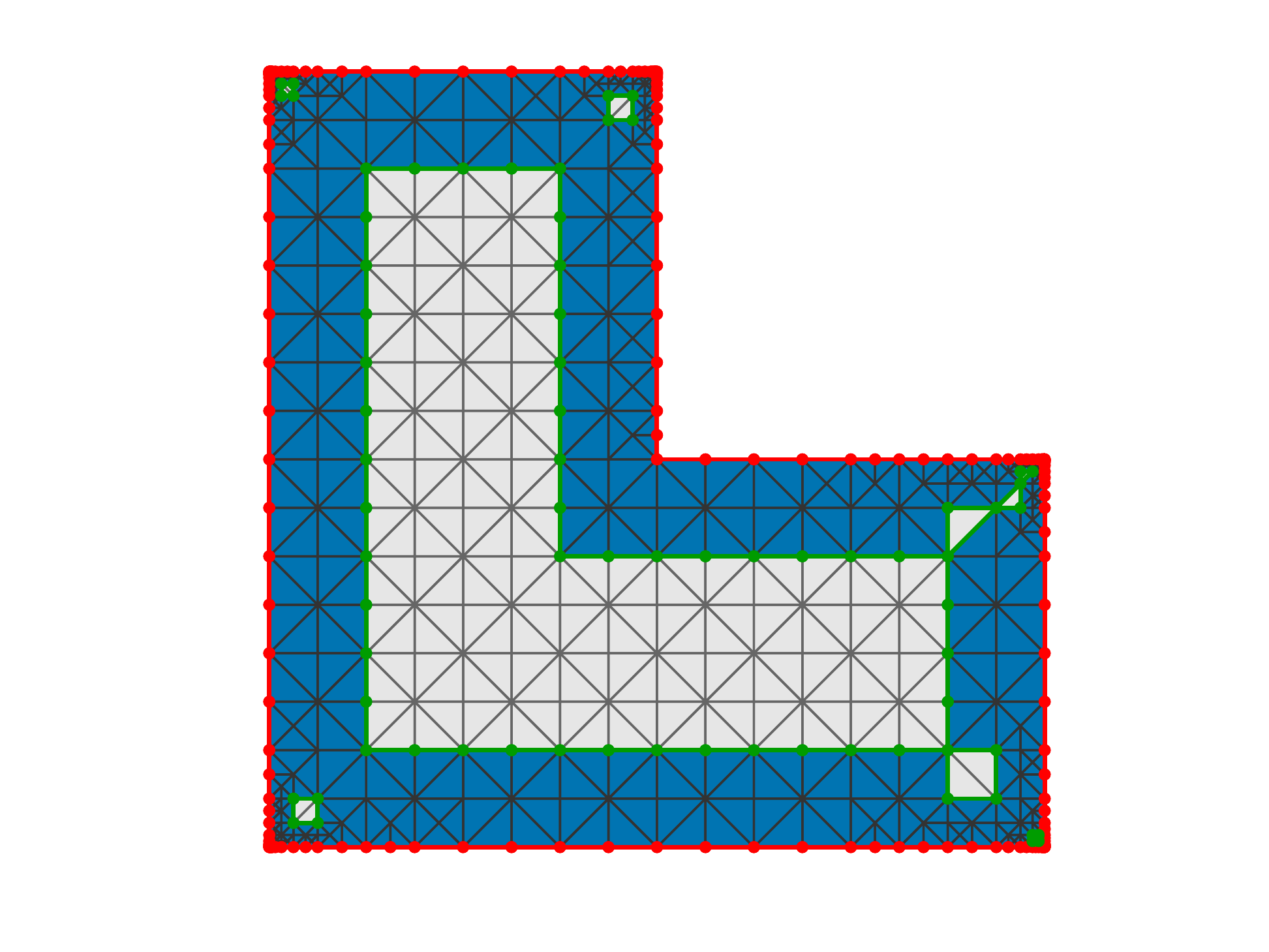}
$\#\TT_{h}^S = 532$, $\#\FF_{h}^\Gamma = 193$
\end{subfigure}
\hfill
\begin{subfigure}[c]{.24\textwidth}
\centering\tiny
$\ell = 27$\\[1mm]
\includegraphics[trim={4.1cm 1.1cm 3.4cm 1cm},clip,width=\textwidth]{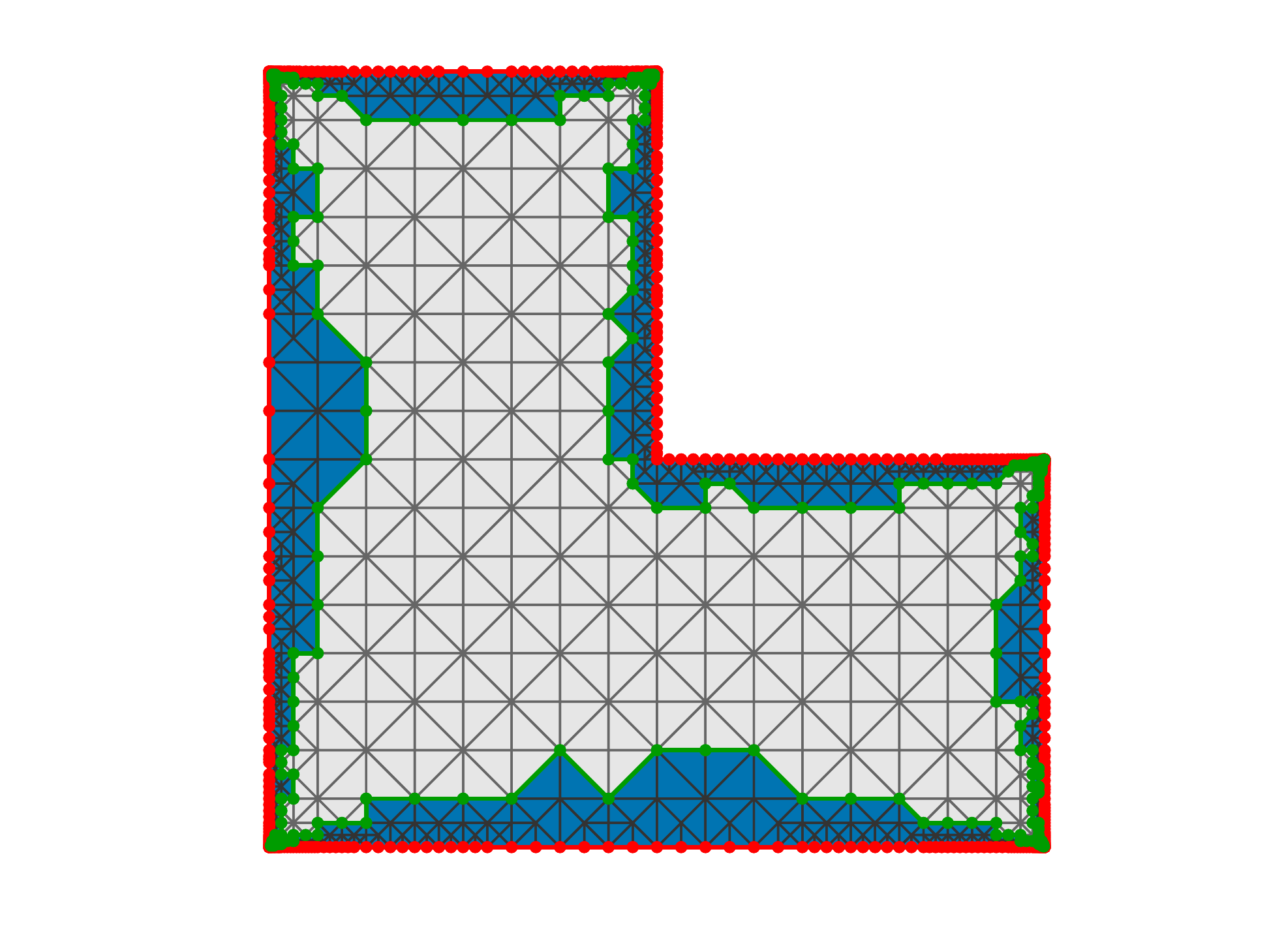}
$\#\TT_{h}^S = 1505$, $\#\FF_{h}^\Gamma = 562$
\end{subfigure}
\hfill
\begin{subfigure}[c]{.24\textwidth}
\centering\tiny
$\ell = 38$\\[1mm]
\includegraphics[trim={4.1cm 1.1cm 3.4cm 1cm},clip,width=\textwidth]{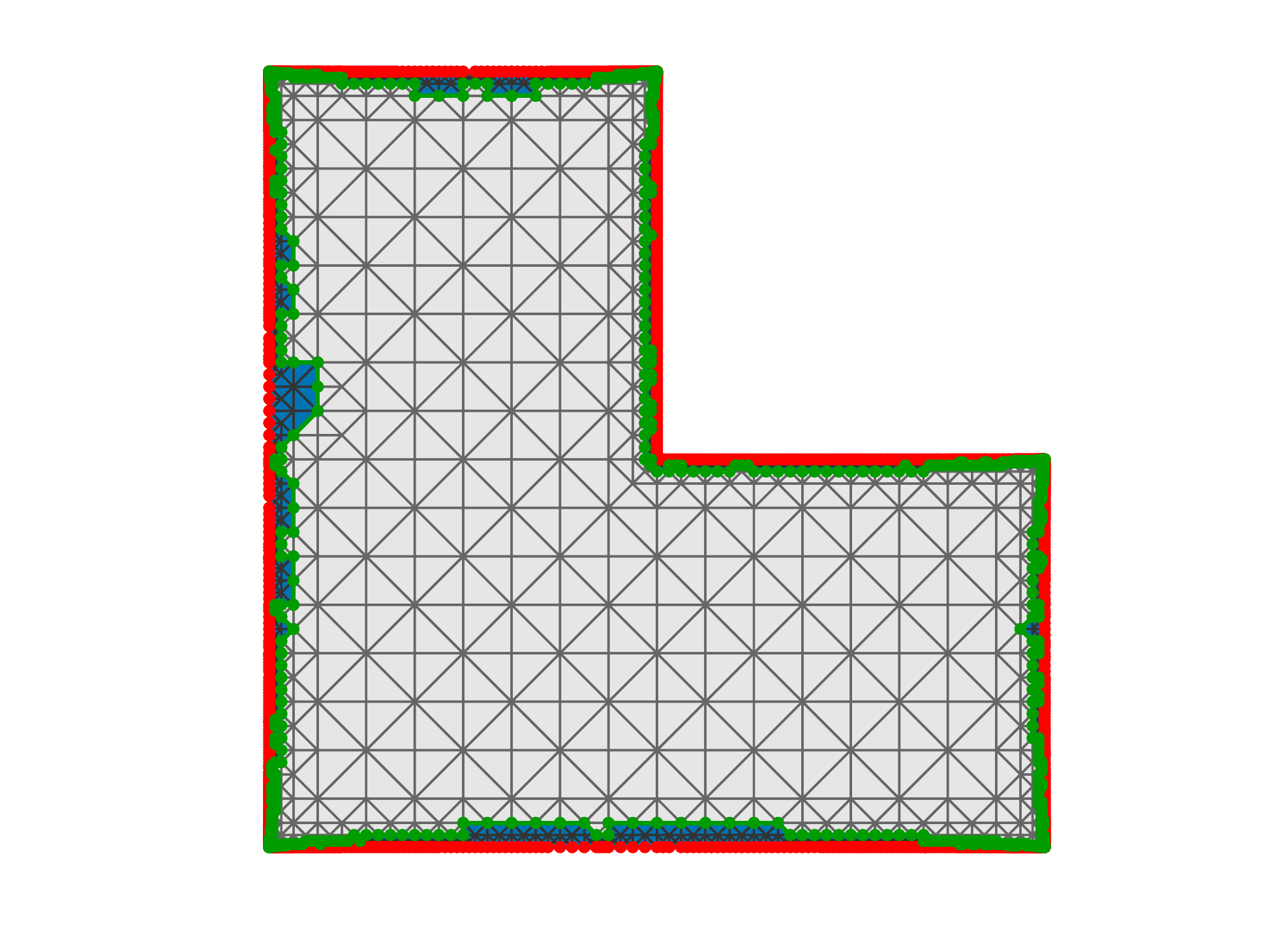}
$\#\TT_{h}^S = 4943$, $\#\FF_{h}^\Gamma = 1835$
\end{subfigure}
\caption{Adaptively generated meshes in Example~\ref{example2} for $p=1$ and $\theta = 0.6$;
 see Figure~\ref{fig:example1:mesh} for the color code.}
\label{fig:example2:mesh}
\end{figure}
\begin{figure}[t]
\centering
\includegraphics[trim={.7cm 0 0 0},clip,width=.49\textwidth]{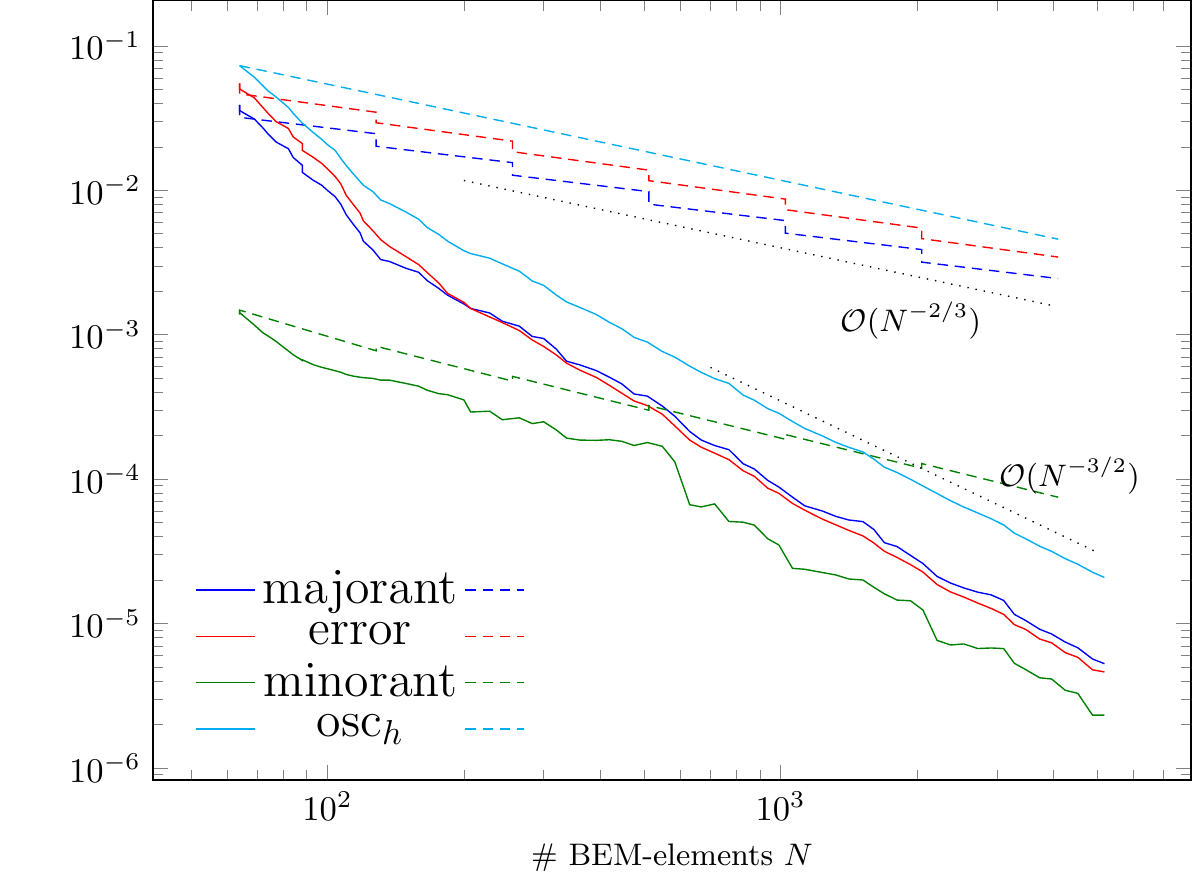}
\includegraphics[trim={.7cm 0 0 0},clip,width=.49\textwidth]{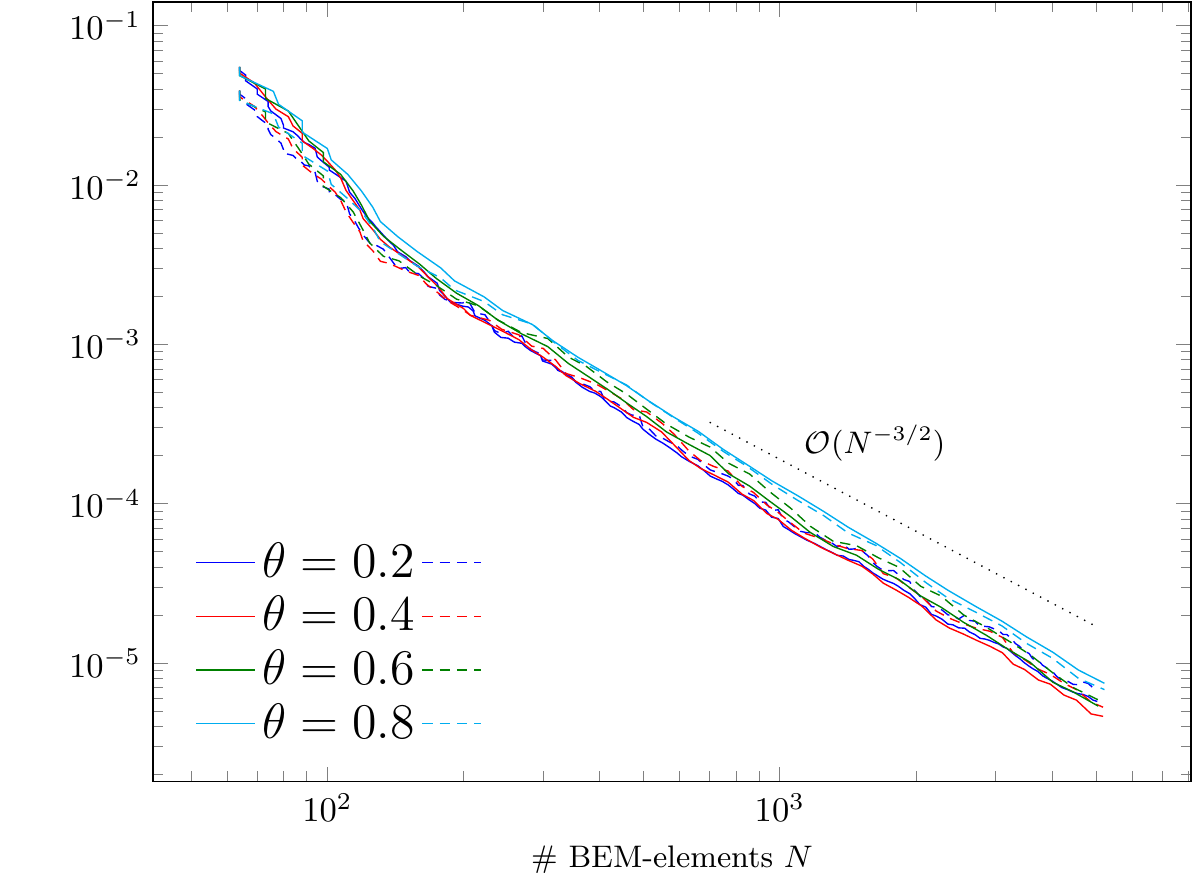}
\caption{Comparison of adaptive vs.\ uniform mesh-refinement in Example~\ref{example2}. The majorant is computed by $\PP^1$-FEM. \emph{Left:} 
We compare the potential error $\norm{\nabla(u-u_h)}{\Lt(\Omega)}$, the majorant $\norm{\nabla w_h}{\Lt(\Omega)}$ from~\eqref{maj:discrete}, the data oscillations ${\rm osc}_h$ from~\eqref{eq:osc}, and the minorant $\mmin(\ttau_h)^{1/2}$ from~\eqref{eq:experiments:minorant} for uniform (dashed) and adaptive mesh-refinement (solid) with $\theta = 0.4$. 
\emph{Right:} We compare the potential error (solid) and the majorant (dashed) for adaptive mesh-refinement for various choices of $\theta$.}
\label{fig:example2}
\end{figure}
%
%%%%%%%%%%%%%%%%%%%%%%%%%%%%%%%%%%%%%%%%%%%%%%%%%%%%%%%%%%%%%%%%%%%%%%%%%%%%%%%%%%%%
\example{Smooth potential in L-shaped domain}
\label{example2}
%%%%%%%%%%%%%%%%%%%%%%%%%%%%%%%%%%%%%%%%%%%%%%%%%%%%%%%%%%%%%%%%%%%%%%%%%%%%%%%%%%%%
%
We consider problem~\eqref{eq:strongform} with prescribed exact solution
\begin{align}
 u(x) = \cosh(x_1) \, \cos(x_2)
 \quad \text{for all } x \in \Omega := (0, 1/2)^2 \backslash \big( [(1/4,1/2]\times[0,1/4] \big)
\end{align}
on the L-shaped domain $\Omega$ with diameter $\diam(\Omega) = \sqrt{1/2}$. We start Algorithm~\ref{algorithm} with an initial triangulation $\TT_h$ of $\Omega$ into $\#\TT_0 = 384$ right triangles.

As in Section~\ref{example1}, the potential $u$ is smooth, but the sought density $\phi$ of the indirect BEM formulation lacks regularity.  
The initial meshes as well as some adaptively generated meshes are visualized in Figure~\ref{fig:example2:mesh}. Figure~\ref{fig:example2} visualizes some numerical results for uniform and adaptive mesh-refinement, where we proceed as in Section~\ref{example1}. Since $p = 1$ and $p = 2$ lead to similar results (not displayed), we only report the results for $p = 1$.

As expected from theory, the shape of $\Omega$ does not impact the functional error estimates: Overall, the results obtained correspond to those from Section~\ref{example1}, where uniform mesh-refinement leads to a suboptimal convergence behavior, which is cured by means of the proposed adaptive strategy.

%%%%%%%%%%%%%%%%%%%%%%%%%%%%%%%%%%%%%%%%%%%%%%%%%%%%%%%%%%%%%%%%%%%%%%%%%%%%%%%%%%%%
%%%%%%%%%%%%%%%%%%%%%%%%%%%%%%%%%%%%%%%%%%%%%%%%%%%%%%%%%%%%%%%%%%%%%%%%%%%%%%%%%%%%
%%%%%%%%%%%%%%%%%%%%%%%%%%%%%%%%%%%%%%%%%%%%%%%%%%%%%%%%%%%%%%%%%%%%%%%%%%%%%%%%%%%%
% EXAMPLE 3
%
\begin{figure}[t]
\begin{subfigure}[c]{.24\textwidth}
\centering\tiny
$\ell = 0$\\[1mm]
\includegraphics[trim={4.1cm 1.1cm 3.4cm 1cm},clip,width=\textwidth]{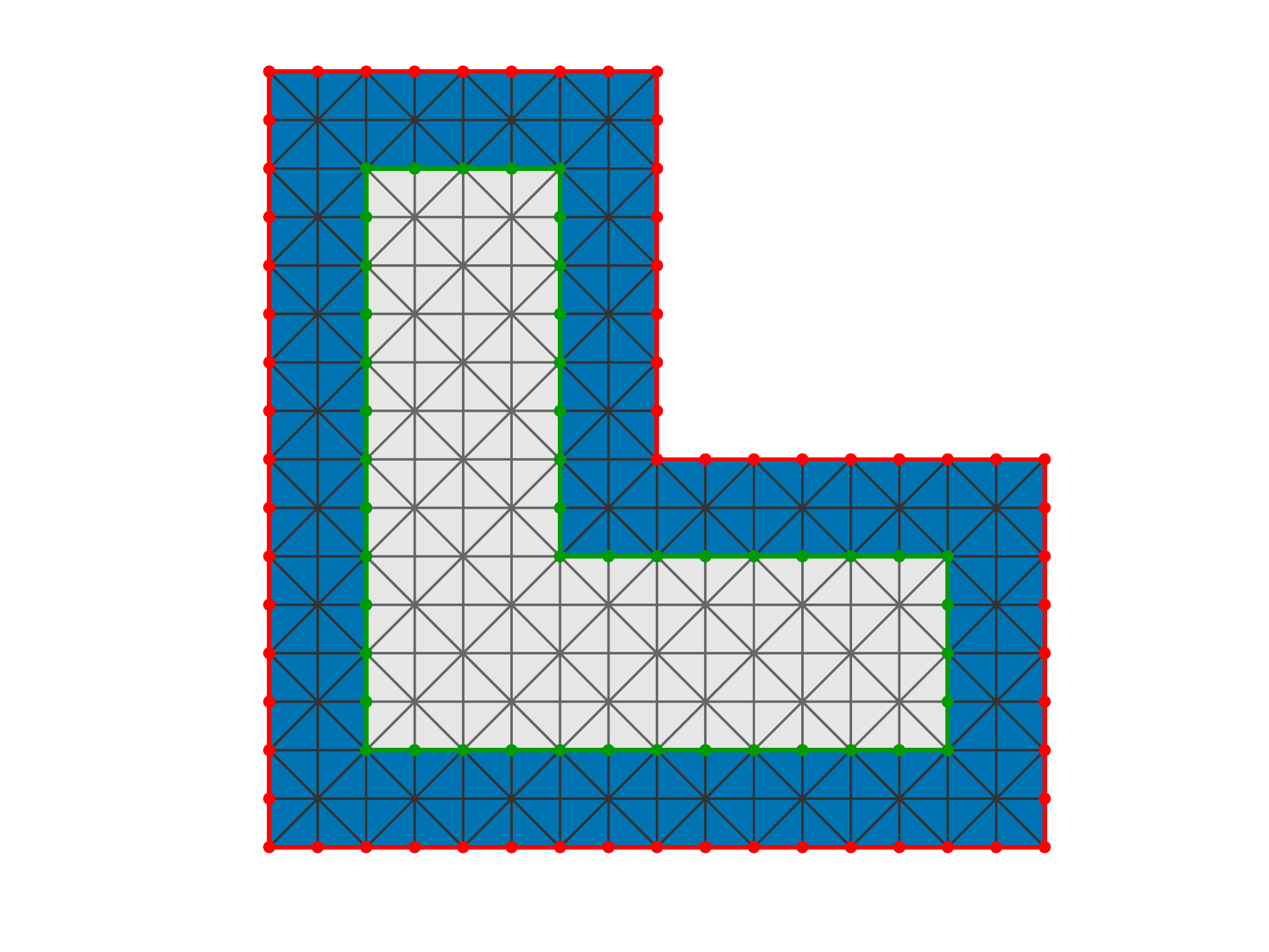}
$\#\TT_h^S = 168$, $\#\FF_h^\Gamma = 64$
\end{subfigure}
\hfill
\begin{subfigure}[c]{.24\textwidth}
\centering\tiny
$\ell = 20$\\[1mm]
\includegraphics[trim={4.1cm 1.1cm 3.4cm 1cm},clip,width=\textwidth]{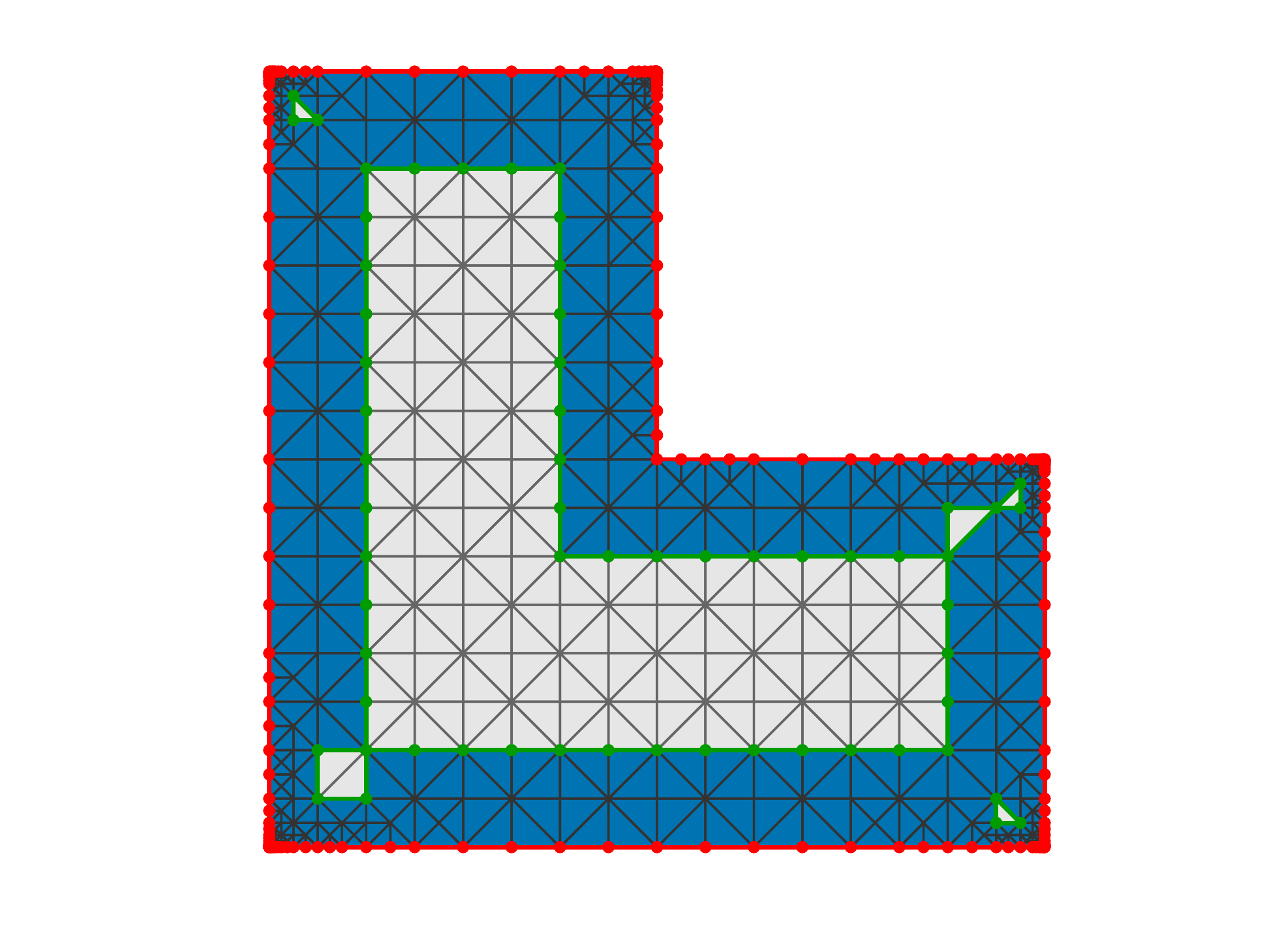}
$\#\TT_{h}^S = 514$, $\#\FF_{h}^\Gamma = 201$
\end{subfigure}
\hfill
\begin{subfigure}[c]{.24\textwidth}
\centering\tiny
$\ell = 26$\\[1mm]
\includegraphics[trim={4.1cm 1.1cm 3.4cm 1cm},clip,width=\textwidth]{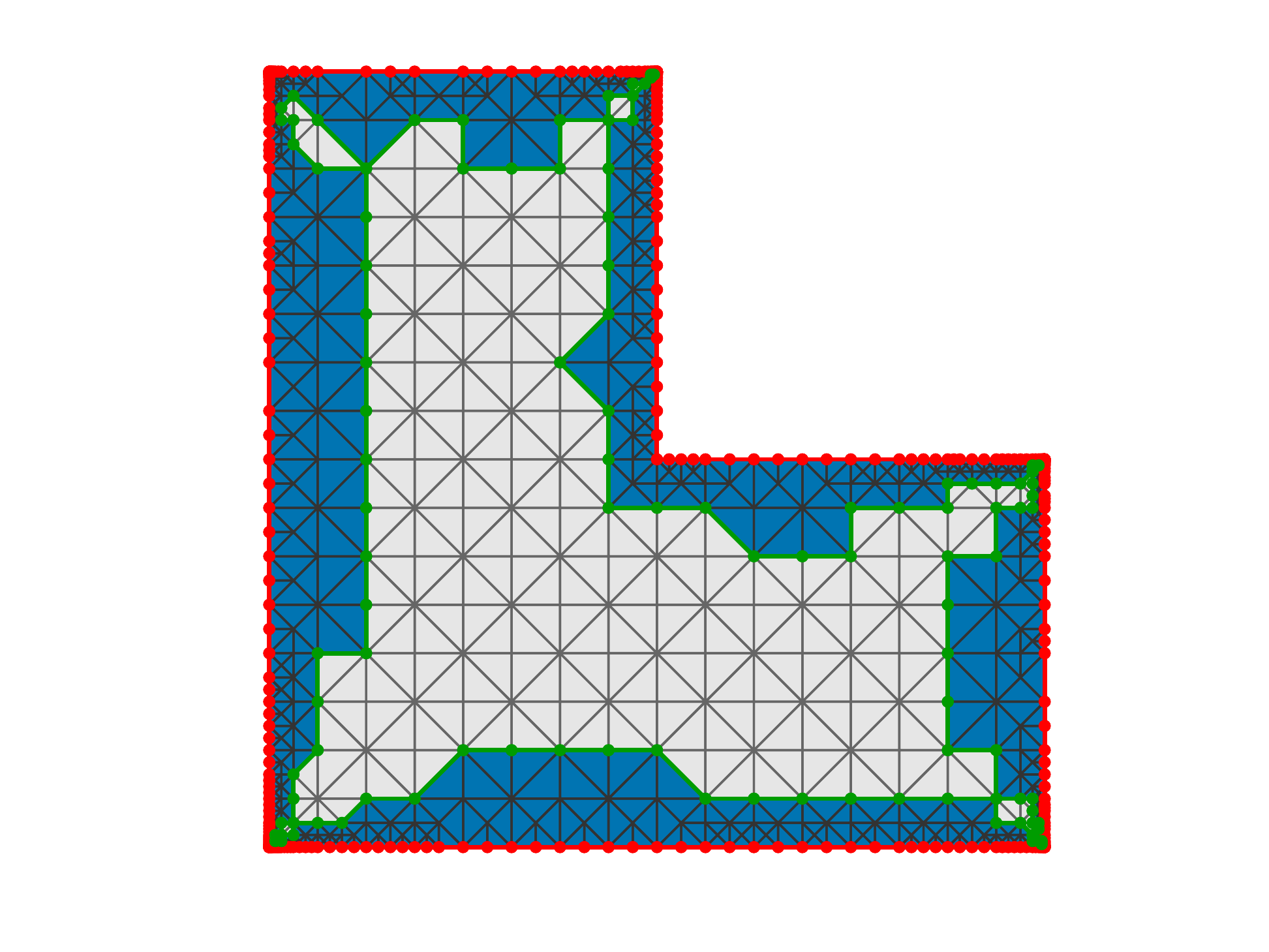}
$\#\TT_{h}^S = 919$, $\#\FF_{h}^\Gamma = 368$
\end{subfigure}
\hfill
\begin{subfigure}[c]{.24\textwidth}
\centering\tiny
$\ell = 38$\\[1mm]
\includegraphics[trim={4.1cm 1.1cm 3.4cm 1cm},clip,width=\textwidth]{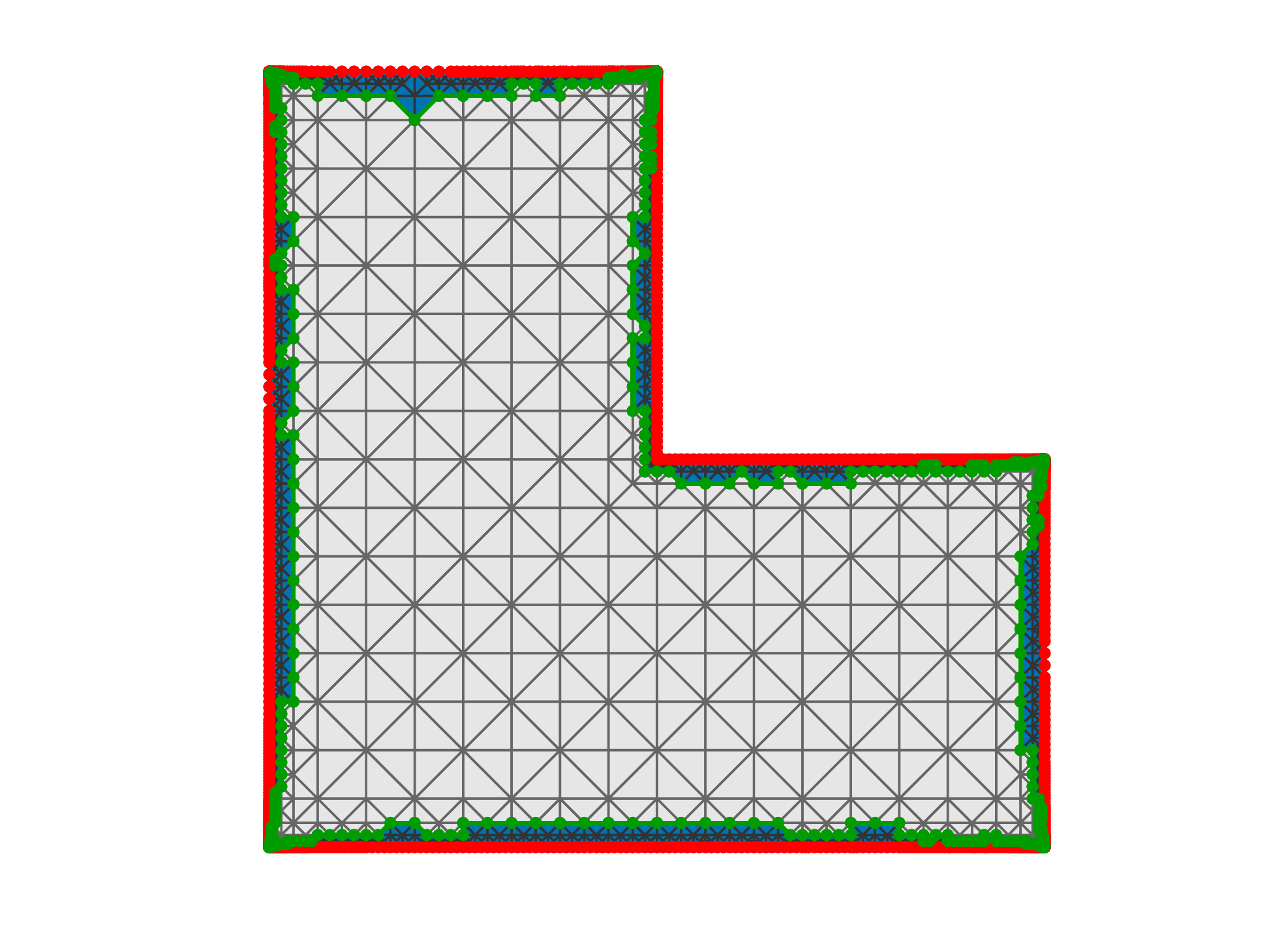}
$\#\TT_{h}^S = 3177$, $\#\FF_{h}^\Gamma = 1216$
\end{subfigure}
\caption{Adaptively generated meshes in Example~\ref{example3} for $p=1$ and $\theta = 0.6$;
 see Figure~\ref{fig:example1:mesh} for the color code.}
\label{fig:example3:mesh}
\end{figure}

\begin{figure}[t]
\centering
\includegraphics[trim={.7cm 0 0 0},clip,width=.49\textwidth]{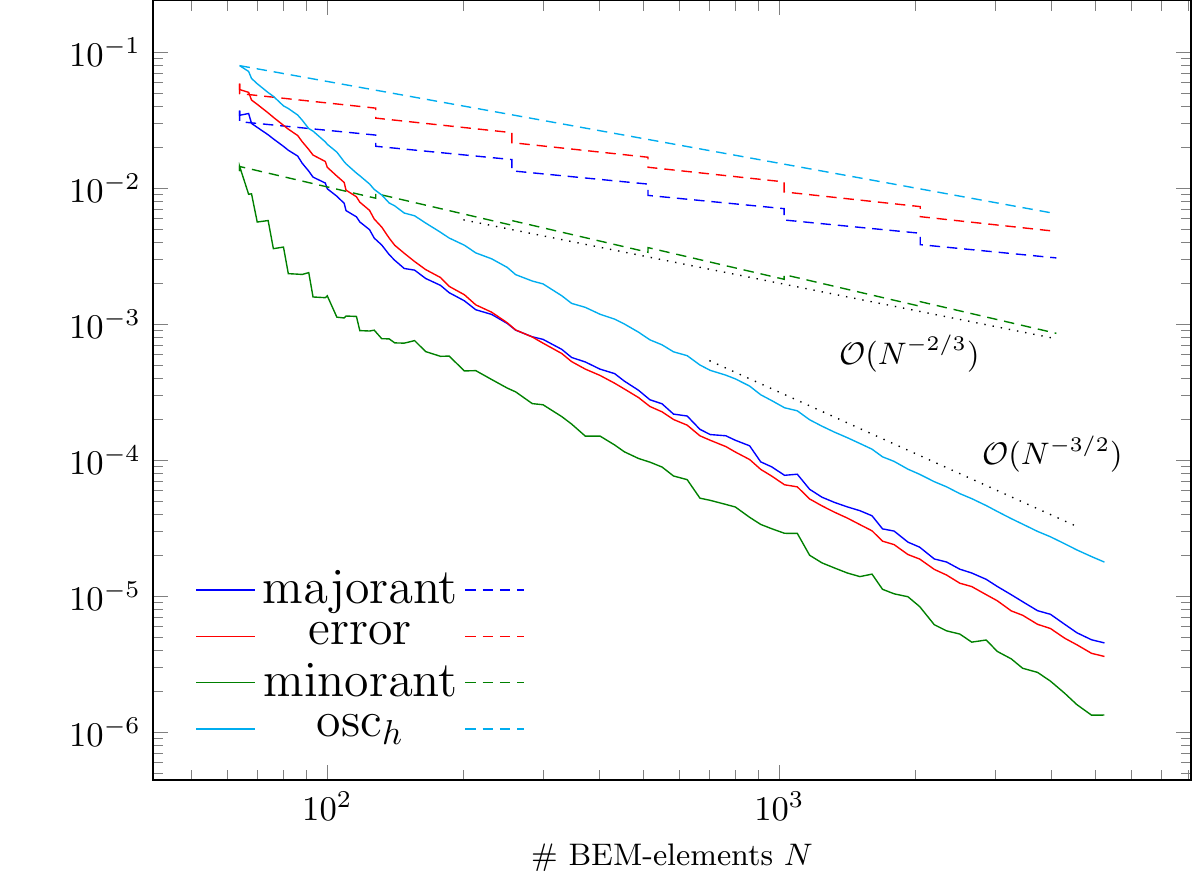}
\includegraphics[trim={.7cm 0 0 0},clip,width=.49\textwidth]{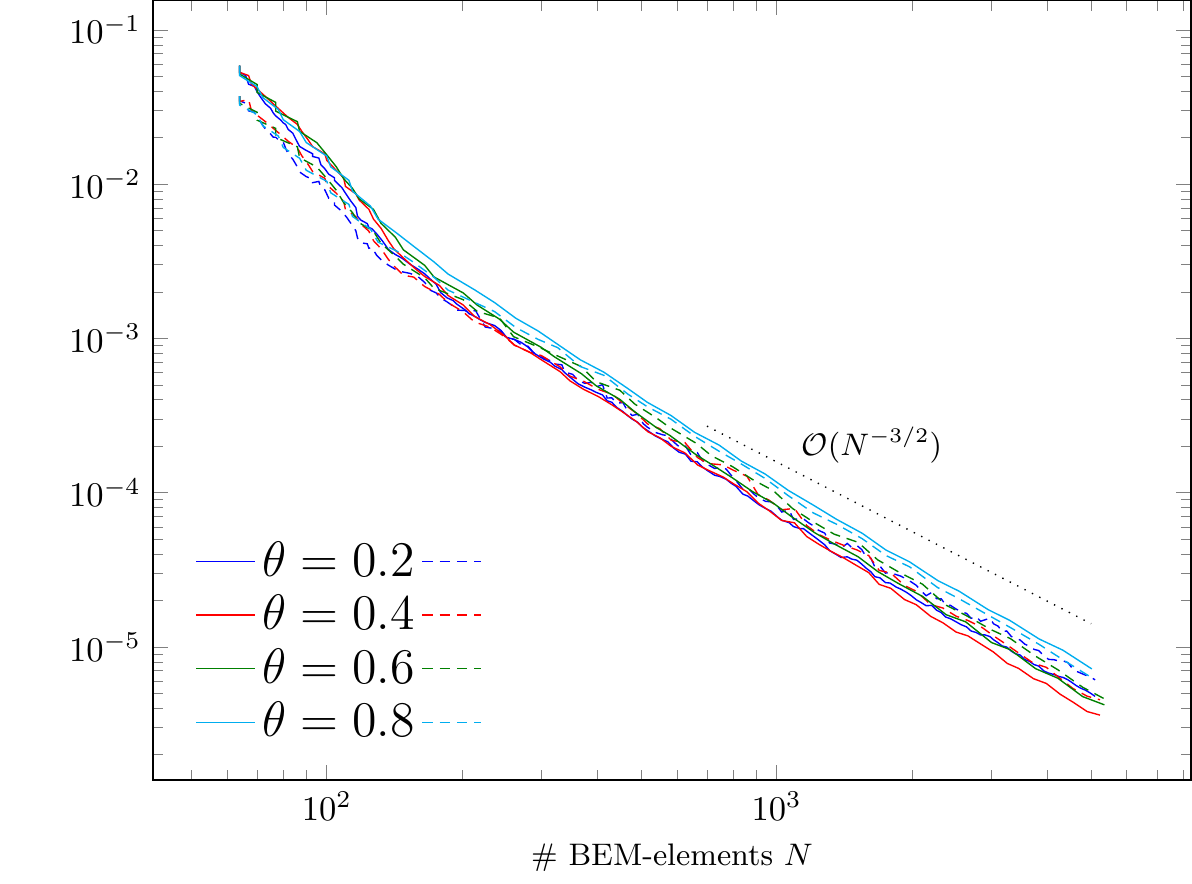}
\caption{Comparison of adaptive vs.\ uniform mesh-refinement in Example~\ref{example3}. The majorant is computed by $\PP^1$-FEM. \emph{Left:} 
We compare the potential error $\norm{\nabla(u-u_h)}{\Lt(\Omega)}$, the majorant $\norm{\nabla w_h}{\Lt(\Omega)}$ from~\eqref{maj:discrete}, the data oscillations ${\rm osc}_h$ from~\eqref{eq:osc}, and the minorant $\mmin(\ttau_h)^{1/2}$ from~\eqref{eq:experiments:minorant} for uniform (dashed) and adaptive mesh-refinement (solid) with $\theta = 0.4$. 
\emph{Right:} We compare the potential error (solid) and the majorant (dashed) for adaptive mesh-refinement for various choices of $\theta$.}
\label{fig:example3}
\end{figure}
\begin{figure}[t]
\centering
\includegraphics[trim={.1cm 0 0 0},clip,width=.49\textwidth]{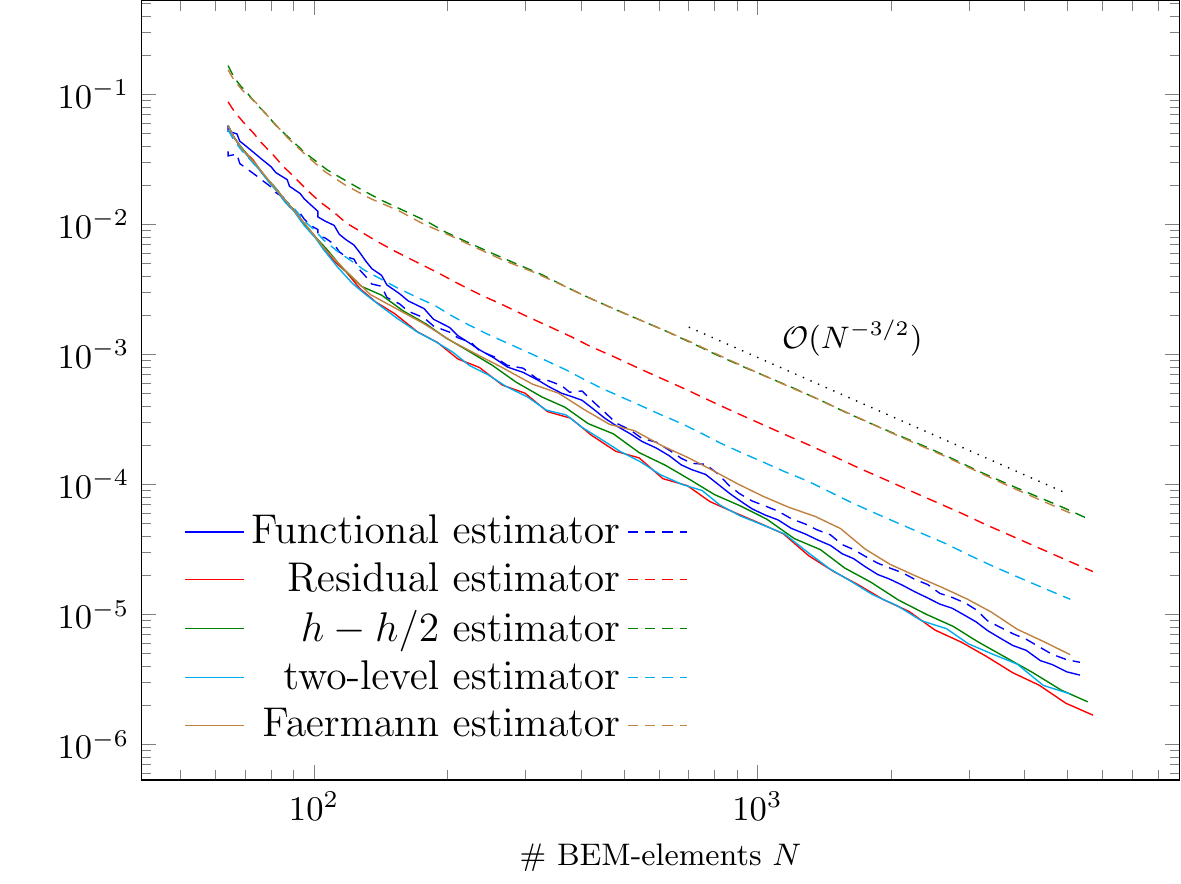}
\includegraphics[trim={.1cm 0 0 0},clip,width=.48\textwidth]{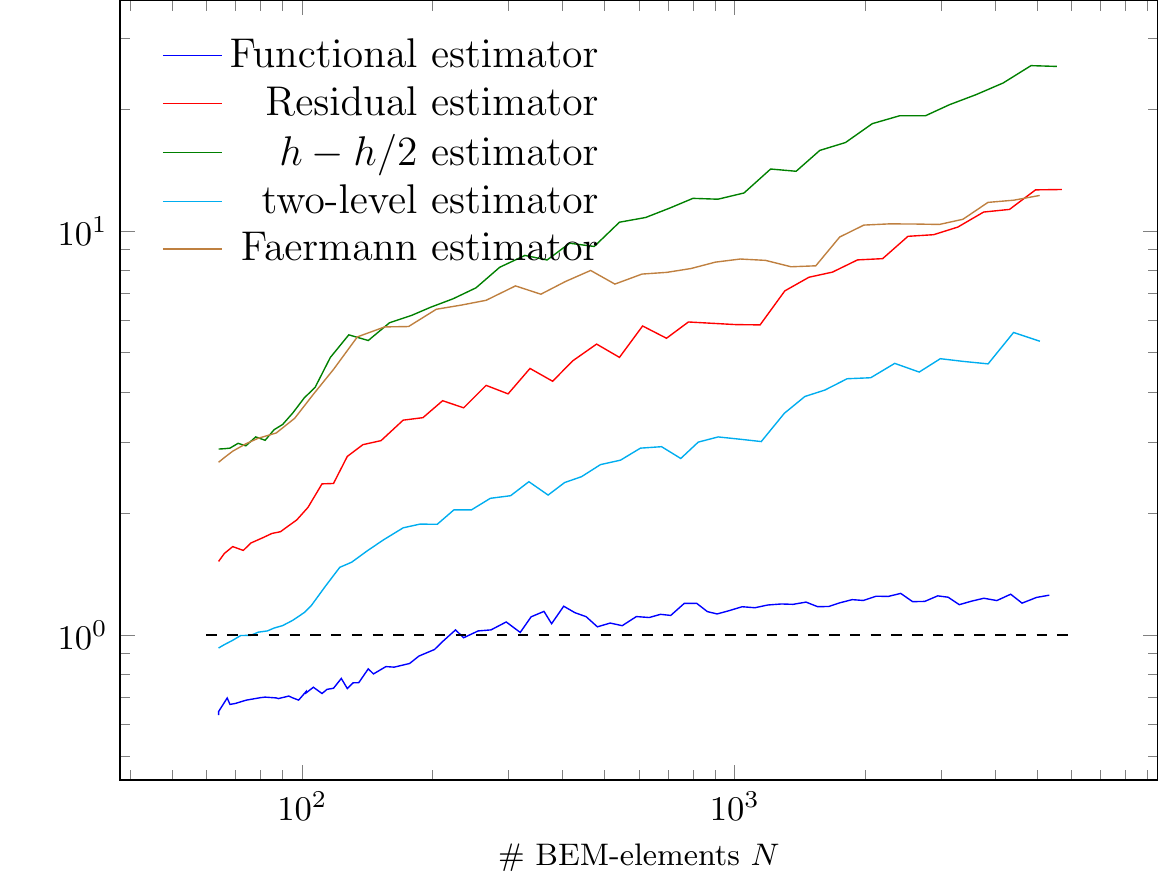}
\caption{\revision{Numerical results for adaptive mesh-refinement ($\theta = 0.4$) in Example~\ref{example3} for different {\sl a~posteriori} BEM error estimators. 
{\em Left:} We plot both the potential error $\norm{\nabla(u-u_h)}{\Lt(\Omega)}$ (solid) as well as the corresponding error estimator (dashed), which drives the adaptive strategy.  
Our functional estimator is the majorant $\norm{\nabla w_h}{\Lt(\Omega)}$ from~\eqref{maj:discrete} based on P1-FEM. {\em Right:} For either estimator $\mu_h$, we plot the quotient $\mu_h / \norm{\nabla(u-u_h)}{\Lt(\Omega)}$ 
to visualize the accuracy of the estimator with respect to the potential error.}}
\label{fig2:example3}
\end{figure}

%%%%%%%%%%%%%%%%%%%%%%%%%%%%%%%%%%%%%%%%%%%%%%%%%%%%%%%%%%%%%%%%%%%%%%%%%%%%%%%%%%%%
\example{Non-smooth potential in L-shaped domain}
\label{example3}
%%%%%%%%%%%%%%%%%%%%%%%%%%%%%%%%%%%%%%%%%%%%%%%%%%%%%%%%%%%%%%%%%%%%%%%%%%%%%%%%%%%%
%
We consider problem~\eqref{eq:strongform} with prescribed exact solution
\begin{align}\label{eq:example3}
 u(x) = r^{2/3} \cos(2\varphi/3)
 \quad \text{for all } x \in \Omega := (0, 1/2)^2 \backslash \big( [(1/4,1/2]\times[0,1/4] \big)
\end{align}
given in standard polar coordinates $x=x(r,\varphi)$ on the L-shaped domain $\Omega$
with diameter $\diam(\Omega) = \sqrt{1/2}$. We start Algorithm~\ref{algorithm} with an initial triangulation $\TT_0$ of $\Omega$ into $\#\TT_0 = 384$ right triangles.

Unlike Section~\ref{example1} and Section~\ref{example2}, 
the potential $u$ is non-smooth at $(0,0)$.   
The initial meshes as well as some adaptively generated meshes are visualized in Figure~\ref{fig:example3:mesh}.
Numerical convergence results are visualized in Figure~\ref{fig:example3}. Moreover, Table~\ref{table3} provides some empirical values for adaptive mesh-refinement. Our observations are the same as in Section~\ref{example1} and Section~\ref{example2} and underline that the functional error bounds do not rely on any {\sl a~priori} smoothness of the unknown potential $u$: While uniform mesh-refinement leads to a suboptimal convergence behavior, the proposed adaptive strategy regains the optimal convergence rate. 

\revision{Figure~\ref{fig2:example3} provides some estimator competition. We consider the functional error estimator proposed in the present work, the residual estimator $\mu_R$ from~\cite{cs1995,cc1997,cms2001}, the $h-h/2$ error estimator $\mu_H$ from~\cite{MR2529605}, the two-level error estimator $\mu_T$ from~\cite{mundstephanweiss98,hms2001,eh2006}, and Faermann's residual estimator $\mu_F$ from~\cite{MR1752263,MR1930387}. For lowest-order BEM, all these estimators are provided by the {\scshape Matlab} toolbox \textsc{Hilbert}~\cite{hilbert}. We recall that
\begin{align*}
 \mu_T \simeq \mu_H \lesssim \norm{\phi - \phi_h}{\H^{-1/2}(\Gamma)} \simeq \mu_F \lesssim \mu_R,
\end{align*}
where the constants hidden in $\simeq$ and $\lesssim$ depend only on $\Gamma$; see, e.g.,~\cite{cf01,MR2566766}. In addition, we stress that the converse estimate $\norm{\phi - \phi_h}{\H^{-1/2}(\Gamma)} \lesssim \mu_T \simeq \mu_H$ is equivalent to a saturation assumption~\cite{MR2566766}. Moreover, as mentioned earlier, there always holds the bound $\norm{\nabla(u-u_h)}{\Lt(\Omega)} \lesssim \norm{\phi - \phi_h}{\H^{-1/2}(\Gamma)}$, where the hidden constant depends on $\Gamma$. We consider Algorithm~\ref{algorithm} (with $\theta = 0.4$), where instead of $\eta_h(T)$ from~\eqref{eq:estimator}, we use 
$\eta_h(T)^2 := \sum_{F \in \FF_h^\Gamma, F \subset T} \mu_h(F)^2$, where $\mu_h(F)$ denote the local contributions of $\mu_h \in \{\mu_R, \mu_H, \mu_T, \mu_F\}$. Figure~\ref{fig2:example3} provides the numerical results. All adaptive strategies yield optimal decay $\norm{\nabla(u-u_h)}{\Lt(\Omega)} = \OO(N^{-3/2})$ with $N = \#\FF_h^\Gamma$. At the same time, we also see that the proposed functional estimator provides the most accurate bound on the potential error $\norm{\nabla(u-u_h)}{\Lt(\Omega)}$.}
%\note{{\bf Referee~\#2:} Is it possible to numerically compare the proposed scheme to the ABEM based on existing a posteriori error estimators such as: Residual, space enrichment, ZZ-type, others? \\ 
%{\bf Editor:} Ein Vergleich k\"onnte darin bestehen, dass man das Problem mit der einspringenden Ecke mit der gleichen adaptiven Strategie aber verschiedenen Fehlersch\"atzern behandelt und in einem Plot \#dof vs.\ error zusammenfasst.\\
%{\bf Dirk:} Ich schlage vor, alle Sch\"atzer aus Hilbert zu vergleichen, d.h.\ Residual, Faermann, twolevel, $h-h/2$, daf\"ur aber averaging wegzulassen, weil keine Implementierung existiert.}

%%%%%%%%%%%%%%%%%%%%%%%%%%%%%%%%%%%%%%%%%%%%%%%%%%%%%%%%%%%%%%%%%%%%%%%%%%%%%%%%%%%
%%%%%%%%%%%%%%%%%%%%%%%%%%%%%%%%%%%%%%%%%%%%%%%%%%%%%%%%%%%%%%%%%%%%%%%%%%%%%%%%%%%
\section{Extension of the analysis}
\label{section:extensions}
%%%%%%%%%%%%%%%%%%%%%%%%%%%%%%%%%%%%%%%%%%%%%%%%%%%%%%%%%%%%%%%%%%%%%%%%%%%%%%%%%%%
%%%%%%%%%%%%%%%%%%%%%%%%%%%%%%%%%%%%%%%%%%%%%%%%%%%%%%%%%%%%%%%%%%%%%%%%%%%%%%%%%%%

So far, we have considered functional {\sl a~posteriori} error estimation for an indirect BEM formulation~\eqref{eq:weakform} discretized by Galerkin BEM~\eqref{eq:discreteweakform}. The following sections address some obvious extensions of our analysis. While the subsequent numerical experiments (as well as those from Section~\ref{section:numerics}) focus on $d = 2$, we again stress that the theoretical results also apply to arbitrary dimensions, in particular to $d = 3$. 
However, 3D experiments are beyond the scope of this work and left to future research. 

%%%%%%%%%%%%%%%%%%%%%%%%%%%%%%%%%%%%%%%%%%%%%%%%%%%%%%%%%%%%%%%%%%%%%%%%%%%%%%%%%%%
\subsection{Collocation BEM}
%%%%%%%%%%%%%%%%%%%%%%%%%%%%%%%%%%%%%%%%%%%%%%%%%%%%%%%%%%%%%%%%%%%%%%%%%%%%%%%%%%%%
%
It is worth noting that all results of Section~\ref{section:approach} hold, in particular, for any $v = \widetilde V\phi_h$ with arbitrary $\phi_h \in \H^{-1/2}(\Gamma)$. Consequently, the computable bounds of Theorem~\ref{cor:continuousupperbound} (resp.\ Corollary~\ref{cor:continuouslowerboundtwo})
hold for any approximation $\phi_h \approx \phi$. In particular, Algorithm~\ref{algorithm} can also be applied to (e.g., lowest-order) collocation BEM, where $\phi_h \in \PP^0(\FF_h^\Gamma)$ is determined by collocation conditions
\begin{align}
 (V\phi_h)(x_F) = g(x_F)
 \quad \text{for all } F \in \FF_h^\Gamma,
\end{align}
where $x_F \in F$ is an appropriate collocation node (e.g., the center of mass). 
We stress that well-posedness of collocation BEM is non-obvious (see, e.g.,~\cite{MR1122060,MR1279236,MR1414417}). 
However, this does not affect our developed functional {\sl a~posteriori} error bounds.

\begin{table}[t]
\begin{tabular}{|c||c|c|c|c|c|c|c|}
\hline
$\ell$ & $\# \FF^\Gamma_h$ & $\frac{\# \TT^S_h}{\# \FF^\Gamma_h}$ & \footnotesize${\rm dof}(\TT_h^S)$ & \footnotesize$\norm{\nabla (u-u_h)}{\Lt(\Omega)}$ & \footnotesize$\norm{\nabla w_h}{\Lt(S)}$ & \footnotesize$\frac{\norm{\nabla w_h}{\Lt(S)}}{\norm{\nabla(u-u_h)}{\Lt(\Omega)}}$ & $\footnotesize\frac{\norm{\grad w_h}{\Lt(S)}}{\mmin(\ttau_h)^{1/2}}$ \\ \hline\hline
$0$ & $64$ & $2.63$ & $33$ & $5.87 e-2$ & $3.72 e-2$ & $0.64$ & $2.79$ \\ \hline
$6$ & $76$ & $2.57$ & $39$ & $3.32 e-2$ & $2.30 e-2$ & $0.69$ & $6.41$ \\ \hline
$12$ & $93$ & $2.62$ & $46$ & $1.75 e-2$ & $1.21 e-2$ & $0.69$ & $7.60$ \\ \hline
$18$ & $116$ & $2.60$ & $81$ & $8.65 e-3$ & $6.16 e-3$ & $0.71$ & $5.40$ \\ \hline
$24$ & $141$ & $2.68$ & $132$ & $3.80 e-3$ & $2.95 e-3$ & $0.78$ & $4.04$ \\ \hline
$30$ & $201$ & $2.68$ & $215$ & $1.65 e-3$ & $1.49 e-3$ & $0.90$ & $3.28$ \\ \hline
$36$ & $300$ & $2.65$ & $336$ & $7.26 e-4$ & $7.74 e-4$ & $1.07$ & $3.02$ \\ \hline
$42$ & $454$ & $2.59$ & $491$ & $3.34 e-4$ & $3.82 e-4$ & $1.14$ & $3.31$ \\ \hline
$48$ & $667$ & $2.62$ & $715$ & $1.51 e-4$ & $1.69 e-4$ & $1.12$ & $3.21$ \\ \hline
$54$ & $961$ & $2.60$ & $1023$ & $7.66 e-5$ & $8.95 e-5$ & $1.17$ & $2.85$ \\ \hline
$60$ & $1412$ & $2.56$ & $1447$ & $3.76 e-5$ & $4.55 e-5$ & $1.21$ & $3.06$ \\ \hline
$66$ & $2042$ & $2.57$ & $2048$ & $1.88 e-5$ & $2.30 e-5$ & $1.22$ & $2.73$ \\ \hline
$72$ & $3031$ & $2.53$ & $2927$ & $9.28 e-6$ & $1.18 e-5$ & $1.27$ & $3.00$ \\ \hline
$78$ & $4548$ & $2.51$ & $4283$ & $4.40 e-6$ & $4.80 e-6$ & $1.23$ & $3.38$ \\ \hline
$80$ & $5232$ & $2.49$ & $4835$ & $3.61 e-6$ & $4.54 e-6$ & $1.26$ & $3.39$ \\ \hline
\end{tabular}
\caption{Adaptive mesh-refinement with $\theta = 0.4$ in Example~\ref{example3}. We focus on the degrees of freedom, the potential error $\norm{\nabla(u-u_h)}{\Lt(\Omega)}$, the accuracy of the $\PP^1$-FEM majorant $\norm{\nabla w_h}{\Lt(\Omega)}$ from~\eqref{maj:discrete}, and the quotient of the majorant and minorant.}
\label{table3}
\end{table}

%%%%%%%%%%%%%%%%%%%%%%%%%%%%%%%%%%%%%%%%%%%%%%%%%%%%%%%%%%%%%%%%%%%%%%%%%%%%%%%%%%%
\subsection{Other BEM ansatz spaces}
%%%%%%%%%%%%%%%%%%%%%%%%%%%%%%%%%%%%%%%%%%%%%%%%%%%%%%%%%%%%%%%%%%%%%%%%%%%%%%%%%%%%
%
With the same argument as for collocation BEM, one can replace the discrete BEM ansatz space $\PP^0(\FF_h^\Gamma) \ni \phi_h$ by an arbitrary discrete space $\PP_h \subseteq \H^{-1/2}(\Gamma)$ (e.g., higher-order piecewise polynomials, splines, isogeometric NURBS, etc.). For $r \in \N_0$ and $\PP_h = \PP^r(\FF_h^\Gamma)$, we expect that the choices $p = r + 1$ and $q = r$ will lead to accurate 
computable upper and lower bounds in Theorem~\ref{cor:continuousupperbound}.
The numerical validation of this expectation is, however, beyond the scope of the present work.
\begin{figure}[t]
\centering
\includegraphics[trim={.7cm 0 0 0},clip,width=.49\textwidth]{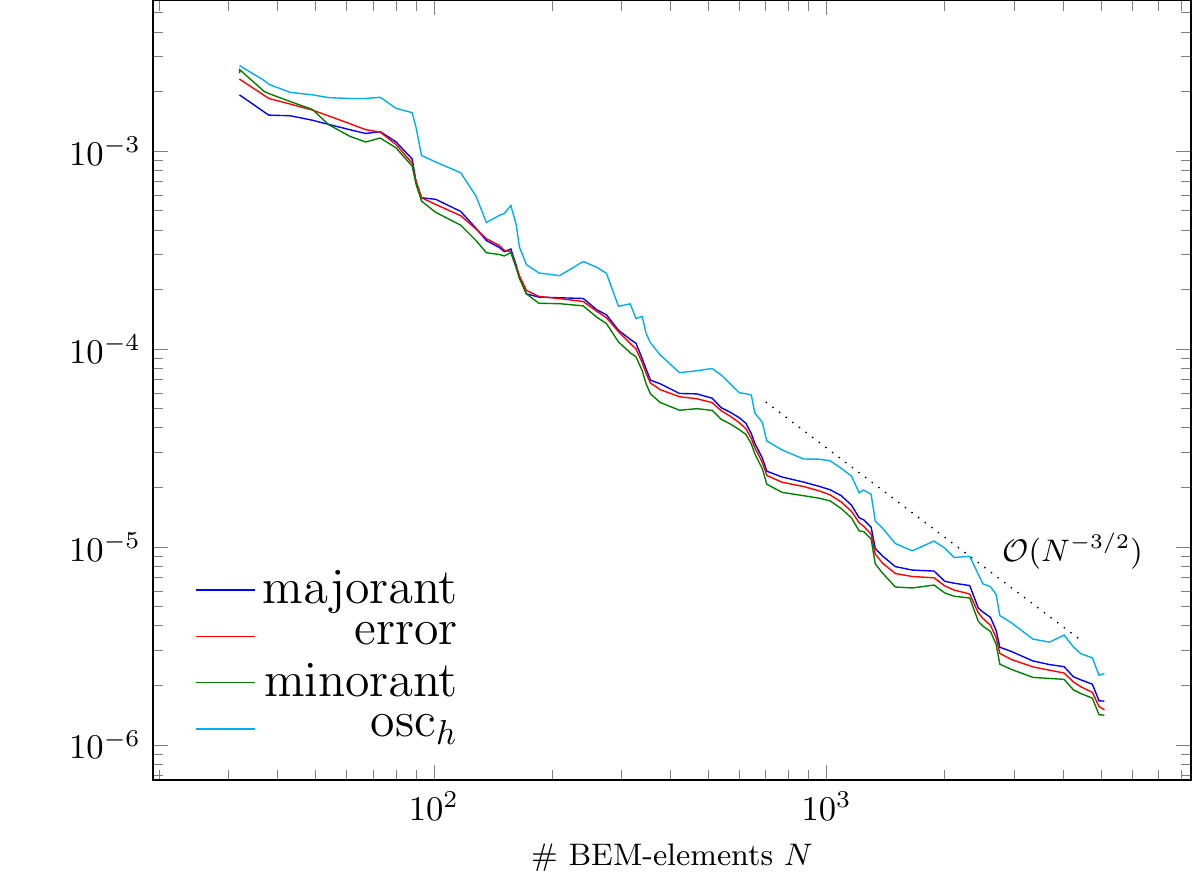}
\includegraphics[trim={.7cm 0 0 0},clip,width=.49\textwidth]{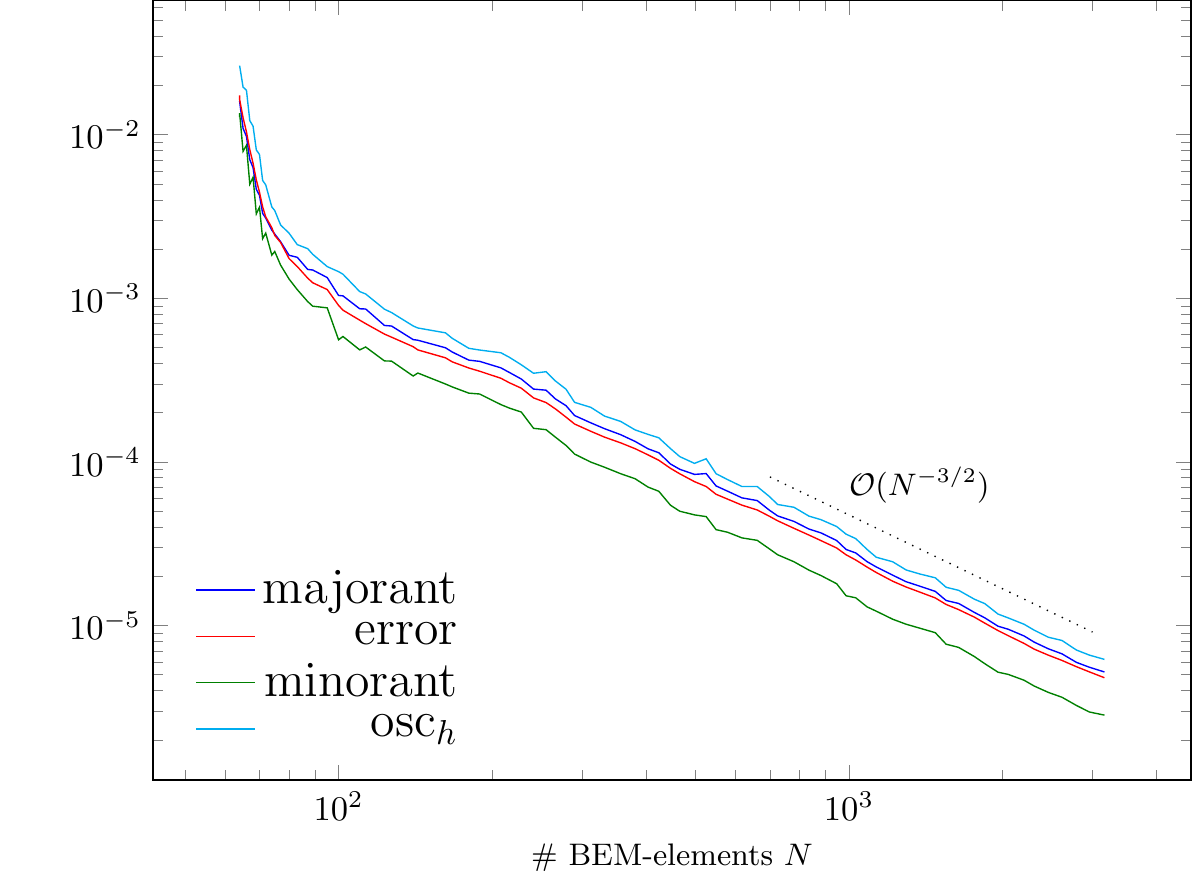}
\caption{Numerical results for adaptive mesh-refinement ($\theta = 0.4$) in Example~\ref{example5} (left) and Example~\ref{example6} (right). We plot the potential error $\norm{\nabla(u-u_h)}{\Lt(\Omega)}$, the majorant $\norm{\nabla w_h}{\Lt(\Omega)}$ from~\eqref{maj:discrete} based on $\PP^1$-FEM, the data oscillations ${\rm osc}_h$ from~\eqref{eq:osc:direct}, and the minorant $\mmin(\ttau_h)^{1/2}$ from~\eqref{eq:experiments:minorant}.}
\label{fig:example5-6}
\end{figure}
\begin{table}[th]
\begin{tabular}{|c||c|c|c|c|c|c|c|}
\hline
$\ell$ & $\# \FF^\Gamma_h$ & $\frac{\# \TT^S_h}{\# \FF^\Gamma_h}$ & \footnotesize${\rm dof}(\TT_h^S)$ & \footnotesize$\norm{\nabla (u-u_h)}{\Lt(\Omega)}$ & \footnotesize$\norm{\nabla w_h}{\Lt(S)}$ & \footnotesize$\frac{\norm{\nabla w_h}{\Lt(S)}}{\norm{\nabla(u-u_h)}{\Lt(\Omega)}}$ & $\footnotesize\frac{\norm{\grad w_h}{\Lt(S)}}{\mmin(\ttau_h)^{1/2}}$ \\ \hline\hline
$0$ & $32$ & $2.25$ & $15$ & $2.30 e-3$ & $1.92 e-3$ & $0.83$ & $0.78$ \\ \hline
$5$ & $49$ & $2.61$ & $54$ & $1.61 e-3$ & $1.43 e-3$ & $0.89$ & $0.88$ \\ \hline
$11$ & $88$ & $2.76$ & $110$ & $8.68 e-4$ & $9.12 e-4$ & $1.05$ & $1.09$ \\ \hline
$17$ & $147$ & $2.79$ & $161$ & $3.61 e-4$ & $3.53 e-4$ & $0.98$ & $1.15$ \\ \hline
$23$ & $172$ & $2.76$ & $203$ & $1.98 e-4$ & $1.90 e-4$ & $0.96$ & $1.00$ \\ \hline
$29$ & $295$ & $2.93$ & $387$ & $1.23 e-4$ & $1.25 e-4$ & $1.02$ & $1.15$ \\ \hline
$35$ & $377$ & $2.85$ & $442$ & $6.23 e-5$ & $6.68 e-5$ & $1.07$ & $1.24$ \\ \hline
$41$ & $599$ & $2.89$ & $743$ & $4.25 e-5$ & $4.50 e-5$ & $1.06$ & $1.15$ \\ \hline
$47$ & $770$ & $2.88$ & $901$ & $2.12 e-5$ & $2.26 e-5$ & $1.06$ & $1.20$ \\ \hline
$53$ & $1210$ & $2.93$ & $1533$ & $1.33 e-5$ & $1.40 e-5$ & $1.06$ & $1.16$ \\ \hline
$59$ & $1652$ & $2.96$ & $2005$ & $7.09 e-6$ & $7.64 e-6$ & $1.08$ & $1.23$ \\ \hline
$65$ & $2500$ & $2.99$ & $3228$ & $4.36 e-6$ & $4.69 e-6$ & $1.08$ & $1.18$ \\ \hline
$71$ & $3696$ & $2.91$ & $4988$ & $2.38 e-6$ & $2.55 e-6$ & $1.07$ & $1.18$ \\ \hline
$77$ & $5099$ & $2.99$ & $6483$ & $1.50 e-6$ & $1.66 e-6$ & $1.11$ & $1.18$ \\ \hline
\end{tabular}
\caption{Adaptive mesh-refinement with $\theta = 0.4$ in Example~\ref{example5}. We focus on the degrees of freedom, the potential error $\norm{\nabla(u-u_h)}{\Lt(\Omega)}$, the accuracy of the $\PP^1$-FEM majorant $\norm{\nabla w_h}{\Lt(\Omega)}$ from~\eqref{maj:discrete}, and the quotient of the majorant and minorant.}
\label{table4}
\end{table}
\begin{table}[th]
\begin{tabular}{|c||c|c|c|c|c|c|c|}
\hline
$\ell$ & $\# \FF^\Gamma_h$ & $\frac{\# \TT^S_h}{\# \FF^\Gamma_h}$ & \footnotesize${\rm dof}(\TT_h^S)$ & \footnotesize$\norm{\nabla (u-u_h)}{\Lt(\Omega)}$ & \footnotesize$\norm{\nabla w_h}{\Lt(S)}$ & \footnotesize$\frac{\norm{\nabla w_h}{\Lt(S)}}{\norm{\nabla(u-u_h)}{\Lt(\Omega)}}$ & $\footnotesize\frac{\norm{\grad w_h}{\Lt(S)}}{\mmin(\ttau_h)^{1/2}}$ \\ \hline\hline
$0$ & $64$ & $2.63$ & $33$ & $1.74 e-2$ & $1.61 e-2$ & $0.93$ & $0.74$ \\ \hline
$6$ & $69$ & $2.62$ & $37$ & $5.30 e-3$ & $4.64 e-3$ & $0.88$ & $0.80$ \\ \hline
$12$ & $77$ & $2.69$ & $42$ & $2.20 e-3$ & $2.22 e-3$ & $1.01$ & $0.95$ \\ \hline
$18$ & $100$ & $2.57$ & $46$ & $9.06 e-4$ & $1.04 e-3$ & $1.15$ & $1.21$ \\ \hline
$24$ & $140$ & $2.50$ & $79$ & $5.07 e-4$ & $5.60 e-4$ & $1.10$ & $1.16$ \\ \hline
$30$ & $208$ & $2.58$ & $159$ & $3.25 e-4$ & $3.76 e-4$ & $1.16$ & $1.16$ \\ \hline
$36$ & $279$ & $2.61$ & $248$ & $1.88 e-4$ & $2.21 e-4$ & $1.18$ & $1.20$ \\ \hline
$42$ & $404$ & $2.59$ & $381$ & $1.10 e-4$ & $1.20 e-4$ & $1.09$ & $1.19$ \\ \hline
$48$ & $549$ & $2.62$ & $531$ & $6.35 e-5$ & $7.14 e-5$ & $1.12$ & $1.30$ \\ \hline
$54$ & $780$ & $2.67$ & $790$ & $3.93 e-5$ & $4.33 e-5$ & $1.10$ & $1.23$ \\ \hline
$60$ & $1085$ & $2.67$ & $1131$ & $2.28 e-5$ & $2.46 e-5$ & $1.08$ & $1.33$ \\ \hline
$66$ & $1550$ & $2.73$ & $1695$ & $1.35 e-5$ & $1.42 e-5$ & $1.05$ & $1.29$ \\ \hline
$72$ & $2203$ & $2.71$ & $2372$ & $7.78 e-6$ & $7.92 e-6$ & $1.02$ & $1.32$ \\ \hline
$78$ & $3166$ & $2.74$ & $3442$ & $4.80 e-6$ & $5.20 e-6$ & $1.08$ & $1.29$ \\ \hline
\end{tabular}
\caption{Adaptive mesh-refinement with $\theta = 0.4$ in Example~\ref{example6}. We focus on the degrees of freedom, the potential error $\norm{\nabla(u-u_h)}{\Lt(\Omega)}$, the accuracy of the $\PP^1$-FEM majorant $\norm{\nabla w_h}{\Lt(\Omega)}$ from~\eqref{maj:discrete}, and the quotient of the majorant and minorant.}
\label{table5}
\end{table}

%%%%%%%%%%%%%%%%%%%%%%%%%%%%%%%%%%%%%%%%%%%%%%%%%%%%%%%%%%%%%%%%%%%%%%%%%%%%%%%%%%%%
%%%%%%%%%%%%%%%%%%%%%%%%%%%%%%%%%%%%%%%%%%%%%%%%%%%%%%%%%%%%%%%%%%%%%%%%%%%%%%%%%%%%
%%%%%%%%%%%%%%%%%%%%%%%%%%%%%%%%%%%%%%%%%%%%%%%%%%%%%%%%%%%%%%%%%%%%%%%%%%%%%%%%%%%%
% EXAMPLE 5 + EXAMPLE 6
%
%%%%%%%%%%%%%%%%%%%%%%%%%%%%%%%%%%%%%%%%%%%%%%%%%%%%%%%%%%%%%%%%%%%%%%%%%%%%%%%%%%%
\subsection{Direct BEM approach}
%%%%%%%%%%%%%%%%%%%%%%%%%%%%%%%%%%%%%%%%%%%%%%%%%%%%%%%%%%%%%%%%%%%%%%%%%%%%%%%%%%%
%
The indirect BEM approach makes ansatz~\eqref{eq:weakform} for the unknown solution of~\eqref{eq:strongform0}. Unlike this, the direct BEM approach is based on the Green's third identity: 
Any solution of~\eqref{eq:strongform0} can be written as the sum of a single-layer and a double-layer potential, i.e.,
\begin{align}\label{eq:green3}
 u(x) &= [\widetilde V\phi](x) - [\widetilde K g](x) 
 := [\widetilde V\phi](x) - \int_{\Gamma} \partial_{\normal(y)} G(x-y) \, g(y) \d{y}
 \quad \text{for all } x\in\Omega,
\end{align}
where $g = \str{u}{\Gamma} \in \H^{1/2}(\Gamma)$ is the trace of $u$ (i.e., the Dirichlet data) and $\phi = \ntr{\nabla u}{\Gamma} \in \H^{-1/2}(\Gamma)$ is the normal derivative (i.e., the Neumann data). Taking the trace of this identity and respecting the jump properties of the double-layer potential (see, e.g.,~\cite{mclean,steinbach,hsiao-wendland,sauter-schwab,gwinner-stephan}), one sees that
\begin{align*}
 g = V\phi - (K - 1/2) g 
 \quad \text{in } H^{1/2}(\Gamma),
\end{align*}
where $K$ formally coincides with $\widetilde K$, but is evaluated for $x \in \Gamma$ instead.
Elementary calculations then lead to the variational formulation
\begin{align}
\label{eq:weakform:direct}
\dualpga{V\phi}{\psi}
=\dualpga{(K+1/2) g}{\psi}
\quad\text{for all }\psi\in\H^{-1/2}(\Gamma).
\end{align}
We stress that the factor $1/2$ is only valid almost everywhere on $\Gamma$ and hence correct for the variational formulation and Galerkin BEM, while collocation BEM would require a modification at corners (and additionally along edges in 3D); see~\cite{mclean,steinbach,hsiao-wendland}.

Usual implementations approximate $g \approx g_h \in \SSS^p(\FF_h^\Gamma)$ so that the integral operators in~\eqref{eq:weakform:direct} are only evaluated for discrete functions. Overall, the lowest-order Galerkin BEM formulation then reads 
\begin{align}\label{eq:galerkin:direct}
\scp{V \phi_h}{\psi_h}{\Lt(\Gamma)} 
= \scp{(K+1/2) \, g_h}{\psi_h}{\Lt(\Gamma)} 
\quad \textit{for all } \psi_h\in \PP^0(\FF^\Gamma_h).
\end{align}
As above, the Lax--Milgram lemma proves that~\eqref{eq:weakform:direct} (resp.~\eqref{eq:galerkin:direct}) admit unique solutions $\phi \in \H^{-1/2}(\Gamma)$ (resp.\ $\phi_h \in \PP^0(\FF^\Gamma_h)$). Moreover, the computed density $\phi_h$ is now indeed an approximation of the Neumann data $\ntr{\nabla u}{\Gamma} = \str{\partial_{\normal} u}{\Gamma} = \phi \approx \phi_h$. Defining
\begin{align}\label{eq:green3:discrete}
 u_h(x) &= [\widetilde V\phi_h](x) - [\widetilde K g_h](x)
 \quad \text{for } x \in \Omega,
\end{align}
one obtains an approximation $u_h$ of the solution $u = \widetilde V\phi - \widetilde K g$ of~\eqref{eq:strongform0} (resp.~\eqref{eq:green3}). We stress that $u_h|_\Gamma = V\phi_h + (1/2 - K) g_h$ so that the data oscillation term in the upper bound of 
Theorem~\ref{cor:continuousupperbound} reads
\begin{subequations}\label{eq:osc:direct}
\begin{align}
\nonumber
 \norm{\nabla(u-u_h)}{\Lt(\Omega)} 
 &\le \!\!\!\!\min_{\substack{w \in \Ho(\Omega) \\ w|_\Gamma = J_h(g-u_h|_\Gamma)}} \!\!\!\!\! \norm{\nabla w}{\Lt(\Omega)}
 + \norm{(1-J_h) \big(g - V\phi_h - (1/2 - K)g_h\big)}{\H^{1/2}(\Gamma)}
 \\&
 \le \!\!\!\!\min_{\substack{w \in \Ho(\Omega) \\ w|_\Gamma = J_h(g-u_h|_\Gamma)}} \!\!\!\!\! \norm{\nabla w}{\Lt(\Omega)}
 + \Cosc^2 \, {\rm osc}_h,
 \intertext{where $J_h : \Lt(\Gamma) \to \SSS^1(\FF_h^\Gamma)$ is the $\Lt(\Gamma)$-orthogonal projection and}
 {\rm osc}_h &:= \norm{h^{1/2}\nabla_\Gamma \big[ (1-J_h) \big(g - V\phi_h - (1/2 - K)g_h\big)\big]}{\Lt(\Gamma)};
\end{align}
\end{subequations}
see Section~\ref{sec:oscillations}. In our implementation, we also employed $g_h = J_h g \in \SSS^1(\FF_h^\Gamma)$.

%%%%%%%%%%%%%%%%%%%%%%%%%%%%%%%%%%%%%%%%%%%%%%%%%%%%%%%%%%%%%%%%%%%%%%%%%%%%%%%%%%%%
\example{Direct BEM for smooth potential in square domain}
\label{example5}
%%%%%%%%%%%%%%%%%%%%%%%%%%%%%%%%%%%%%%%%%%%%%%%%%%%%%%%%%%%%%%%%%%%%%%%%%%%%%%%%%%%%
%
We consider the setting~\eqref{eq:example1} from Section~\ref{example1}. Applying the direct BEM approach~\eqref{eq:weakform:direct}, we know that $\phi_h \approx \phi = \ntr{\nabla u}{\Gamma}$, where $\phi$ (as well as the potential $u$) is smooth. In this particular situation, we know that uniform mesh-refinement would already lead to the optimal convergence behavior (not displayed). The same is 
%empirically 
observed for the proposed adaptive strategy, where we even observe that the majorant $\norm{\nabla w_h}{\Lt(\Omega)}$ from~\eqref{maj:discrete} as well as the minorant $\mmin(\ttau_h)^{1/2}$ from~\eqref{eq:experiments:minorant} provide sharp error bounds for the potential error $\norm{\nabla(u-u_h)}{\Lt(\Omega)}$; see Figure~\ref{fig:example5-6} (left) as well as Table~\ref{table4}.

%%%%%%%%%%%%%%%%%%%%%%%%%%%%%%%%%%%%%%%%%%%%%%%%%%%%%%%%%%%%%%%%%%%%%%%%%%%%%%%%%%%%
\example{Direct BEM for non-smooth potential in L-shaped domain}
\label{example6}
%%%%%%%%%%%%%%%%%%%%%%%%%%%%%%%%%%%%%%%%%%%%%%%%%%%%%%%%%%%%%%%%%%%%%%%%%%%%%%%%%%%%
%
We consider the setting~\eqref{eq:example3} from Section~\ref{example3}. Applying the direct BEM approach~\eqref{eq:weakform:direct}, we know that $\phi_h \approx \phi = \ntr{\nabla u}{\Gamma}$, where $\phi$ (as well as the potential $u$) is only non-smooth with a singularity at $(0,0)$. Also for this case, the proposed adaptive strategy regains the optimal convergence rate; see Figure~\ref{fig:example5-6} (right) as well as Table~\ref{table5}. Even though the quotient $\norm{\nabla w_h}{\Lt(\Omega)} / \mmin(\ttau_h)^{1/2}$ of the computable upper and lower bound is larger than for the smooth problem of Section~\ref{example5}, we observe that the lower bound is, in fact, 
much more accurate for the direct BEM than for the indirect BEM computations from Section~\ref{section:numerics}.

%%%%%%%%%%%%%%%%%%%%%%%%%%%%%%%%%%%%%%%%%%%%%%%%%%%%%%%%%%%%%%%%%%%%%%%%%%%%%%%%%%%%
%%%%%%%%%%%%%%%%%%%%%%%%%%%%%%%%%%%%%%%%%%%%%%%%%%%%%%%%%%%%%%%%%%%%%%%%%%%%%%%%%%%%
%%%%%%%%%%%%%%%%%%%%%%%%%%%%%%%%%%%%%%%%%%%%%%%%%%%%%%%%%%%%%%%%%%%%%%%%%%%%%%%%%%%%
% EXAMPLE 4
%
\begin{figure}[t]
\begin{subfigure}[c]{.24\textwidth}
\centering\tiny
$\ell = 0$\\[1mm]
\includegraphics[trim={4.1cm 1.1cm 3.4cm 1cm},clip,width=\textwidth]{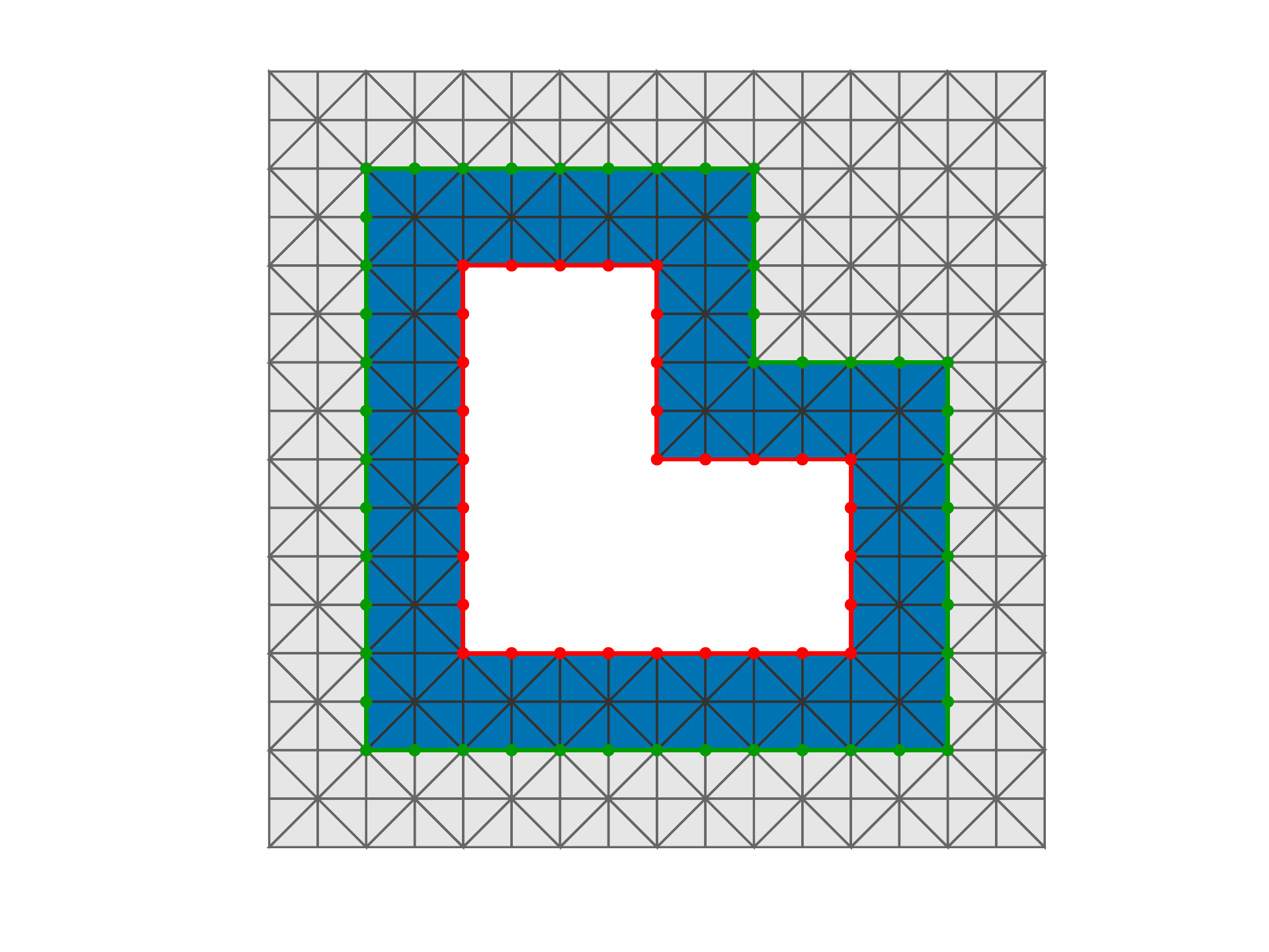}
$\#\TT_h^S = 120$, $\#\FF_h^\Gamma = 32$
\end{subfigure}%
\hfill
\begin{subfigure}[c]{.24\textwidth}
\centering\tiny
$\ell = 13$\\[1mm]
\includegraphics[trim={4.1cm 1.1cm 3.4cm 1cm},clip,width=\textwidth]{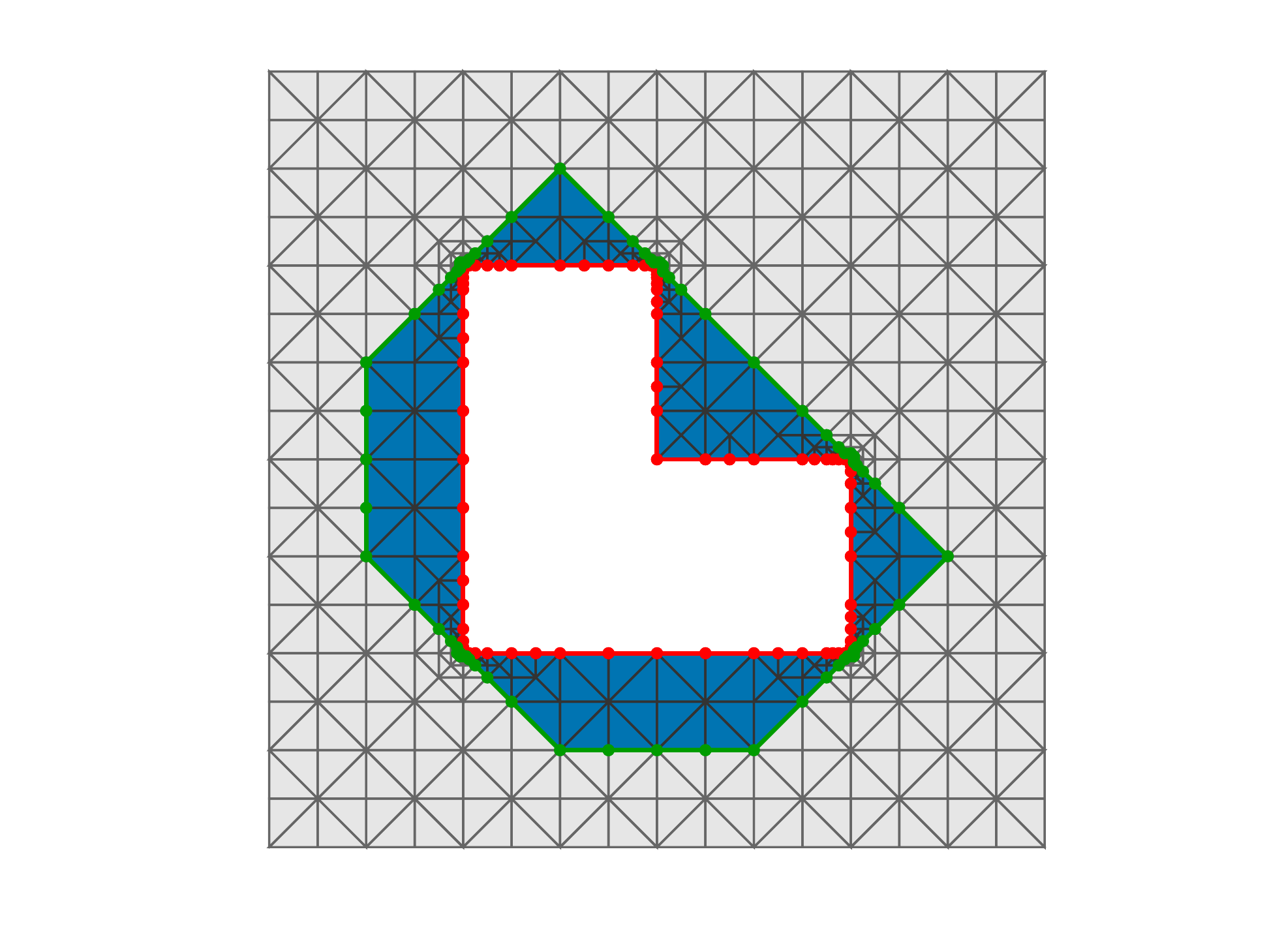}
$\#\TT_{h}^S = 347$, $\#\FF_{h}^\Gamma = 113$
\end{subfigure}
\hfill
\begin{subfigure}[c]{.24\textwidth}
\centering\tiny
$\ell = 27$\\[1mm]
\includegraphics[trim={4.1cm 1.1cm 3.4cm 1cm},clip,width=\textwidth]{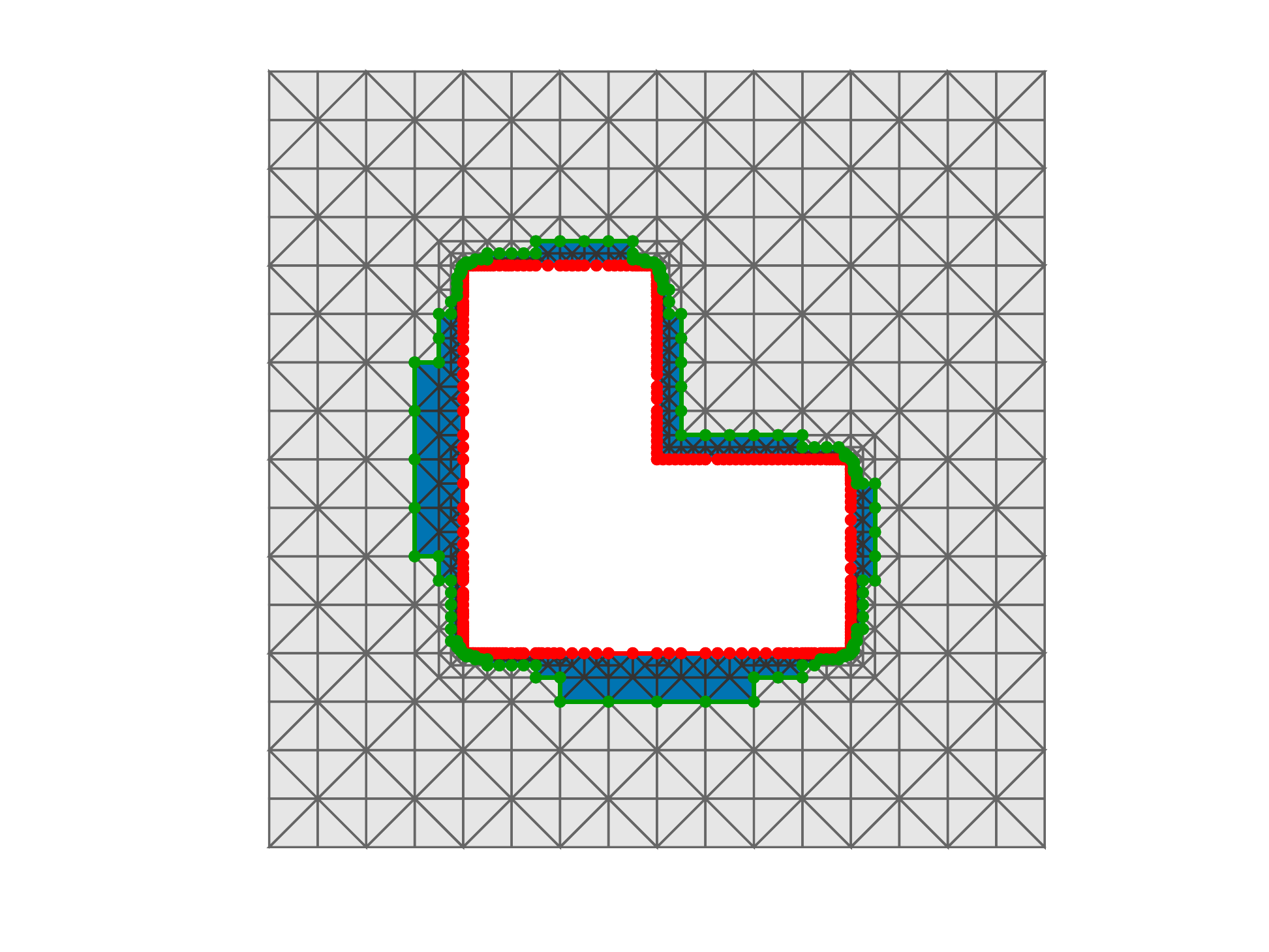}
$\#\TT_{h}^S = 1316$, $\#\FF_{h}^\Gamma = 461$
\end{subfigure}
\hfill
\begin{subfigure}[c]{.24\textwidth}
\centering\tiny
$\ell = 35$\\[1mm]
\includegraphics[trim={4.1cm 1.1cm 3.4cm 1cm},clip,width=\textwidth]{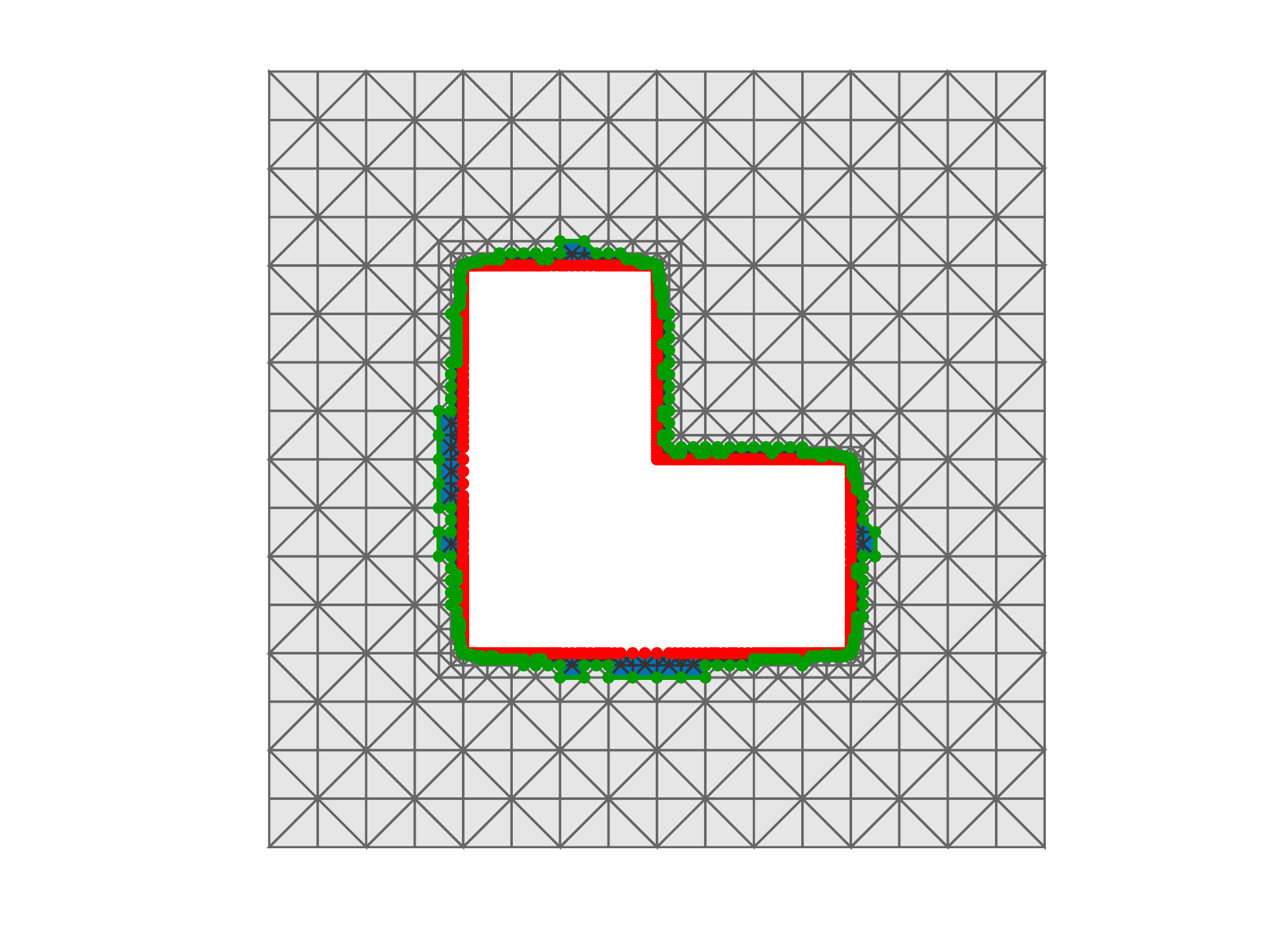}
$\#\TT_{h}^S = 2891$, $\#\FF_{h}^\Gamma = 1024$
\end{subfigure}
\caption{Adaptively generated meshes ($\theta=0.6$) in Example~\ref{example4}. We indicate the boundary layer {\bf\color{tuwBlue}$\boldsymbol{S}$~(blue)}, the boundary {\bf\cred$\boldsymbol{\Gamma}$~(red)}, and the interior boundary {\bf\color{green!50!black}$\boldsymbol{\Gamma^c = \partial S \setminus \Gamma}$~\bf(green)};  see Figure~\ref{fig:example1:mesh} for the color code.}
\label{fig:example4:mesh}
\end{figure}
\begin{figure}[t]
\centering
\includegraphics[trim={.7cm 0 0 0},clip,width=.49\textwidth]{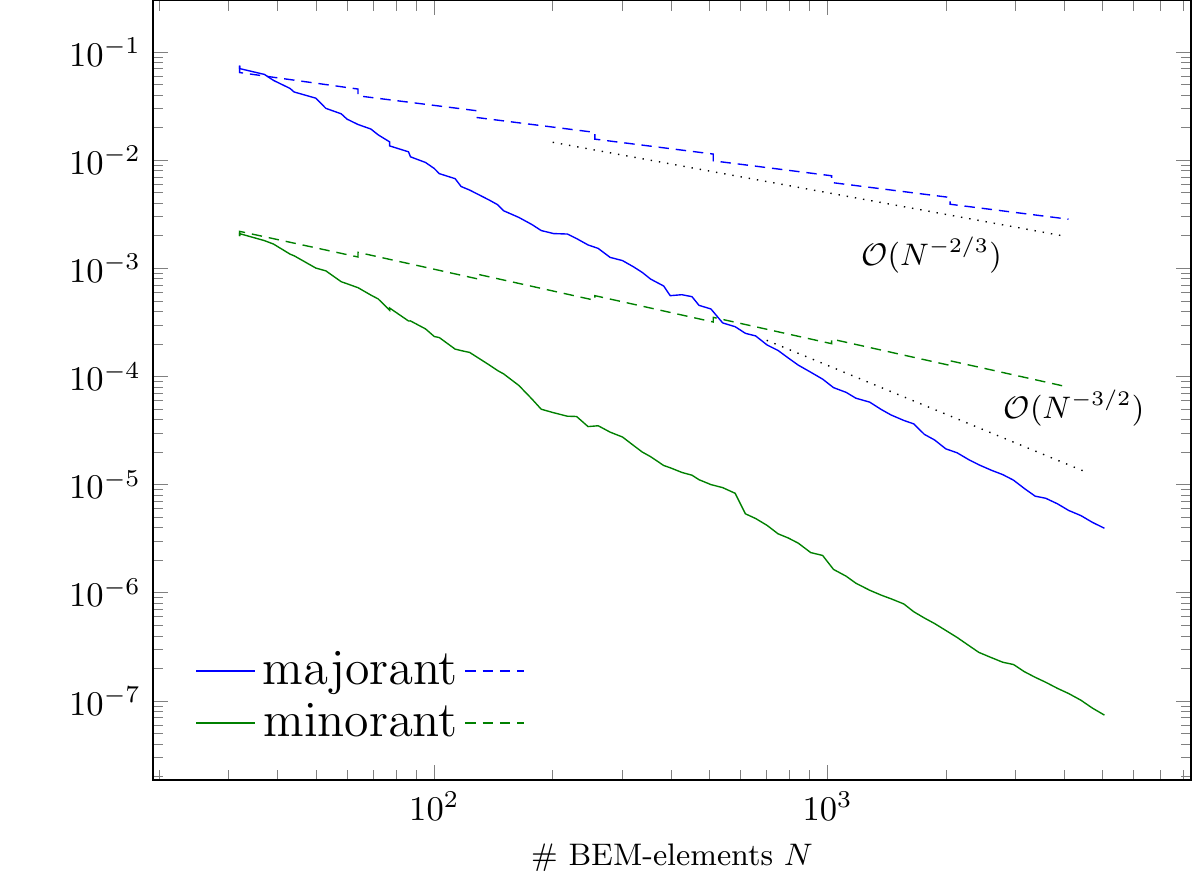}
\includegraphics[trim={.7cm 0 0 0},clip,width=.49\textwidth]{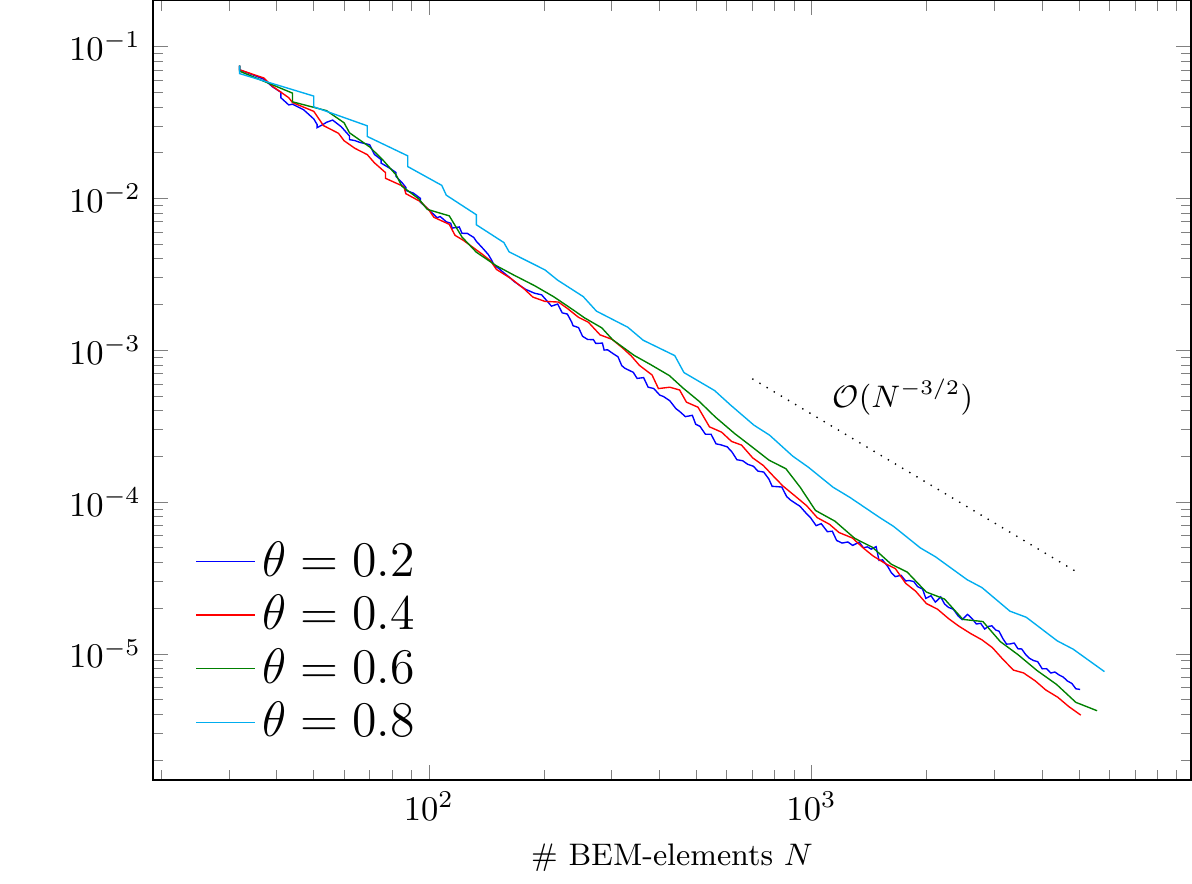}
\caption{Comparison of adaptive vs.\ uniform mesh-refinement in Example~\ref{example4}, The majorant is computed by $\PP^1$-FEM. \emph{Left:} Since the potential error $\norm{\nabla(u-u_h)}{\Lt(\Omega)}$ is unknown and ${\rm osc}_h = 0$, we only compare the majorant $\norm{\nabla w_h}{\Lt(\Omega)}$ from~\eqref{maj:discrete} and the minorant $\mmin(\ttau_h)^{1/2}$ from~\eqref{eq:experiments:minorant} for uniform (dashed) and adaptive mesh-refinement (solid) with $\theta = 0.4$. 
\emph{Right:} We compare the majorant for adaptive mesh-refinement for various choices of $\theta$.}
\label{fig:example4}
\end{figure}
%
%%%%%%%%%%%%%%%%%%%%%%%%%%%%%%%%%%%%%%%%%%%%%%%%%%%%%%%%%%%%%%%%%%%%%%%%%%%%%%%%%%%
\subsection{\revision{Exterior} domains}
%%%%%%%%%%%%%%%%%%%%%%%%%%%%%%%%%%%%%%%%%%%%%%%%%%%%%%%%%%%%%%%%%%%%%%%%%%%%%%%%%%%%
One particular strength of BEM is that it naturally allows to consider also 
exterior domain problems formulated on unbounded Lipschitz domains $\Omega^c := \R^d \backslash \overline\Omega$. In this case, the homogeneous Dirichlet--Laplace problem subject to given inhomogeneous boundary data $g$ reads
\begin{subequations}
\label{eq:exterior}
\begin{align}\label{eq:exterior:strongform}
 \Delta u = 0 \quad \text{in } \Omega^c, 
 \qquad
 u = g \quad \text{on } \Gamma,
\end{align}
supplemented by the radiation (decay) condition (for $|x| \to \infty$)
\begin{align}\label{eq:exterior:radiation}
 u(x) = \OO(\log|x|) \quad \text{for } d = 2
 \qquad \text{resp.} \qquad 
 u(x) = \OO(1/|x|) \quad \text{for } d = 3.
\end{align}
\end{subequations}
We note that the latter radiation condition is naturally incorporated 
into the potential operators (due to the choice of the fundamental solution
with right decay) arising in BEM, e.g., 
any single-layer potential $\widetilde V\phi_h$ satisfies~\eqref{eq:exterior:radiation}.

We note that the functional error identities from Theorem~\ref{prop:continuousmajmin} 
(with $\Omega$ being replaced by the exterior domain $\Omega^c$) 
remain valid (in principal) for any 
\revision{
\begin{align*}
v \in \Lt_{\sf loc}(\overline{\Omega^c})
:=\set{\varphi}{\str{\varphi}{\Xi\,\cap\,\Omega^{c}}\in\Lt(\Xi\,\cap\,\Omega^{c})
\text{ for all compact }\Xi\subset\R^{d}}
\end{align*}
with $\nabla v \in \Lt(\Omega^c)$ and $\Delta v = 0$. 
More precisely, a proper solution theory for \eqref{eq:exterior}
is available in the weighted Sobolev space $\H^{1}_{-1}(\Omega^c)$ defined by, e.g., for $d=3$,
\begin{align*}
\H^{1}_{-1}(\Omega^c)
:=\set{\varphi}{\varphi(\cdot)/|\cdot| \in\Lt(\Omega^c),\;\nabla\varphi\in\Lt(\Omega^c)};
\end{align*}
see, e.g., \cite{PR2009a,zbMATH06245058},
where also functional {\sl a~posteriori} error estimates 
for corresponding exterior domain problems for the Poisson equation $\Delta u=f$ have been proved.
}%
Consequently, the computable upper and lower bounds of Theorem~\ref{cor:continuousupperbound} 
(resp.\ Corollary~\ref{cor:continuouslowerboundtwo})
hold (with appropriate modifications)
for any approximation $\phi_h \approx \phi$ and $v := \widetilde V \phi_h$. In particular, Algorithm~\ref{algorithm} can also be applied to BEM for exterior domain problems.

\example{Direct BEM for exterior problem}
\label{example4}
To illustrate the latter observation, we consider the exterior domain  
\begin{align*}
\Omega^c := \R^2 \backslash \overline\Omega,\qquad
 \Omega = (0, 1/2)^2 \backslash \big([(1/4,1/2]\times[0,1/4] \big),
\end{align*}
where $\Omega$ is the L-shaped domain from Section~\ref{example2}.
We consider~\eqref{eq:exterior} with constant Dirichlet data
\begin{align}
 g = 1 = (1/2 - K)1 \quad \text{on } \Gamma,
\end{align}
where $K$ is the double-layer integral operator. Consequently, the corresponding indirect BEM formulation~\eqref{eq:weakform} turns out to be a direct BEM formulation for the exterior domain problem~\cite{mclean,sauter-schwab}, where all data oscillation terms vanish. Thus, one can expect that the sought density $\phi \in \H^{-1/2}(\Gamma)$ has singularities at the convex corners of $\Omega$ (but \emph{not} at the reentrant corner). 

We employ Algorithm~\ref{algorithm} (with Galerkin BEM). The initial mesh $\TT_h$ with $\#\TT_h = 416$ right triangles is a triangulation of $(-1/4,3/4)^2 \backslash \overline\Omega \subset \Omega^{c}$; see Figure~\ref{fig:example4:mesh}. Some numerical results are shown in Figure~\ref{fig:example4}. Since the exact potential $u$ is unknown, we cannot compute the potential error $\norm{\nabla(u-u_h)}{\Lt(\Omega)}$. However, adaptive mesh-refinement leads to the optimal convergence behavior of majorant and minorant (and hence also of $\norm{\nabla(u-u_h)}{\Lt(\Omega)}$).

%%%%%%%%%%%%%%%%%%%%%%%%%%%%%%%%%%%%%%%%%%%%%%%%%%%%%%%%%%%%%%%%%%%%%%%%%%%%%%%%%%%
%%%%%%%%%%%%%%%%%%%%%%%%%%%%%%%%%%%%%%%%%%%%%%%%%%%%%%%%%%%%%%%%%%%%%%%%%%%%%%%%%%%
\section{Conclusion}
\label{section:conclusion}
%%%%%%%%%%%%%%%%%%%%%%%%%%%%%%%%%%%%%%%%%%%%%%%%%%%%%%%%%%%%%%%%%%%%%%%%%%%%%%%%%%%
%%%%%%%%%%%%%%%%%%%%%%%%%%%%%%%%%%%%%%%%%%%%%%%%%%%%%%%%%%%%%%%%%%%%%%%%%%%%%%%%%%%

We have presented, for the first time, functional error estimates for BEM. Not only that the presented estimates are independent of the specific discretization method (i.e., Galerkin or collocation), they also provide guaranteed upper and lower bounds for the unknown energy error. This is in contrast to existing techniques, which usually contain generic constants. The error bounds are obtained by solving auxiliary variational problems by \revision{FEM} on a \revision{boundary layer $S \subset \Omega$}.
\revision{One possible disadvantage of our approach is that it needs a volume mesh for $S$ to solve the auxiliary FEM problems. However, this appears to be a standard problem for FEM mesh generation.}
%\note{{\bf Editor:} Gebt bitte in der Introduction + Conclusion an, dass man ein Volumengitter konstruieren muss.}

In the paper, we consider the Dirichlet problem of the Laplace equation, but the approach is expected to generalize to other boundary value problems. In the considered case, the upper error bound is based on the Dirichlet principle, while the lower error bound is based either on a variational problem in terms of a potential (scalar stream function in 2D and vector potential in 3D) or a mixed problem (in 2D and 3D). The upper bound is localized and drives an adaptive refinement of the boundary mesh. Since 
%the finite element strip 
\revision{$S$} contains always two layers of elements, it geometrically shrinks towards the boundary during refinement.  This way, the ratio between \revision{the FEM DoF} for obtaining the error estimates and the BEM DoF remains bounded. We have examined various 2D test problems on square and L-shaped domains, with and without singular potential, including exterior problems. The proposed adaptive algorithm exhibited excellent performance. In all cases, the optimal convergence rates could be achieved.

\revision{Ongoing work concerns the
%We already work on a 
further analysis of the oscillations of $g-\str{u_{h}}{\Gamma}$ 
and the implementation of higher-order $\Lt$-projections, 
which may overcome the lack of accuracy of the majorant and minorant 
observed in our numerical experiments for very coarse BEM meshes.}
\revision{An implementation of the proposed algorithm in 3D 
and the extension to electromagnetic problems 
is 
%(except of numerical challenges) straight forward and 
also the subject of future research.}

%%%%%%%%%%%%%%%%%%%%%%%%%%%%%%%%%%%%%%%%%%%%%%%%%%%%%%%%%%%%%%%%%%%%%%%%%%%%%%%%%%%
%%%%%%%%%%%%%%%%%%%%%%%%%%%%%%%%%%%%%%%%%%%%%%%%%%%%%%%%%%%%%%%%%%%%%%%%%%%%%%%%%%%

%\bibliographystyle{alpha}
\bibliographystyle{plain}
\bibliography{literature}

%%%%%%%%%%%%%%%%%%%%%%%%%%%%%%%%%%%%%%%%%%%%%%%%%%%%%%%%%%%%%%%%%%%%%%%%%%%%%%%%%%%
%%%%%%%%%%%%%%%%%%%%%%%%%%%%%%%%%%%%%%%%%%%%%%%%%%%%%%%%%%%%%%%%%%%%%%%%%%%%%%%%%%%

%\begin{comment}
%%%%%% APPENDIX %%%%%%%

%\appendix
\section*{Appendix: Some remarks on the analysis}
\label{section:appendix}

We recall the notations introduced in Section~\ref{section:domainsfspaces}, 
in particular, the harmonic extension operator \ $\widehat{(\cdot)} : \H^{1/2}(\Gamma) \to \Ho(\Omega)$
from~\eqref{eq:lifting}. Moreover, we add the definitions of
\begin{align*}
\Ho_{\Gamma^{c}}(S)
&:=\overline{\set{\varphi|_{S}}{\varphi\in\mathsf{C}^{\infty}(\R^{d}),\,
\supp\varphi\text{ compact},\,\dist(\supp\varphi,\Gamma^{c})>0}}^{\,\Ho(S)}\\
&\;=\set{\varphi\in\Ho(S)}{\str{\varphi}{\Gamma^{c}}=0},
\intertext{see, e.g., \cite{bauerpaulyschomburg2016}, for the density result, and}
{\H_{\Gamma^{c}}(\div,S)} 
&:=\overline{\set{\ssigma|_{S}}{\ssigma\in\mathsf{C}^{\infty}(\R^{d}),\,
\supp\ssigma\text{ compact},\,\dist(\supp\ssigma,\Gamma^{c})>0}}^{\,\H(\div,S)}.
\end{align*}
The latter space generalizes the (partial) homogeneous boundary condition 
$\ntr{\ssigma}{\Gamma^c} = 0$ to functions $\ssigma \in \H(\div,S)$.
Note that the natural trace $\ntr{\ssigma}{\partial S}$ 
can be 'restricted' to, e.g., $\Gamma$ in the following sense.

\begin{remark} 
\label{Hgraddivgammarem}
Functions in $\Ho_{\Gamma^{c}}(S)$ vanish at $\Gamma^{c}$.
In particular, any function $\varphi\in\Ho_{\Gamma^{c}}(S)$ can be extended 
by zero to a function $\varphi\in\Hoom$.
Analogously, vector fields in ${\H_{\Gamma^{c}}(\div,S)}$ 
have vanishing normal component at $\Gamma^{c}$ in a weak sense.
In particular, any vector field $\ssigma\in{\H_{\Gamma^{c}}(\div,S)}$ can be extended 
by zero to a vector field $\ssigma\in\Hdivom$. Hence, 
the normal trace $\ntr{\ssigma}{\partial S}\in\H^{-1/2}(\partial S)$ of $\ssigma$
may be identified with a well defined element $\ntr{\ssigma}{\Gamma}\in\H^{-1/2}(\Gamma)$
vanishing on $\Gamma^{c}$ in a weak sense.
\end{remark}

Let us discuss the minimiser $\overline w=u-v$ of the upper bound
and the maximiser $\underline{\ttau}=\grad\overline w$ of the lower bound 
from Theorem \ref{prop:continuousmajmin} in more detail. 
Note that\footnote{As the exterior derivative commutes 
with the trace operator, which is simply the pull-back
of the canonical embedding of the boundary manifold $\Gamma$ into $\overline{\Omega}$
(i.e., $\iota^{*}\text{d}=\text{d}\iota^{*}$), we see for the special case
of $\varphi\in\Hoom$ that 
$\ttr{\grad\varphi}{\Gamma}
=\grad_{\Gamma}\str{\varphi}{\Gamma}$
in $\H^{-1/2}(\Gamma)$, where $\ttr{(\cdot)}{\Gamma}:\H(\curl,\Omega)\to\H^{-1/2}(\Gamma)$ 
denotes the tangential trace and $\grad_{\Gamma}:\H^{1/2}(\Gamma)\to\H^{-1/2}(\Gamma)$ the surface gradient.}
\begin{align}
\nonumber
\overline w
&=u-v\in\Hoom,
&
\underline{\ttau}
&=\grad\overline w\in\Hdivom\cap\Hcurlom,\\
\label{eqtauwone}
\Delta\overline w&=0
\text{ in }\Omega,
&
\div\underline{\ttau}&=0
\text{ in }\Omega,\\
\nonumber
&&
\curl\underline{\ttau}&=0
\text{ in }\Omega,\\
\nonumber
\str{\overline w}{\Gamma}&=g-\str{v}{\Gamma}
\text{ on }\Gamma,
&
\ttr{\underline{\ttau}}{\Gamma}&=\grad_{\Gamma}(g-\str{v}{\Gamma})
\text{ in }\H^{-1/2}(\Gamma).
\end{align}
Moreover, by replacing $w$ with $\overline w+\eps\varphi$
and by replacing $\ttau$ with $\underline{\ttau}+\eps\ssigma$ 
in~\eqref{eq:prop:continuousmajmin}, where
$\varphi\in\Hozom$ and $\ssigma\in\Hdivom$ with $\div\ssigma=0$ 
as well as $\eps\in\R$,
we obtain the variational formulations
\begin{subequations}
\begin{align}
\label{wvarform}
\forall\,\varphi&\in\Hozom
&
\scp{\grad\overline w}{\grad\varphi}{\Lt(\Omega)}
&=0,\\
\label{tauvarform}
\forall\,\ssigma&\in\Hdivom\;\text{with}\;\div\ssigma=0
&
\scp{\underline{\ttau}}{\ssigma}{\Lt(\Omega)} 
&=\dualpga{g-\str{v}{\Gamma}}{\ntr{\ssigma}{\Gamma}}.
\end{align}
\end{subequations}
Let $\psi\in\H^{1/2}(\Gamma)$ and let $\widehat\psi\in\Hoom$ be its harmonic extension. 
Testing the second variational formulation~\eqref{tauvarform}
with $\ssigma=\grad\widehat\psi$ shows
\begin{align*}
\dualpga{\psi}{\ntr{\underline{\ttau}}{\Gamma}}
=\dualpga{g-\str{v}{\Gamma}}{\ntr{\grad\widehat\psi}{\Gamma}}.
\end{align*}
Thus, additionally to the scalar and tangential boundary conditions 
for $\overline w$ and $\underline{\ttau}$ in~\eqref{eqtauwone}, respectively,
we have also found a normal boundary condition for $\underline{\ttau}=\grad\overline w$, namely
\begin{align}
\label{nbctau}
\ntr{\underline{\ttau}}{\Gamma}
=\bdualpga{g-\str{v}{\Gamma}}{\ntr{\grad\widehat{(\,\cdot\,)}}{\Gamma}}
\quad\text{in}\quad\H^{-1/2}(\Gamma).
\end{align}
This shows that there are different options for computing $\overline w$ and $\underline{\ttau}$.

\begin{remark}
Note that 
\begin{align*}
\str{\partial_{n}\widehat{(\,\cdot\,)}}{\Gamma}
=\ntr{\grad\widehat{(\,\cdot\,)}}{\Gamma}
:\H^{1/2}(\Gamma)
\to\H^{-1/2}(\Gamma)
\end{align*}
is the well known Dirichlet-to-Neumann operator for the homogeneous Laplacian.
Moreover, the normal trace of $\underline{\ttau}$ in~\eqref{nbctau}
does not depend on the harmonic extension
as 
\begin{align*}
\dualpga{g-\str{v}{\Gamma}}{\ntr{\grad\widehat\psi}{\Gamma}}
=\dualpga{\psi}{\ntr{\underline{\ttau}}{\Gamma}}
=\bdualpga{\psi}{\ntr{\grad(u-v)}{\Gamma}}
\quad \text{for all $\psi\in\H^{1/2}(\Gamma)$}
.
\end{align*}
\end{remark}

Recalling Theorem \ref{prop:continuousmajmin}
and Theorem~\ref{cor:continuousupperbound}, we note the following.

\begin{remark}[Minimiser of the upper bound]
\label{prop:continuousmajmin:remub}
The unique minimiser $\overline w$ of the upper bound
is the unique harmonic extension of
$g-\str{v}{\Gamma}$ to $\Omega$, i.e., $\overline w\in\Hoom$ 
is the unique solution of the Dirichlet--Laplace problem
\begin{align*}
\Delta\overline w=0\text{ in }\Omega,\qquad
\str{\overline w}{\Gamma}=g-\str{v}{\Gamma}\text{ on }\Gamma.
\end{align*}
It holds 
$\scp{\grad\overline w}{\grad\varphi}{\Lt(\Omega)}=0$ for all $\varphi\in\Hozom$.
Moreover, $\overline w\in\Hoom$ 
solves the Neumann Laplace problem
\begin{align*}
\Delta\overline w=0\text{ in }\Omega,\qquad
\ntr{\grad\overline w}{\Gamma}
=\bdualpga{g-\str{v}{\Gamma}}{\ntr{\grad\widehat{(\,\cdot\,)}}{\Gamma}}
\quad\text{in}\quad\H^{-1/2}(\Gamma).
\end{align*}
It holds 
$\scp{\grad\overline w}{\grad\varphi}{\Lt(\Omega)}=
\bdualpga{g-\str{v}{\Gamma}}{\ntr{\grad\widehat{(\str{\varphi}{\Gamma})}}{\Gamma}}$ 
for all $\varphi\in\Hoom$.
Note that, at least analytically, both formulations can also be used to find
the unique maximiser $\underline{\ttau}=\grad\overline w$ of the upper bound.
For numerical purposes the Dirichlet--Laplace problem
is the better and easier choice to compute $\overline w$.
\end{remark}

Next we want to find equations and variational formulations for $\underline{\ttau}$
not involving $\overline w$.
For this, let us introduce Dirichlet and Neumann fields
\begin{align*}
\mathcal{H}_{D}(\Omega)
&:=\set{\ssigma\in\H_{0}(\curl,\Omega)\cap\H(\div,\Omega)}{\curl\ssigma=0,\,\div\ssigma=0},\\
\mathcal{H}_{N}(\Omega)
&:=\set{\ssigma\in\Hcurlom\cap\H_{0}(\div,\Omega)}{\curl\ssigma=0,\,\div\ssigma=0},
\end{align*}
where
\begin{align*}
\H_{0}(\curl,\Omega)
&:=\overline{\set{\ssigma\in\mathsf{C}^{\infty}(\Omega)}
{\supp\ssigma\text{ compact in }\Omega}}^{\, \Hcurlom},\\
\H_{0}(\div,\Omega)
&:=\overline{\set{\ssigma\in\mathsf{C}^{\infty}(\Omega)}
{\supp\ssigma\text{ compact in }\Omega}}^{\, \Hdivom}.
\end{align*}
We compute
\begin{align*}
\forall\,\ssigma&\in\mathcal{H}_{D}(\Omega)
&
\scp{\underline{\ttau}}{\ssigma}{\Lt(\Omega)}
&=\dualpga{g-\str{v}{\Gamma}}{\ntr{\ssigma}{\Gamma}},\\
\forall\,\ssigma&\in\mathcal{H}_{N}(\Omega)
&
\scp{\underline{\ttau}}{\ssigma}{\Lt(\Omega)}
&=0.
\end{align*}

\begin{remark}[Maximiser of the lower bound]
\label{prop:continuousmajmin:remlb}
The unique maximiser $\underline{\ttau}=\grad\overline w$ 
of the lower bound is the unique solution of the electro-static Maxwell problem
\begin{align*}
\curl\underline{\ttau}
&=0
&&\text{in }\Omega,\\
\div\underline{\ttau}
&=0
&&\text{in }\Omega,\\
\ttr{\underline{\ttau}}{\Gamma}
&=\grad_{\Gamma}(g-\str{v}{\Gamma})
&&\text{in }\H^{-1/2}(\Gamma),\\
\scp{\underline{\ttau}}{\ssigma}{\Lt(\Omega)}
&=\dualpga{g-\str{v}{\Gamma}}{\ntr{\ssigma}{\Gamma}}
&&\text{for all }\ssigma\in\mathcal{H}_{D}(\Omega),
\intertext{as well as the unique solution of the magneto-static Maxwell problem}
\curl\underline{\ttau}
&=0
&&\text{in }\Omega,\\
\div\underline{\ttau}
&=0
&&\text{in }\Omega,\\
\ntr{\underline{\ttau}}{\Gamma}
&=\bdualpga{g-\str{v}{\Gamma}}{\ntr{\grad\widehat{(\,\cdot\,)}}{\Gamma}}
&&\text{in }\H^{-1/2}(\Gamma),\\
\scp{\underline{\ttau}}{\ssigma}{\Lt(\Omega)}
&=0
&&\text{for all }\ssigma\in\mathcal{H}_{N}(\Omega).
\end{align*}
See \cite{picard1981,picard1982,pauly2019a,bauerpaulyschomburg2016}
for proper solution theories.
\end{remark}

As $\underline{\ttau}\in\Hdivom$ with $\div\underline{\ttau}=0$,
by~\eqref{tauvarform} the vector field $\underline{\ttau}$ solves for all $\omega\in\Lt(\Omega)$
the mixed problem
\begin{align*}
\scp{\underline{\ttau}}{\ssigma}{\Lt(\Omega)} 
+\scp{\div\ssigma}{\omega}{\Lt(\Omega)} 
&=\dualpga{g-\str{v}{\Gamma}}{\ntr{\ssigma}{\Gamma}},\\
\scp{\div\underline{\ttau}}{\psi}{\Lt(\Omega)}
&=0
\end{align*}
for all $(\ssigma,\psi)\in\Hdivom\times\Lt(\Omega)$ with $\div\ssigma=0$.
On the other hand, especially for numerical reasons, 
we want to skip the solenoidal conditions,
which leads to the following mixed variational saddle point formulation
(cf. Theorem~\ref{cor:continuousupperbound}).

\begin{lemma}[Mixed problem for the lower bound]
\label{mixedprobtau}
Let $v\in\Hoom$ with $\Delta v=0$. 
Then the mixed problem
\begin{align*}
\scp{\ttau}{\ssigma}{\Lt(\Omega)} 
+\scp{\div\ssigma}{\omega}{\Lt(\Omega)} 
&=\dualpga{g-\str{v}{\Gamma}}{\ntr{\ssigma}{\Gamma}},\\
\scp{\div\ttau}{\psi}{\Lt(\Omega)}
&=0
\end{align*}
for all $(\ssigma,\psi)\in\Hdivom\times\Lt(\Omega)$,
admits a unique solution $(\ttau,\omega)\in\Hdivom\times\Lt(\Omega)$.
Moreover, $(\ttau,\omega)=(\underline{\ttau},\overline w)$, i.e.,
the latter mixed formulation 
can be used to compute the unique maximiser $\underline{\ttau}=\grad\overline w$
and the unique minimiser $\overline w$ simultaneously.
\end{lemma}

\begin{remark}
\label{mixedprobtaurem}
The mixed formulation in Lemma~\ref{mixedprobtau} 
is the mixed formulation of the Dirichlet--Laplace problem 
from Remark~\ref{prop:continuousmajmin:remub}.
For numerical purposes the latter mixed formulation
is only a good choice to compute $\underline{\ttau}$,
since the numerical approximations will only satisfy
$(\ttau,\omega)\in\Hdivom\times\Lt(\Omega)$ 
but in general not $\omega\in\Hoom$.
\end{remark}

\begin{proof}[Proof of Lemma~\ref{mixedprobtau}.]
The inner product $\scp{\,\cdot\,}{\,\cdot\,}{\Lt(\Omega)}$ is positive
on the kernel $\set{\ssigma\in\Hdivom}{\div\ssigma=0}$
and the inf-sup-condition is satisfied as for $\psi\in\Lt(\Omega)$
\begin{align*}
\sup_{\ssigma\in\Hdivom}
\frac{\scp{\div\ssigma}{\psi}{\Lt(\Omega)}}
{\norm{\ssigma}{\Hdivom}\norm{\psi}{\Lt(\Omega)}}
\geq\frac{\norm{\psi}{\Lt(\Omega)}}
{\sqrt{\norm{\ssigma_{\psi}}{\Lt(\Omega)}^2+\norm{\psi}{\Lt(\Omega)}^2}}
\geq\frac{1}{\sqrt{c_{F}^2+1}}.
\end{align*}
This follows by solving a Dirichlet--Laplace problem, i.e.,
by finding the unique solution $\omega_{\psi}\in\Hozom$ of
\begin{align*}
\Delta \omega_{\psi}=\psi\text{ in }\Omega,\qquad
\str{\omega_{\psi}}{\Gamma}
=0\text{ on }\Gamma,
\end{align*}
and setting $\ssigma_{\psi}=\grad\omega_{\psi}$.
Note that the estimate 
$\norm{\ssigma_{\psi}}{\Lt(\Omega)}\leq c_{F}\norm{\psi}{\Lt(\Omega)}$
holds true, where $0<c_{F}\leq\diam(\Omega)/\pi$ is the Friedrichs constant
for the gradient operator on $\Hozom$.
Therefore, the standard saddle point theory for mixed problems
shows the unique solvability, see, e.g., \cite{boffibrezzifortin}.
Moreover, we have $\div\ttau=0$ by the second line of the mixed formulation.
Testing the first line with compactly supported test vector fields $\ssigma$
shows that $\omega\in\Hoom$ with $\grad\omega=\ttau$.
Hence $\Delta\omega=\div\ttau=0$. Furthermore, the first line implies 
\begin{align*}
\dualpga{\str{\omega}{\Gamma}}{\ntr{\ssigma}{\Gamma}}
=\dualpga{g-\str{v}{\Gamma}}{\ntr{\ssigma}{\Gamma}}
\end{align*}
for all $\ssigma\in\Hdivom$, yielding $\str{\omega}{\Gamma}=g-\str{v}{\Gamma}$
by the surjectivity of the normal trace operator $\ntr{(\cdot)}{\Gamma}$.
Thus $\omega=u-v=\overline w$ and $\ttau=\grad\omega=\grad\overline w=\underline{\ttau}$.
\end{proof}

\begin{remark}
\label{cor:continuouslowerboundtwo:rem}
The continuous version of \eqref{eq:lower-bound2ddiscrete}
in Corollary~\ref{cor:continuouslowerboundtwo} reads:
Find $w\in\Ho_{\Gamma^{c}}(S)$ such that
\begin{equation}
\label{eq:lower-bound2d}
\scp{\grad w}{\grad\varphi}{\Lt(S)}
=\dualpga{g-\str{v}{\Gamma}}{\ntr{\vec\curl\,\varphi}{\Gamma}}
\qquad\text{for all }\varphi\in\Ho_{\Gamma^c}(S).
\end{equation}
Then $w\in\Ho(S)$ is the unique solution of the mixed Dirichlet--Neumann--Laplace problem 
\begin{align*}
\Delta w=0\text{ in }S,\quad
\str{w}{\Gamma^{c}}=0\text{ on }\Gamma^{c},\quad
\ntr{\grad w}{\Gamma}
=\bdualpga{g-\str{v}{\Gamma}}{\ntr{\vec\curl\,\widehat{(\,\cdot\,)_{\partial S}}}{\Gamma}}
\text{ in }\H^{-1/2}(\Gamma).
\end{align*}
To see this, we pick different test functions.
Testing \eqref{eq:lower-bound2d} with compactly supported (in $S$) smooth functions
shows $\Delta w=0$ in $S$, and by definition, i.e.,
$w\in\Ho_{\Gamma^{c}}(S)$, it is clear that $\str{w}{\Gamma^{c}}=0$ on $\Gamma^{c}$.
Let $\phi\in\H^{1/2}(\Gamma)$, define
\begin{align*}
\phi_{\partial S}
:=\begin{cases}
\phi&\text{on }\Gamma,\\
0&\text{on }\Gamma^{c},
\end{cases}
\end{align*}
and let $\varphi:=\widehat{\phi_{\partial S}}\in\Ho(S)$ 
be the unique harmonic extension to $S$ of $\phi_{\partial S}$,
compare to~\eqref{eq:lifting}. 
Then $\varphi\in\Ho_{\Gamma^{c}}(S)$ and
testing~\eqref{eq:lower-bound2d} with $\varphi$ yields
\begin{align*}
\dualpga{g-\str{v}{\Gamma}}{\ntr{\vec\curl\,\widehat{\phi_{\partial S}}}{\Gamma}}
=\dualpga{\phi}{\ntr{\grad w}{\Gamma}},
\end{align*}
i.e., $\ntr{\grad w}{\Gamma}=\bdualpga{g-\str{v}{\Gamma}}{\ntr{\vec\curl\,\widehat{(\,\cdot\,)_{\partial S}}}{\Gamma}}$
in $\H^{-1/2}(\Gamma)$.

By construction, $\widetilde w$, the extension by zero to $\Omega$, belongs to $\Hoom$ 
and hence we have $\ttau:=\vec\curl \, \widetilde w \in\Hdivom$ with $\div\ttau=0$.
Theorem~\ref{prop:continuousmajmin} shows
\begin{align*}
2\dualpga{g-\str{v}{\Gamma}}{\ntr{\vec\curl \, \widetilde w}{\Gamma}}
-\norm{\grad w}{\Lt(S)}^2
\leq\bnorm{\grad(u-v)}{\Lt(\Omega)}^2.
\end{align*}
\end{remark}
%\end{comment}

%%%%%%%%%%%%%%%%%%%%%%%%%%%%%%%%%%%%%%%%%%%%%%%%%%%%%%%%%%%%%%%%%%%%%%%%%%%%%%%%%%%
%%%%%%%%%%%%%%%%%%%%%%%%%%%%%%%%%%%%%%%%%%%%%%%%%%%%%%%%%%%%%%%%%%%%%%%%%%%%%%%%%%%
%%%%%%%%%%%%%%%%%%%%%%%%%%%%%%%%%%%%%%%%%%%%%%%%%%%%%%%%%%%%%%%%%%%%%%%%%%%%%%%%%%%
%%%%%%%%%%%%%%%%%%%%%%%%%%%%%%%%%%%%%%%%%%%%%%%%%%%%%%%%%%%%%%%%%%%%%%%%%%%%%%%%%%%

\end{document}